\documentclass[aip,graphicx]{revtex4-1}

\usepackage{amsmath,amssymb,amsthm,amsfonts,mathrsfs}
\usepackage{caption}
\usepackage{subcaption}
\usepackage{xcolor}
\usepackage{dsfont}
\usepackage{stmaryrd}
\usepackage{pifont}
\usepackage{textcomp}
\usepackage{tabularx}
\usepackage{multirow}

\usepackage{upgreek}
\usepackage{natbib}
\usepackage{graphicx}
\usepackage{float}
\usepackage{anyfontsize} 
\usepackage{url}
\usepackage{dcolumn}% Align table columns on decimal point
\usepackage{bm}% bold math
\usepackage[utf8]{inputenc}
\usepackage[T1]{fontenc}
\usepackage{mathptmx}
\usepackage{etoolbox}

\newtheorem{The}{Theorem}[section]
\newtheorem{Note}{Note}[section]

\newtheorem{Def}{Definition}[section]

\newtheorem{Ex}{Example}[section]

\draft % marks overfull lines with a black rule on the right

\begin{document}
	
	%\preprint{}
	
	\title{Stability and Bifurcation Analysis of Two-Term Fractional Differential Equation with Delay} %Title of paper
	
	\author{Sachin Bhalekar}
	\email[corresponding author: ]{sachinbhalekar@uohyd.ac.in}
	%\thanks{}
	\author{Deepa Gupta}
	\email[]{deepanew96@gmail.com}
	\affiliation{School of Mathematics and Statistics, University of Hyderabad, Hyderabad, Telangana, India- 500046}

	\begin{abstract}
		This manuscript deals with the stability and bifurcation analysis of the equation $D^{2\alpha}x(t)+c D^{\alpha}x(t)=a x(t)+b x(t-\tau)$, where $0<\alpha<1$ and $\tau>0$. We sketch the boundaries of various stability regions in the parameter plane under different conditions on $\alpha$ and $b$. First, we provide the stability analysis of this equation with $\tau=0$. Change in the stability of the delayed counterpart is possible only when the characteristic roots cross the imaginary axis. This leads to various delay-independent as well as delay-dependent stability results. The stability regions are bifurcated on the basis of the following behaviors with respect to the delay $\tau$ viz. stable region for all $\tau>0$, unstable region, single stable region, stability switch, and instability switch.
	\end{abstract}
	
	% insert suggested PACS numbers in braces on the next line
	
	\maketitle %\maketitle must follow title, authors, abstract and \pacs
	
	\begin{quotation}
		\textbf{Fractional derivatives and the delay are the popular tools used by applied scientists to model the accurate behaviors of natural systems. As the fractional derivative is a nonlocal operator, one has to provide the complete history of the state during its computation. On the other hand, the delay uses the short history of the state. These equations have proved useful in modeling the systems in applied sciences and engineering. Though these equations look very similar to the classical ordinary differential equations, the analysis is, however, not simple. The characteristic equations of these equations are of a transcendental nature and admit infinitely many roots, in general. Therefore, the simple conditions in terms of the parameters, as provided in this manuscript, are very useful to the researchers. The complex dynamics of these systems include not only stable and unstable behaviors but the delay-dependent features such as a single stable region provided by the delay interval and stability/instability switches.}
	\end{quotation}

\section{Introduction}
	\noindent 
Differential equations (DE) can be treated as the heart of mathematical analysis. These equations emerge as models in natural systems. As the derivative represents the rate of change of a quantity, the DE models are widely used by Scientists, Engineers, Economists, and so on \cite{simmons2016differential,khalil2002control,strogatz2018nonlinear}. 
The second-order ordinary differential equations are found in classic examples such as Newton's law of gravitation, simple harmonic motion (damped undamped and forced oscillations), Kepler's laws of planetary motion, and LRC circuit in electronics \cite{simmons2016differential}. B van der Pol \cite{van1926lxxxviii} used these equations to model the reaction oscillators.
% some recent work on 2nd order ode

% DDE
Though researchers widely use the ODE models, they are unsuitable for modeling the memory properties in natural systems. Volterra \cite{volterra} proposed the integral equations containing the nonlocal integral operator to model the biological systems. Furthermore, the delay in the model also has the ability to include the history of the state in the model \cite{lakshmanan2011dynamics,hale2012functional}. If the rate of change of the present state depends on the state at past times then the resulting equation is called the delay differential equation (DDE). Smith \cite{smith2011introduction} discussed the applications of the DDEs in life science. The stability and oscillation theory of the DDEs arising in the population dynamics is developed by Gopalsamy \cite{gopalsamy2013stability}. Various applications of these equations in fluid dynamics, economics, mechanical engineering, life science, chemistry, and physics are presented by Erneux in the book \cite{erneux2009applied}.
% FDE
\par The constant delay in the equation is adequate for modeling the ``short-time memory". However, we need nonlocal operators such as ``fractional order derivative" to include the ``long-term memory" in the model. The order of the derivative in the ``fractional order differential equation (FDE)" is ``non-integer". As one can expect, the generalization to fractional derivative (FD) can be done in many ways \cite{podlubny1998fractional,kilbas2006theory}. However, we stick to the definition provided by Caputo because it is more suitable to real-life problems.

% FDDE
\par Thus, there are improvements in the model involving both ``fractional order derivative" and the ``delay" so as to describe the ``real behavior" of the natural systems. Bhalekar, Daftardar-Gejji, and coworkers propose numerical methods to solve fractional-order delay differential equations (FDDE) in \cite{bhalekar2011predictor,daftardar2015solving}. Stability analysis of the scalar FDDE is proposed by Bhalekar in \cite{bhalekar2016stability} where the regions of the stability are provided in the parameter plane. Stability, bifurcations and chaos in FDDEs are investigated in \cite{bhalekar2011fractional,bhalekar2010fractional,baleanu2015chaos,daftardar2012dynamics,bhalekar2012dynamical,bhalekar2013stability,bhalekar2022stability,gu2017artificial,shi2022chaos}. The detailed bifurcation analysis of scalar FDDE is presented by Bhalekar and Gupta in \cite{bhalekar2023can}. 
% 2nd order DDE
\par	The DDEs involving the second-order derivative are useful in modeling a wide range of systems. 
Second-order DDE with distributed delay appears as a model of hereditary dynamics (see \cite{volterra}, pp. 191, eq. (3)). Milton and Longtin \cite{milton1990evaluation} used these equations to study the human pupil cycle. Campbell et al. \cite{campbell1995complex} proved the existence of limit cycles, two-tori, and multi-stability in the damped harmonic oscillator with delayed feedback. Hopf bifurcations in the delayed Duffing oscillator are studied in \cite{erneux2023short}. Hybrid bistable device \cite{vallee1987second} in the optics can also be modeled by using these equations.
%Motivation % difficulties in analyzing ddes
\par This discussion motivated us to analyze the stability of the ``two-term fractional order delay differential equation," which is the generalization of second-order DDE. We provide the complete bifurcation analysis of this equation under various conditions on parameters and fractional order. We observed various stability behaviors of this system.
% We consider the linear autonomous equation $D^{2\alpha}x(t)+c D^{\alpha}x(t)=a x(t)+b x(t-\tau)$.
% structure of paper
\par The paper is organized as follows: 
Section \eqref{sec3} deals with some basic definitions and Theorems. Section \eqref{sec3.1} provides stability analysis of non-delayed case of equation \eqref{eq3.1}. In Section \eqref{sec3.4} we give the sufficient condition under which our equation \eqref{eq3.2} have the non-existence of critical values of delay. Section \eqref{sec3.5} deals with the conditions on the parameter for the existence of positive root. Some conditions under which the characteristic equation has roots with positive real part are described in Section \eqref{sec3.6}. The open problems are given in Section \eqref{sec3.7}. Section \eqref{section3.8} will provide the region in the parametric plane where the stability depends on delay parameter. In section \eqref{sec3.9}, we have some examples which validate our results. Section \eqref{sec3.10} summarizes the results given in the paper.
\section{Preliminaries}\label{sec3}
In this section, we provide some basic definitions described in the literature \cite{podlubny1998fractional,lakshmanan2011dynamics,kilbas2006theory,bhalekar2013stability,diethelm2002analysis}.
\begin{Def}
$\mathcal{L}_p[a,b]:=\{f:[a,b]\rightarrow\mathbb{R}; \textit{ f is measurable on $[a,b]$}  \textit{ and } \int_{a}^{b}|f(x)|^p dx<\infty\}$.
\end{Def}
\begin{Def}[Fractional Integral]
For any $f \in \mathcal{L}_{1}(0,b)$ the Riemann-Liouville fractional integral of order $\upmu >0$, is given by 

\begin{equation*}
\textit{I}^\upmu f(t)=\dfrac{1}{\Gamma(\upmu)}\int_{0}^{t}(t-\tau)^{\upmu-1}f(\tau)d\tau  , \quad   0<t<b.
\end{equation*}

\end{Def}
\begin{Def}[Caputo Fractional Derivative]
For $f ^m\in \mathcal{L}_{1}(0,b)$, $0<t<b$ and $m-1<\upmu\leq m$, $m \in \mathbb{N}$, the Caputo fractional derivative of function $f$ of order $\upmu$ is defined by,
\[\textit{D}^{\upmu} f(t)=
\begin{cases}
\frac{d^m}{dt^m} f(t) ,\textit{ if } \quad \upmu = m \\ \textit{I}^{m-\upmu}\dfrac{d^m f(t)}{dt^m}, \textit{ if } \quad  m-1< \upmu < m. 
\end{cases}\]
 Note that for  $m-1 < \upmu \leq m$, $m\in \mathbb{N},$

\[\textit{I}^\upmu\textit{D}^\upmu f(t)=f(t)-\sum_{k=0}^{m-1}\dfrac{d^k f(0)}{dt^k}\dfrac{t^k}{k!}.\]

\end{Def}
\begin{Def}\cite{vcermak2015stability} The two term FDE
 \begin{equation}\label{eq3.1}
 D^\alpha x(t) + c D^{2\alpha} x(t)=a_1 x(t),\quad 0<\alpha<1
 \end{equation}
 where $a_1\in\mathbb{R}$ is said to be \\
 (a) stable if all its solutions are bounded as $t\rightarrow\infty$.\\
 (b) asymptotically stable if all its solutions tend to zero as $t\rightarrow\infty$;\\
 (c) $t^{-\gamma}$ asymptotically stable if there is a real scalar $\gamma>0$ such that any solution of equation (\ref{eq3.1}) tends to zero like $O(t^{-\gamma})$ as $t\rightarrow\infty$.
  \end{Def}
 
 \begin{The}\cite{vcermak2015stability}\label{th3.1.1}
The FDE (\ref{eq3.1}) is asymptotically stable if and only if all the zeros $s$ of the polynomial  
 \begin{equation}\label{eq3.1.18}
 p(s)=s^2+\frac{1}{c} s-\frac{a_1}{c}
\end{equation}
satify
 \begin{equation}\label{eq3.1.12}
|arg (s)|>\alpha\pi/2.
\end{equation}
More precisely, the condition (\ref{eq3.1.12}) is necessary and sufficient for the $t^{-\gamma}$ asymptotic stability of equation (\ref{eq3.1}), where
\begin{eqnarray*}
\gamma=
\begin{cases}

&\alpha, \textit{ if } 0<\alpha\leq1/2\\
& 2\alpha-1,\textit{ if } 1/2<\alpha<1.\\
\end{cases}
\end{eqnarray*}
 \end{The}

 \section{Stability region for equation (\ref{eq3.1}) in the $a_1 c$-plane} \label{sec3.1}
 Roots of equation (\ref{eq3.1.18}) are
 
 \begin{equation*}
 s_1=\frac{-1+\sqrt{1+4 a_1 c}}{2 c},\quad s_2=\frac{-1-\sqrt{1+4 a_1 c}}{2 c}.
 \end{equation*}
  \textbf{Case 1} If $a_1 c>-1/4$ then $s_1, s_2\in\mathbb{R}$.\\
  The condition for stability of (\ref{eq3.1.12}) becomes  $s_1<0$ and $s_2<0$.
If $c<0$ then $s_2>0$ and the system is unstable. Therefore we assume that $c>0$. In this case, $s_2<0$. Further, $s_1<0$ if $\frac{-1}{4c}<a_1<0.$ Thus, the stability condition (\ref{eq3.1.12}) becomes
 \begin{equation}\label{eq3.1.5}
 c>0,\quad \frac{-1}{4c}<a_1<0.
\end{equation}
 \textbf{Case 2} If $a_1 c<\frac{-1}{4}$ then the roots of equation (\ref{eq3.1.18}) are

 \begin{align*}
 s_1=\frac{-1+i\sqrt{-1-4a_1 c}}{2c}, \quad s_2=\frac{-1-i\sqrt{-1-4a_1 c}}{2c}.
 \end{align*}

 Therefore, $|arg(s_1)|=|arg(s_2)|$.\\
 \textbf{Subcase 2.1}  $c>0$.\\
  So, 
 $arg(s_{1})=\pi-\arctan(\sqrt{-1-4a_1 c}).$

% $\Rightarrow arg(\lambda^\alpha)=\pi-\arctan(\sqrt{-1-4a_1 c}).$\\
 Since, \begin{align*}
 \frac{-\pi}{2}\leq-\arctan(\sqrt{-1-4a_1 c})\leq 0,\\
 \pi-\frac{\pi}{2}\leq \pi-\arctan(\sqrt{-1-4a_1 c})\leq \pi.
 \end{align*}
Therefore, we get $arg(s_{1})\geq\frac{\pi}{2}>\frac{\alpha \pi}{2}$ for any $\alpha\in(0,1)$. Thus, the stability condition (\ref{eq3.1.12}) reduces to

 \begin{equation}\label{eq3.1.4}
  c>0 \textit{ and } a_1 c<\frac{-1}{4}.
 \end{equation}\\

 \textbf{Subcase 2.2} $c<0.$

 We have, $|arg(s_1)|=|\arctan (\sqrt{-1-4a_1 c})|$. 
 
 If $a_1 c<\dfrac{-\tan^2(\alpha\pi/2)-1}{4}$ then, we get
 
  $4 a_1 c+1<-\tan^2(\alpha\pi/2)$.\\
 
 $\Rightarrow \sqrt{-(4 a_1 c+1)}>\tan(\alpha\pi/2)$.\\
 
 So, $\arctan(\sqrt{-(4 a_1 c+1)})>\alpha\pi/2$.\\
 
 $\Rightarrow|arg(s_1)|>\alpha\pi/2.$\\

 Hence, by Theorem (\ref{th3.1.1}), equation (\ref{eq3.1}) is stable if
  \begin{equation}
  c<0 \textit{ and } a_1 c<\frac{-\tan^2(\alpha\pi/2)-1}{4}.\label{eq3.1.24}
  \end{equation}

 Using the conditions (\ref{eq3.1.5}), (\ref{eq3.1.4}) and (\ref{eq3.1.24}), we sketch the stability region of equation (\ref{eq3.1}) in $a_1c$ plane as shown in Figure(\ref{fig1}).\\

Note that, in the first and third quadrant both the roots $s_1$ and $s_2$ are real and one of the root is positive. So, the first and third quadrants are unstable. The second quadrant is divided in two parts by the curve $\Gamma_1: a_1 c=-1/4$. on the right side of $\Gamma_1$ the roots $s_j$ of (\ref{eq3.1.4}) are real and negative (Case 1) whereas on the left side they are complex and satisfying (\ref{eq3.1.12}) (subcase 2.1). Therefore, we get the stable solutions of equation (\ref{eq3.1}) for all $a_1<0$ and $c>0$. \\

In the fourth quadrant on the left side of the curve $\Gamma_1$ both the roots are positive and hence the system is unstable. There exists one more curve $\Gamma_2: a_1 c=\frac{-\tan^2(\alpha\pi/2)-1}{4}$ such that on the right side of it, the $s_j$ are complex and their argument greater than $\alpha\pi/2$ (refer (\ref{eq3.1.24})). Between $\Gamma_1$ and $\Gamma_2$, the roots $s_j$ are complex but violating the condition (\ref{eq3.1.24}). Thus, this bounded region is unstable.\\
Note that, for any fixed $a_1>0$, $\Gamma_2(a_1,c)<\Gamma_1(a_1,c),\forall c<0.$\\

Furthermore, $\Gamma_1$ and $\Gamma_2$ will not intersect each other.\\
\begin{figure}
    \centering
    \includegraphics{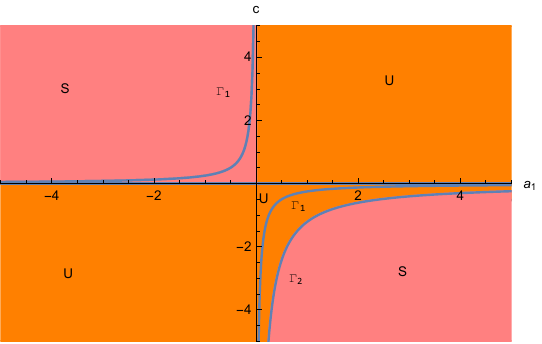}
    \caption{Stability region for the non-delayed equation (\ref{eq3.1}), where S=Stable Region and U=Unstable Region. }
    \label{fig1}
\end{figure}
\section{Two-term fractional delay differential equation (FDDE)}\label{sec3.4}
 The stability analysis of an autonomous delay differential equation
 \begin{equation}\label{eq3.2.1}
 D^\alpha x(t)+ c D^{\beta}x(t)=a x(t)+b x(t-\tau)
\end{equation} with $\alpha>\beta$  involving two fractional-order derivatives is given in \cite{bhalekar2019analysis}. In this work, we provide detailed stability and bifurcation analysis of the following FDDE:  
 \begin{equation}
 D^\alpha x(t)+ c D^{2\alpha}x(t)=a x(t)+b x(t-\tau)\label{eq3.2}
 \end{equation}
 and sketch the stable/unstable regions in the parameter planes.
 This system (\ref{eq3.2}) reduces to the system (\ref{eq3.1}) when $\tau=0$ with $a_1=a+b.$\\
 The characteristics equation \cite{bhalekar2019analysis} of the given FDDE (\ref{eq3.2}) is \\
 \begin{equation}
 \lambda^\alpha+ c \lambda^{2\alpha}=a+b \exp(-\lambda\tau).\label{eq3.3}
 \end{equation}
 There is a change in stability when the characteristics root $\lambda$ crosses the imaginary axis $\lambda=i v$, $(v>0)$ \cite{bhalekar2019analysis}.\\
 So, by substituting $\lambda=i v$ in the equation (\ref{eq3.3}), as in \cite{bhalekar2019analysis}, we get\\
 \begin{equation}
 (i v)^\alpha+ c (i v)^{2\alpha}=a+b \exp(-iv\tau_*),
 \end{equation}
 where $\tau_*$ is the critical value of delay where change in stability can occur.\\ 
 Separating the real and imaginary parts, we get
      \begin{equation}
      v^\alpha\cos\left(\frac{\alpha\pi}{2}\right)+ c v^{2\alpha}\cos(\alpha\pi)-a=b \cos(v\tau_*),\label{eq3.4}
      \end{equation}
      \begin{equation}
      v^\alpha \sin\left(\frac{\alpha\pi}{2}\right)+c v^{2\alpha}\sin(\alpha\pi)=-b\sin(v\tau_*).\label{eq3.5}
      \end{equation}
Squaring and adding \eqref{eq3.4} and \eqref{eq3.5}, we get
\begin{equation}\label{value of valpha}
v^{2\alpha}+c^2v^{4\alpha}+2 c v^{3\alpha}\cos(\alpha\pi/2)-2 a v^{\alpha}\cos(\alpha\pi/2)-2 a c v^{2\alpha}\cos(\alpha\pi)+a^2-b^2=0.
\end{equation} 
Furthermore, equation \eqref{eq3.4} gives 
\begin{equation}\label{critical value of tau with plus sign}
 \tau_+(n)(v)=\frac{2n\pi+\arccos(\frac{v^\alpha\cos(\frac{\alpha\pi}{2})+c v^{2\alpha}\cos(\alpha\pi)-a}{b})}{v^{1/\alpha}},\quad n=0,1,2,3,\ldots 
\end{equation}
and 
\begin{equation}\label{critical value of tau with minus sign}
 \tau_-(n)(v)=\frac{2n\pi-\arccos(\frac{v^\alpha\cos(\frac{\alpha\pi}{2})+c v^{2\alpha}\cos(\alpha\pi)-a}{b})}{v^{1/\alpha}},\quad n=1,2,\ldots.
\end{equation}
which are the critical values of delay for each root $v$ of the equation \eqref{value of valpha}. Note that, for given root $v$ of equation \eqref{value of valpha}, either $\tau_{+}(n)$ or $\tau_{-}(n)$ will be the critical value that will be decided by the equation \eqref{eq3.5}. So, we can take  $\tau_+(n)(v)$ if $\frac{v^\alpha\sin(\frac{\alpha\pi}{2})+c v^{2\alpha}\sin(\alpha\pi)}{-b}$ lies between $2k\pi$ to $(2k+1)\pi$ and $\tau_-(n)(v)$ are the critical values if LHS of \eqref{eq3.5} lies between $(2k+1)\pi$ to $(2k+2)\pi$, $k=0,1,2,\ldots$.

\subsection{ The conditions for the non existence of $\tau_*$ }
Critical value $\tau_*$ exists if and only if equations (\ref{eq3.4}) and (\ref{eq3.5}) are satisfied. Let us define the curves $L(v)= v^\alpha\cos\left(\frac{\alpha\pi}{2}\right)+ c v^{2\alpha}\cos(\alpha\pi)-a$ and\\  $R(v)=b \cos(v\tau_*)$.\\The curves $L(v)$ and $R(v)$ intersect each other if and only if the equation (\ref{eq3.4}) is satisfied. Note that if one of the equations (\ref{eq3.4}) and (\ref{eq3.5}) is not satisfied then the stability will remain same as the non-delayed case i.e. if the FDE system (\ref{eq3.1}) is stable (respectively, unstable), then the FDDE system (\ref{eq3.2}) will also be stable (respectively, unstable) for all $\tau\geq0$.\\
So, the non-existence of $\tau_*$ will give the existence of delay independent stability regions.\\
In this section, we provide some (sufficient) conditions for the nonexistence of the critical value $\tau_*$.
\begin{The}\label{th3.1}
If $c>0$, $0<\alpha\leq 1/2$ and $a<-|b|$, then there does not exists critical value $\tau_*$ and the stability of system (\ref{eq3.2}) is independent of delay $\tau.$
\end{The}
\begin{proof}
For $c>0$ and $0<\alpha\leq1/2$ we have $v^\alpha\cos\left(\frac{\alpha\pi}{2}\right)+ c v^{2\alpha}\cos(\alpha\pi)>0.$\\
So, $-a<-a+(v^\alpha\cos\left(\frac{\alpha\pi}{2}\right)+ c v^{2\alpha}\cos(\alpha\pi)).$
\begin{equation}
\Rightarrow -a<L(v).\label{eq3.6}
\end{equation}
Also by the assumption we have,
\begin{equation}
|b|<-a.\label{eq3.7}
\end{equation}
Therefore, from (\ref{eq3.6}) and (\ref{eq3.7}) we have,\\
$|b|<L(v).$\\
Further, the range of $R(v)$ is $(-|b|,|b|)$. Hence, there is no intersection point between the two curves $L(v)$ and $R(v)$ for any $v$, $\alpha$, $c$, $a$ and $b$.\\
Therefore equation (\ref{eq3.4}) is not satisfied. Hence $\tau_*$ does not exists.
\end{proof}
\textbf{Note 1:} Recall, $a_1=a+b$. In Theorem (\ref{th3.1}), the condition $a<-|b|$ is equivalent to $a_1<0$, if $b>0$ and $a_1<2b,$ if $b<0$.\\
This will be used to sketch the delay-independent stability region in Figure (\ref{fig4.1}).
\begin{The}\label{th3.2}
For $c<0$ and $0<\alpha<1/2$, there will be no delay dependent stability region if $|b|<a+\frac{1}{4c}\cos^2({\alpha\pi/2})\sec(\alpha\pi).$
\end{The}
\begin{proof}Suppose $0<\alpha<1/2$ and $c<0$.
In this case, the curve $L(v)$ has local maxima at $v=(\frac{-\cos(\frac{\alpha\pi}{2})\sec(\alpha\pi)}{2 c})^{\frac{1}{\alpha}}$ and the maximum value is \\
$max=-a-\frac{\cos^2(\frac{\alpha\pi}{2})\sec(\alpha\pi)}{4 c}.$ \\
By assumption, $max<0.$\\
Therefore, $L(v)$ is negative for all $v>0.$\\
Further, by assumption, $-|b|>max$, where $-|b|$ is the minimum value of $R(v)$.\\
This shows that there is no intersection between the curves $L(v)$ and $R(v).$\\
Therefore, there does not exists any $\tau_*$ satisfying equation (\ref{eq3.4}). Hence there will be no delay dependent stability region.
\end{proof}
\begin{Note}\label{note1}
 Now, we utilize Theorem (\ref{th3.2}) and sketch the delay independent stability region in $a_1 c$-plane as below:\\
The condition $|b|<a+\frac{1}{4c}\cos^2({\alpha\pi/2})\sec(\alpha\pi)$ in Theorem (\ref{th3.2}) is equivalent to

\begin{eqnarray}\label{eq3.8}
\begin{cases}

a_1>\frac{-\cos^2(\frac{\alpha\pi}{2})\sec(\alpha\pi)}{4 c},\textit{ if } b<0&\\
\textsl{ and } a_1>\frac{-\cos^2(\frac{\alpha\pi}{2})\sec(\alpha\pi)}{4 c}+ 2 b,\textit{ if } b>0.
\end{cases}
\end{eqnarray}
We define the boundary curves of the regions in (\ref{eq3.8}) as\\

$\Gamma_3: a_1=\frac{-\cos^2(\frac{\alpha\pi}{2})\sec(\alpha\pi)}{4 c}$ and\\
$\Gamma_4: a_1=\frac{-\cos^2(\frac{\alpha\pi}{2})\sec(\alpha\pi)}{4 c}+2b$, respectively.\\
These regions are in the fourth quadrant of $a_1 c$-plane because $c<0$ and $0<\alpha<1/2$.
\end{Note}
Recall that, the boundary of stable region of non-delayed equation (\ref{eq3.1}) in this region is described by the curve $\Gamma_2$ (cf. Figure (\ref{fig1})).\\

We need to find the positions of $\Gamma_3$ and $\Gamma_4$ relative to $\Gamma_2$ to discuss the delay independent stability of equation (\ref{eq3.2}).\\

$\bullet$ $\Gamma_2$ \textit{ and } $\Gamma_3$:\\
Since $-(\tan^2(\alpha\pi/2)+1)>-\cos^2(\alpha\pi/2)\sec(\alpha\pi)$,\\
the curve $\Gamma_3$ cannot intersect $\Gamma_2$ and is on the right side of $\Gamma_2$ in the $a_1c$-plane.\\
Thus, the intersection between the stable region bounded by $\Gamma_2$ and the delay independent stable region bounded by $\Gamma_3$ is the region defined by the first inequality in (\ref{eq3.8}). This is sketched in Figure (\ref{fig4.1})(Fourth quadrant). In this region, the system (\ref{eq3.2}) is stable for all $\tau\geq 0$.\\

$\bullet$ $\Gamma_2$ \textit{ and } $\Gamma_4$: \\
Using the similar arguments, the intersection between the regions bounded by $\Gamma_2$ and $\Gamma_4$ is given by the second inequality in (\ref{eq3.8}).
This delay independent stable region is sketched in the fourth quadrant of Figure (\ref{fig4.1}).\\

Now, we provide some results for the case $1/2<\alpha<1.$\\
\begin{The}\label{th3.3}
The stability of system (\ref{eq3.2}) is independent of delay $\tau$ if $c<0$, $1/2<\alpha<1$ and $a<-|b|$.\\
\end{The}
\begin{proof}
Note that $v^\alpha\cos\left(\frac{\alpha\pi}{2}\right)+ c v^{2\alpha}\cos(\alpha\pi)$ is always positive for $c<0$ and $1/2<\alpha<1.$\\

So, $-a<L(v).$\\
Therefore, $|b|<L(v)$ by assumption.\\
Since, $|b|$ is maximum value of $R(v),$ the two curves $L(v)$ and $R(v)$ will never intersect each other under the assumptions of this Theorem.\\
Hence, the stability region of system (\ref{eq3.2}) is independent of delay $\tau$.
\end{proof}
\begin{The}\label{th3.4}
If $c>0$, $1/2<\alpha<1$ then, the stability region is delay independent if $|b|<a+\frac{\cos^2(\alpha\pi/2)\sec(\alpha\pi)}{4 c}$.
\end{The}
\textbf{Note} Proof of Theorem (\ref{th3.4}) is analogous with that of Theorem (\ref{th3.2}).

Using the similar arguments as in Note (1) and Note (2), we sketch the delay independent stability regions of system (\ref{eq3.2}) in Figure \eqref{fig4.1} using Theorem (\ref{th3.3}) and Theorem (\ref{th3.4}) respectively. Note that, the system (\ref{eq3.2}) is unstable for all $\tau\geq 0$ if the conditions of Theorem (\ref{th3.3}) or Theorem (\ref{th3.4}) are satisfied.\\
\begin{figure}
	\subfloat[Delay independent stability region of equation (\ref{eq3.2}) for $b>0$ and $0<\alpha<1/2$]{%
		\includegraphics[scale=0.87]{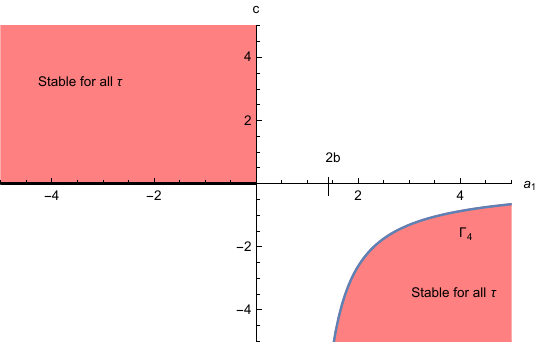}
	}\hspace{0.1cm}
	\subfloat[Delay independent stability region of equation (\ref{eq3.2}) for $b<0$ and $0<\alpha<1/2$]{%
		\includegraphics[scale=0.87]{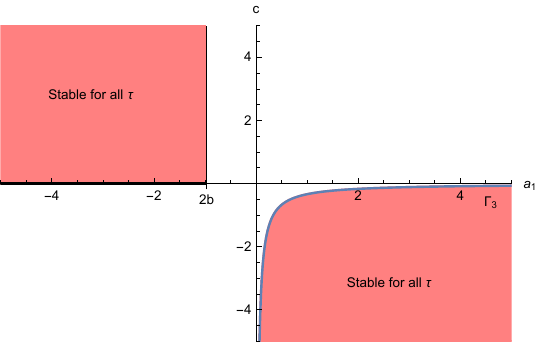}
	}\hspace{0.1cm}
	\subfloat[The delay independent stability region of equation (\ref{eq3.2}) for $b>0$ and $1/2<\alpha<1$]{%
		\includegraphics[scale=0.87]{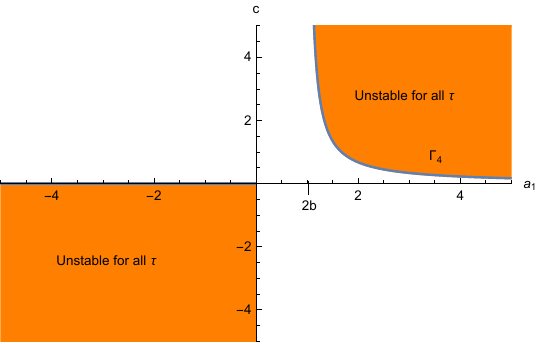}
	}\hspace{0.1cm}
	\subfloat[The delay independent stability region of equation (\ref{eq3.2}) for $b<0$ and $1/2<\alpha<1$]{%
		\includegraphics[scale=0.87]{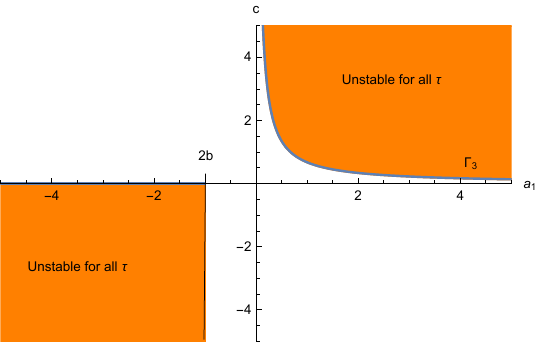}
	 }
\caption{}
	\label{fig4.1}
\end{figure}

\section{Geometrical method to find the stability}\label{sec3.5}
 If 
 \begin{equation}\label{expression for p}
 P(\lambda)=\dfrac{\lambda^\alpha+ c \lambda^{2\alpha}-a}{b}
 \end{equation}
  and 
  \begin{equation}\label{expression for q}
  Q(\lambda,\tau)=\exp{(-\lambda\tau)}
  \end{equation}
   then the characteristic equation \eqref{eq3.3} becomes
 \begin{equation}
 P(\lambda)=Q(\lambda,\tau).   
 \end{equation}
 Note that $\exists$ a characteristic root $\lambda_0$ if the graph of $P(\lambda)$ intersects the graph of $Q(\lambda,\tau)$ at $\lambda_0$, for some $\tau>0$. Moreover, the image set $\{Q(\lambda,\tau)\quad|\lambda\in\mathbb{C}, \tau>0\}$ will be a punctured unit disc in $\mathbb{C}$.\\
 If $\lambda\in\mathbb{R}$ (i.e. $Im(\lambda)=0$) then the graphs of $\mathcal{P}$ and $Q$ are subsets of $\mathcal{R}^2$. In this case, the intersection of these graphs at $\lambda>0$ will be sufficient condition for the instability of system \eqref{eq3.2}. This will be discussed in Section \eqref{section of positive real root}.\\
 If $Im(\lambda)\neq0$ then the graphs of $\mathcal{P}$ and $Q$ are in $\mathcal{C}^2$ and we cannot observe these intersections. However, we can compare the image sets of $\mathcal{P}$ and $Q$ with $Re(\lambda)>0$ and get some information on the stability of equation \eqref{eq3.2}. This will be provided in Section \eqref{sec3.6}. 
 \subsection{Existence of positive real root $\lambda$ }\label{section of positive real root}
 Now, for any real $\lambda>0$ and $\tau>0$, $0<exp(-\lambda\tau)\leq1$, i.e. the range of $Q(\lambda,\tau)=(0,1]$.\\
 Furthermore, if $\exists$ $\lambda>0$ such that $P(\lambda)\in(0,1]$ then we can find some $\tau$ such that $P(\lambda)=Q(\lambda,\tau)$. Therefore, this $\lambda$ is positive real root of characteristic equation.
 
  Note that the existence of a positive real characteristic root is sufficient for the instability of system \eqref{eq3.2}.
% \end{frame}
 %\begin{frame} 
 If $\exists$ $\lambda>0$ such that $P(\lambda)\in(0,1]=Range(Q)$ then we can find some $\tau$ such that $P(\lambda)=Q(\lambda,\tau)$. Therefore, this $\lambda$ is positive real root of characteristic equation.\\
 
% \end{frame}
 %\begin{frame} 

    \begin{figure}
    \centering
    \includegraphics[scale=0.8]{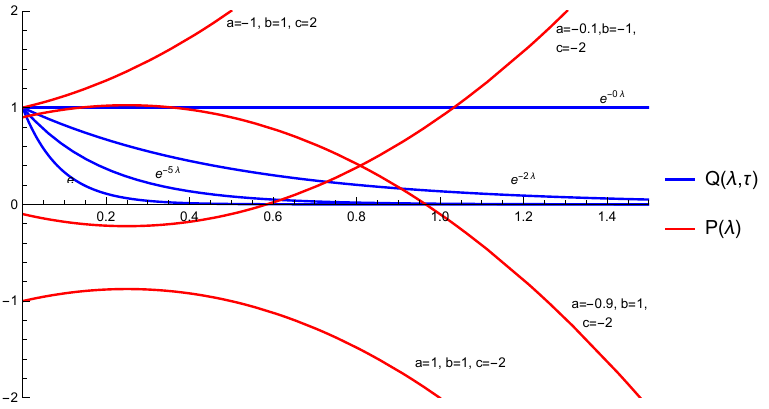}
    \caption{ Graph of $P(\lambda)$ and $Q(\lambda,\tau)$ with respect to $\lambda$}
  % \label{fig1}
\end{figure} 
% \end{frame}
%\begin{frame}
\begin{itemize}
\item Various conditions under which the characteristic equation has a positive real root are described below.
\begin{enumerate}
\item[(1)] If Range($P$)$\supset(0,1]$ then system \eqref{eq3.2} is unstable $\forall\tau>0$.
\item[(2)] If $P(0)=-a/b<1$ and $P(\lambda)=1$ for some $\lambda>0$ then system \eqref{eq3.2} is unstable $\forall$ $\tau\geq0$ even if $(0,1]$ is not the subset of Range(P) as shown in Figure (\ref{maingraph1}).
\end{enumerate}
\begin{figure}
     \centering
     \begin{subfigure}[b]{0.38\textwidth}
         \centering
         \includegraphics[width=\textwidth]{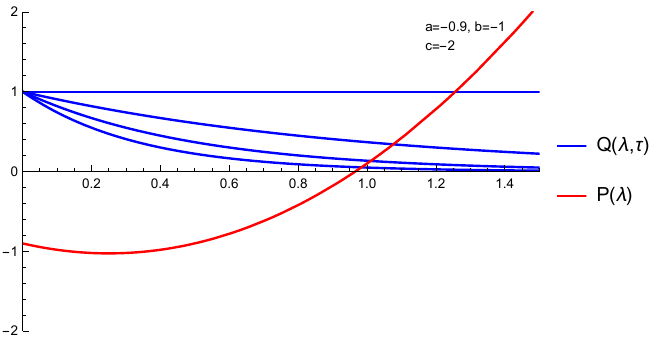}
         \caption{$\frac{-a}{b}<0<1$}
      %   \label{fig:y equals x}
     \end{subfigure}
     \hfill
     \begin{subfigure}[b]{0.37\textwidth}
         \centering
         \includegraphics[width=\textwidth]{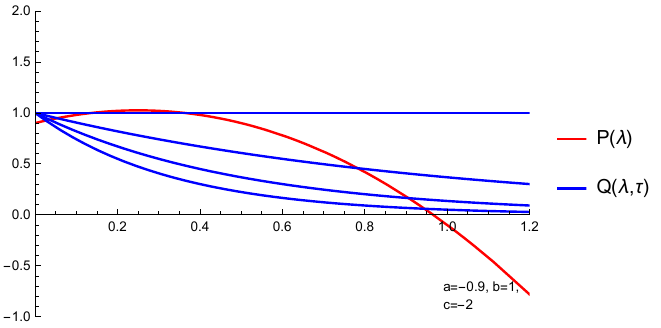}
         \caption{$0<\frac{-a}{b}<1$ and $P(\lambda)\shortrightarrow -\infty$}
       %  \label{fig:three sin x}
     \end{subfigure}
     \hfill
     \begin{subfigure}[b]{0.34\textwidth}
         \centering
         \includegraphics[width=\textwidth]{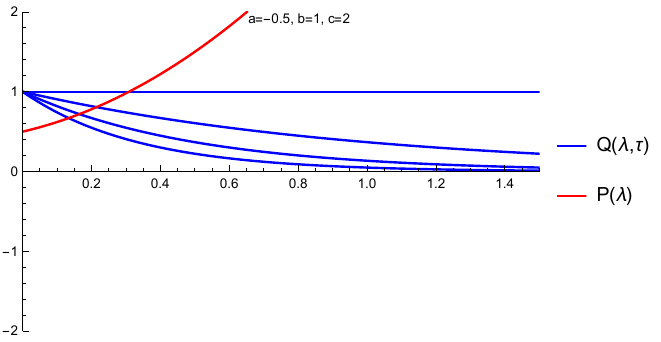}
         \caption{$0<\frac{-a}{b}<1$ and $P(\lambda)\shortrightarrow\infty$}
        % \label{fig:five over x}
     \end{subfigure}
        \caption{ $-a/b<1 $ and $P(\lambda)=1$ has one positive root $\lambda$ for some $\lambda$}
       \label{maingraph1}
\end{figure} 
\end{itemize}
  \begin{enumerate}
      \item[(3)] If $P(0)=-a/b>1$ and $P(\lambda)=0$ for some $\lambda>0$ then also Range($P$)$\supset(0,1]$ (cf. Figure (\ref{maingrapf2})) and system \eqref{eq3.2} is unstable $\forall$ $\tau\geq0$. 
  \end{enumerate}
  \begin{figure}
     \centering
     \begin{subfigure}[b]{0.36\textwidth}
         \centering
         \includegraphics[width=\textwidth]{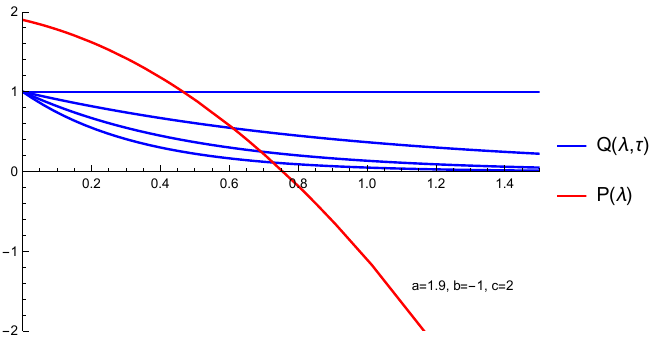}
         \caption{$\frac{-a}{b}>1$ and $P(\lambda)\shortrightarrow -\infty$}
       %  \label{fig:y equals x}
     \end{subfigure}
     \hfill
     \begin{subfigure}[b]{0.36\textwidth}
         \centering
         \includegraphics[width=\textwidth]{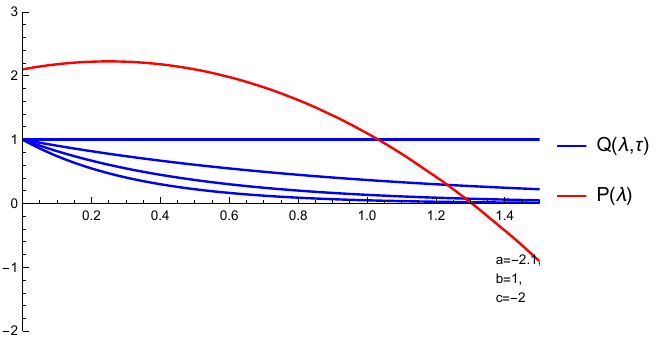}
         \caption{$\frac{-a}{b}>1$ and $P(\lambda)$ has a local maxima}
      %   \label{fig:three sin x}
     \end{subfigure}
     \hfill
     \begin{subfigure}[b]{0.37\textwidth}
         \centering
\includegraphics[width=\textwidth]{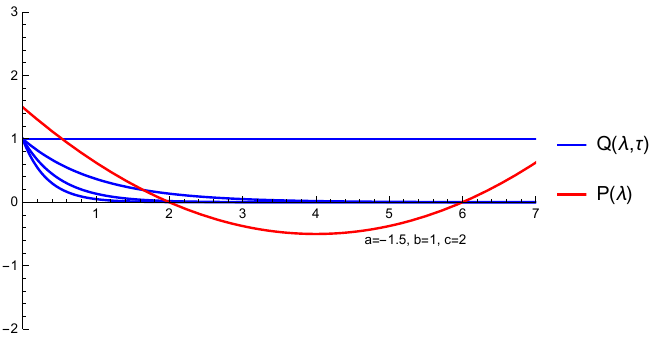}
         \caption{$\frac{-a}{b}>1$ and $P(\lambda)$ has a local minima}
      %   \label{fig:five over x}
     \end{subfigure}
        \caption{ $-a/b>1$ and $P(\lambda)=0$ for some $\lambda$}
       \label{maingrapf2}
\end{figure} 
     \begin{enumerate}
         \item [(4)] If $P(0)=-a/b>1$ and the curve $P(\lambda)$ has local minima between $0$ to $1$ then there exists critical value $\tau_*$ such that $Q(\lambda,\tau_*)$ touches $\mathcal{P}(\lambda)$ and equation \eqref{eq3.2} is unstable for $0<\tau<\tau_*$(Figure \eqref{maingrapf3}). 
         \end{enumerate}
     \begin{figure}
  \centering
    \includegraphics[scale=1.3]{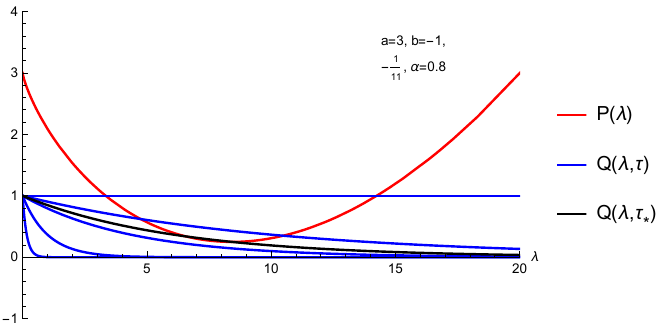}
    \caption{ $-a/b>1$ and $P(\lambda)$ for some $\lambda>0$ has a local minima between 0 to 1}
    \label{maingrapf3}
\end{figure}

% \end{frame}
 %\begin{frame}
 \begin{enumerate}
     \item [(5)] If $P(0)=-a/b<1$ and the curve $P(\lambda)$ for some $\lambda>0$ has local maxima lies between 0 to 1 then also there exists critical value $\tau_*$ such that for all $\tau>\tau_*$ the curve $P(\lambda)$ intersects with $Q(\lambda,\tau)$ (Figure(\ref{maingrapf4})). Hence, equation \eqref{eq3.2} is unstable for $\tau>\tau_*$.
 \end{enumerate}
    \begin{figure}
  \centering
    \includegraphics[scale=0.9]{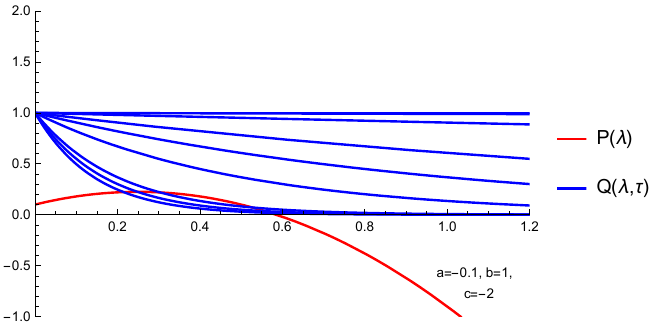}
    \caption{ $-a/b<1$ and $P(\lambda)$ for some $\lambda>0$ has a local maxima between 0 to 1}
   \label{maingrapf4}
\end{figure}  
% \end{frame}
 %\begin{frame}    
\begin{The}\label{drcthm3.1}
	The equation \eqref{eq3.2} is unstable $\forall\tau\geq 0$ if any one of the following conditions hold:
	\begin{itemize}
		\item[(i)] $c>0$, $b<0$ and $\frac{-a}{b}>1$.
		\item[(ii)] $c>0$, $b>0$ and $\frac{-a}{b}<1$.
		\item[(iii)] $c<0$, $b<0$ and $\frac{-a}{b}<1$.
		\item[(iv)] $c<0$, $-b<a<\frac{-1}{4 c}$ and $b<0$.
	\end{itemize}
\end{The}
\begin{proof}
	(i) Given that $P(0)>1$.
	For $b<0$ and $\frac{-a}{b}>1$ we have $a>0$.\\
	Since $c>0$ and $a>0$, $1+4ac>1$. Hence, one root of $P(\lambda)=0$ i.e. $\big(\frac{-1+\sqrt{1+4a c}}{2c}\big)^\frac{1}{\alpha}$ is positive. Therefore, Range($P)\supset(0,1]$. So, $P(\lambda)$ intersects with $Q(\lambda,\tau)$ for every $\tau\geq0$ as shown in Figure (\ref{maingrapf2}). Therefore, system \eqref{eq3.2} is unstable $\forall$ $\tau\geq0$.
	
	\vspace{0.2cm}
	(ii)  	We have $-a/b<1$ and $b>0$ so $a+b>0$. Hence for $c>0$, $1+4(a+b)c>1$. Therefore, $P(\lambda)=1$ has one positive root viz. $\Big(\frac{-1+\sqrt{1+4c(a+b))}}{2c}\Big)^\frac{1}{\alpha}$.
	So, the curve $P(\lambda)$ intersects $Q(\lambda,\tau)$  $\forall\tau\geq0$ by the Figure (\ref{maingraph1}) and system \eqref{eq3.2} is unstable.
	
	\vspace{0.2cm}
	(iii) If $P(\lambda)=1$ has any positive root and $P(0)=\frac{-a}{b}<1$ then, from Figure(\ref{maingraph1}) the curve $P(\lambda)$ will intersect with $Q(\lambda,\tau)$ for every $\tau\geq0$ and $\lambda>0$.  
	Also, by solving $P(\lambda)=1$ we get
	\begin{equation*}
		\lambda=\Big(\frac{-1\pm \sqrt{1+4 c(a+b) }}{2 c}\Big)^\frac{1}{\alpha}.
	\end{equation*}
	Since $-a/b<1$, we get $a+b<0$. Furthermore, by assumption $c<0$.
	 So we have $1+4 c (a+b)>0$. Hence,  $\frac{-1-\sqrt{1+4 c(a+b)}}{2c}>0$. \\
	This implies that $P(\lambda)=1$ has a positive root. Hence, system \eqref{eq3.2} is unstable $\forall\tau>0$.
	
	\vspace{0.2cm}
	(iv) We have $-b<a<\frac{-1}{4c}$\\
	$\Rightarrow 1+4 a c>0$.\\
	Hence, for $c<0$ one root of $P(\lambda)=0$ i.e. $\frac{-1-\sqrt{1+4 a c}}{2c}$ is always positive. Hence, by the Figure(\ref{maingrapf2}) the curve $P(\lambda)$ intersects with $Q(\lambda,\tau)$ for every $\tau\geq0$. So, the system \eqref{eq3.2} is unstable $\forall \tau\geq0$.

\end{proof}

     \begin{The}\label{drcthm1}
         When $c<0$, $b>0$ and $-\infty<a<\frac{-1}{4c}-b$ then system \eqref{eq3.2} is unstable $\forall\tau\geq 0$
     \end{The}
     \begin{proof}
     
        \textbf{Step-1} We have $P(0)=\frac{-a}{b}>1$. If we could prove that $P(\lambda)=0$ has a positive real root $\lambda$ then from Figure(\ref{maingrapf2}) the curve $P(\lambda)$ will intersects with $Q(\lambda,\tau)$ for every $\tau\geq0$. 
If $c<0$ then  $\frac{-1-\sqrt{1+4a c}}{2 c}$ is positive because $a<0$ and $c<0$ so $1+4 a c>1$.
 Therefore, the root $\lambda=\Big(\frac{-1-\sqrt{1+4a c}}{2 c}\big)^\frac{1}{\alpha}$ of $P$ is positive. So, Range($P\supset(0,1]$.\\
 Hence, in this case $P(\lambda)$ intersects with $Q(\lambda,\tau)$ for every $\tau\geq 0$.\\
\textbf{Step-2} Consider $a_1=0$ or $a=-b$ in this case $P(0)=1$. Also $a=-b$ implies that $a$ is a negative number. So, $\frac{-1-\sqrt{1+4 a c}}{2c}$ is positive for $c<0$. Hence, $P(\lambda)=0$ has one positive root. So, the system \eqref{eq3.2} is unstable for all $\tau\geq0$ as the graph of $P(\lambda)$ intersects with the curves $Q(\lambda,\tau)$ for every $\tau\geq0$.\\ 
\textbf{Step-3} We have $-b<a$ so $P(0)=\frac{-a}{b}<1$. Also, $a<\frac{-1}{4c}-b$ therefore, $1+4 c a_1>0$. Hence, $\frac{-1-\sqrt{1+4 c(a+b)}}{2c}>0$ for $c<0$. So, one root of $P(\lambda)=1$ i.e. $\Big(\frac{-1-\sqrt{1+4 c(a+b)}}{2c}\Big)^\frac{1}{\alpha}$ is positive. Therefore, by the Figure \eqref{maingraph1} the curve $P(\lambda)$ intersects with the curves $Q(\lambda,\tau)$ for every $\tau\geq0$. Hence, system \eqref{eq3.2} is unstable for all $\tau\geq0$.     
     \end{proof}

\begin{The}
 If $0<a+b<\frac{-1}{4 c}<a$ and $b<0$ then $\exists$ $\tau_*=\frac{-\log(\frac{-1-4ac}{4 bc})}{\Big(\frac{-1}{2c}\Big)^\frac{1}{\alpha}}$ such that the system is unstable if $0\leq\tau<\tau_*.$       
\end{The}
\begin{proof}
Since $b<0$, the function $P$ has local minima at $\lambda=\Big(\frac{-1}{2c}\Big)^\frac{1}{\alpha}$ and the minimum value of $P$ is $\frac{-1-4 a c}{4 bc}$.\\ 
Also, by the assumption, $P(0)=-a/b>1$, we can expect such $\tau_*$ if the minimum value will lies from 0 to 1.  \\
We have $\frac{-1}{4c}<a$ and $c<0$ so $-1-4ac>0$. Therefore, for $b<0$, $\frac{-1-4 a c}{4 bc}>0$ i.e. the minimum value is positive. \\
Also, $0<a+b<\frac{-1}{4 c}$ we have $4c(a+b)>-1$. 
Therefore, $\frac{-1-4 ac}{4 b c}<1$ which shows that the minimum of $P$ lies between $0$ to $1$. Hence, there exists $\tau_*$ such that
\begin{equation*}
\frac{-1-4ac}{4bc}=e^{-\Big(\frac{-1}{2c}\Big)^{1/\alpha}\tau_*}.
\end{equation*}
Taking log both side we get, \\
$\tau_*=\frac{-\log(\frac{-1-4ac}{4 bc})}{\Big(\frac{-1}{2c}\Big)^\frac{1}{\alpha}}$. So, for every $0\leq\tau<\tau_*$ $P(\lambda)$ intersects $Q(\lambda,\tau)$ for some $\lambda>0$. Hence, the system will be unstable $\forall 0\leq\tau<\tau_*$ by Figure(\ref{maingrapf3}).
\end{proof}
\begin{The}
If $c<0$, $b>0$, $a>-b$ and $\frac{-1}{4 c}-b<a<\frac{-1}{4 c}$ then 
\begin{equation}\label{drc6.1.1} 
\tau_*=\frac{-\log(\frac{-1-4ac}{4 bc})}{\Big(\frac{-1}{2c}\Big)^\frac{1}{\alpha}}
\end{equation}
such that $\tau>\tau_*$ the system \eqref{eq3.2} is unstable.
\end{The}
\begin{proof} If $b>0$ the curve $P(\lambda)$ gives local maxima at $\lambda=(\frac{-1}{2c})^{-1/\alpha}$ and the maximum value is $\frac{-1-4ac}{4 bc}$.\\
Since, $a<\frac{-1}{4 c}$ so for $c<0$ we get $-1-4ac<0$. Therefore, $\frac{-1-4ac}{4 bc}>0$ for $b<0$ and $c>0$. Hence, the maximum value is greater than 0.\\
Also, by assumption, $\frac{-1}{4 c}<a+b$ so for $c<0$,\\
$\Rightarrow -1>4 c(a+b)$\\
$\Rightarrow \frac{-1-4ac}{4 b c}<1$.\\
So, the maximum value lies in between 0 to 1. Hence, here also there exists a $\tau_*$ given by equation \eqref{drc6.1.1} such that for all $\tau>\tau_*$, $P(\lambda)$ intersects $Q(\lambda,\tau)$ for some $\lambda>0$ from Figure(\ref{maingrapf4}). So, equation \eqref{eq3.2} is unstable $\forall$ $\tau>\tau_*$.
\end{proof}
Using these Theorems we get the delay independent stability results as:
\begin{figure}
	\subfloat[When $b>0$ and $0<\alpha<1/2$]{%
		\includegraphics[scale=0.7]{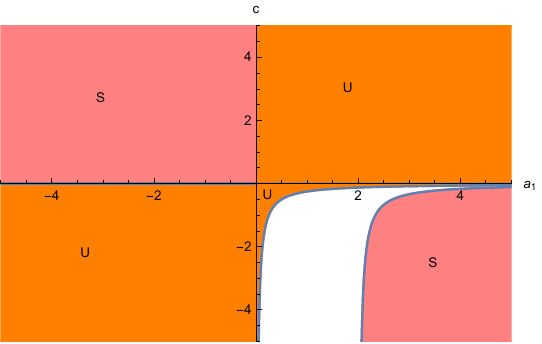}
	}\hspace{0.1cm}
	\subfloat[ When $b<0$ and $0<\alpha<1/2$]{%
		\includegraphics[scale=0.7]{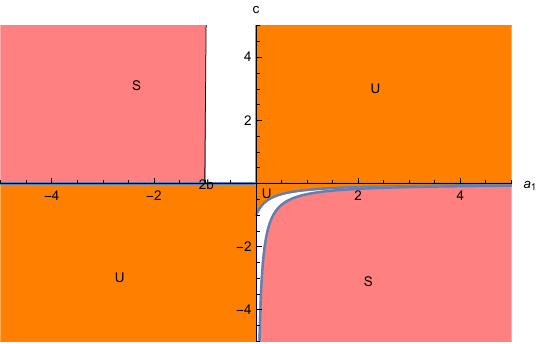}
	}\hspace{0.1cm}
	\subfloat[ When $b>0$ and $1/2<\alpha<1$]{%
		\includegraphics[scale=0.7]{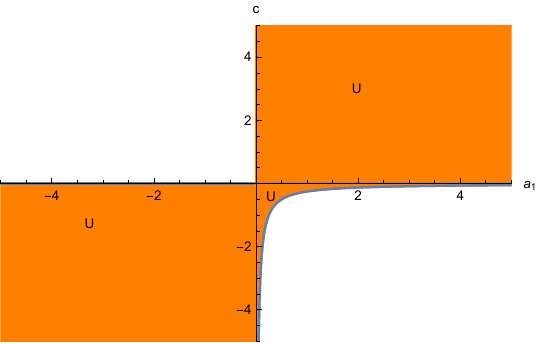}
	}\hspace{0.1cm}
	\subfloat[ When $b<0$ and $1/2<\alpha<1$]{%
		\includegraphics[scale=0.7]{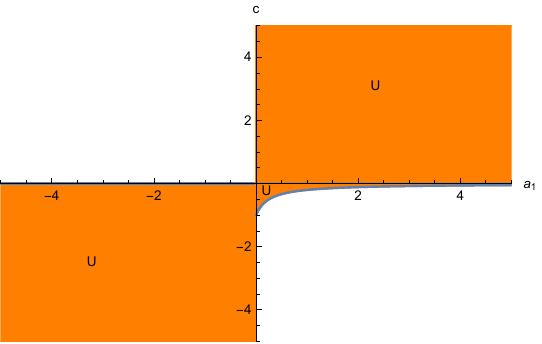}
	 }	
\caption{}
	\label{fig4.2}
\end{figure}

%\begin{figure}
  %  \centering
 %   \includegraphics[scale=0.9]{gh11.pdf}
  %  \caption{When $b>0$ and $0<\alpha<1/2$}
  %  \label{drcfig1}
%\end{figure}
%\begin{figure}
 %   \centering
 %   \includegraphics[scale=0.9]{gh12.pdf}
  %  \caption{ When $b<0$ and $0<\alpha<1/2$}
 %   \label{fig1}
%\end{figure}
%\begin{figure}
    %\centering
 %   \includegraphics[scale=0.9]{gh13.pdf}
 %   \caption{ When $b>0$ and $1/2<\alpha<1$}
 %   \label{drcfig2}
  %  \end{figure}
  %  \begin{figure}
 %   \centering
 %   \includegraphics[scale=0.9]{gh14.pdf}
 %   \caption{ When $b<0$ and $1/2<\alpha<1$}
 %   \label{fig2}
  %  \end{figure}
\section{Complex root $\lambda$ with positive real part}\label{sec3.6}
Now, we consider $\lambda=u+iv\in \mathbb{C}$, with $u>0$. Note that, the boundary of the set $\widehat{Q}=\{Q(\lambda,\tau)\quad|\quad\lambda=u+i v, u>0,\tau\geq0\}$ is the unit circle $x^2+y^2=1$ in $\mathbb{C}$.\\
If $|P(\lambda)|>1$, $\forall \lambda=u+i v$, $u>0$ then the image set $\mathcal{P}=\{P(\lambda)|\lambda=u+iv\in\mathbb{C}, u>0\}$ will not intersect the set $\widehat{Q}$. Therefore, the characteristic equation will not have any root $\lambda$ with positive real part.\\
Therefore, the bifurcation curve in the parameter plane is obtained when the $\mathcal{P}$ touches the $\widehat{Q}$, (where $\partial\mathcal{P}$ and $\partial\widehat{Q}$ are boundaries of the sets $\mathcal{P}$ and $\widehat{Q}$ respectively.)\\
In this case, the sets $\mathcal{P}$ and $\widehat{Q}$ have a common tangent in $\mathbb{C}$. 
\subsection{Method to find the common tangent}\label{section to find the common tangent}
Consider $P(\lambda)$ given by equation \eqref{expression for p}. For $\lambda=u+iv\in \mathbb{C}$, ${\lim_{u\to\infty}}Re(P(u+iv))$ goes to $\infty$ or $-\infty$ depending on the sign of $b$ and $c$. If $b$ and $c$ are of same sign then ${\lim_{u\to\infty}}Re(P(u+iv))\to \infty$ and the region of $\mathcal{P}$ is unbounded at the right end and bounded at the left as shown in Figure \eqref{complex root when b negative c negative} whereas, if $b$ and $c$ are of opposite sign then ${\lim_{u\to\infty}}Re(P(u+iv))\to -\infty$ and $\mathcal{P}$ is unbounded at the left and bounded at the right(cf. Figure(\ref{p(iv) bounded from right and unbounded from left})). 

The boundary of the set $\mathcal{P}$ is at $u=0$ i.e. $\partial \mathcal{P}=\{P(i v)|v\in\mathbb{R}\}$.\\
Similarly, $\partial \widehat{Q}=\{Q(iv\tau)|v\in\mathbb{R},\tau>0\}$. The unit circle $x^2+y^2=1$ is the boundary of $\widehat{Q}$.
Since, $\partial\mathcal{P}$ is symmetric about x-axis, we may take $v\geq0$. Moreover, note that the initial point of $\partial\mathcal{P}$ i.e. $P(0)=\frac{-a}{b}$. depends on $a$ and $b$ only. Therefore, the region $\mathcal{P}$ moves in horizontal direction if we change the parameter values $a$ and $b$ as shown in Figures \eqref{for the different sign of b and c behaviour of a} and \eqref{for the same sign of b and c behaviour of a}. \\
 \begin{figure}
     \centering
     \begin{subfigure}[b]{0.49\textwidth}
         \centering
         \includegraphics[width=\textwidth]{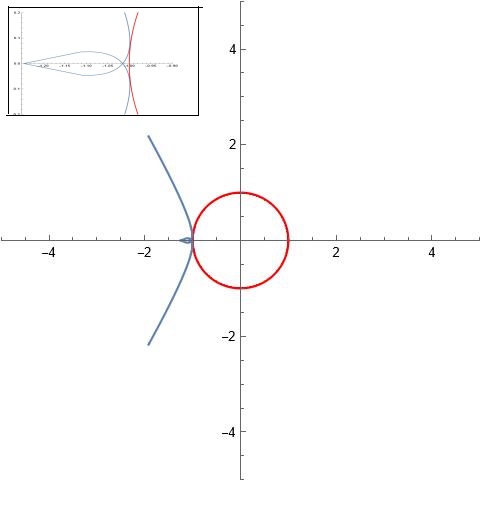}
         \caption{$\partial\mathcal{P}$ touches $\widehat{Q}$ from the left side}
         \label{for the same sign of b and c decrease a}
     \end{subfigure}
     \hfill
     \begin{subfigure}[b]{0.49\textwidth}
         \centering
         \includegraphics[width=\textwidth]{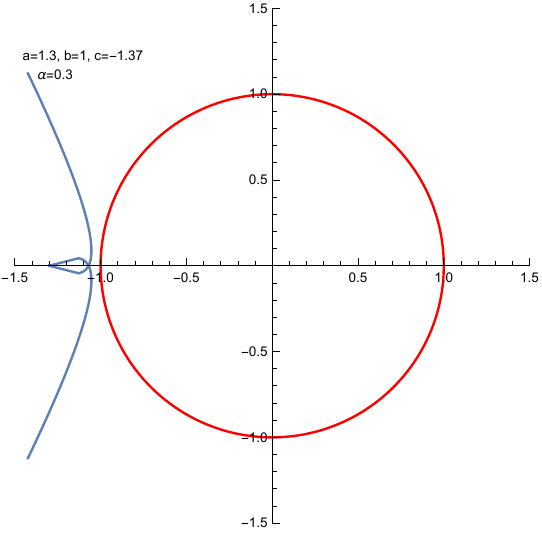}
         \caption{$\partial\mathcal{P}$ goes away from $\widehat{Q}$ as we increase $a$}
   % \label{for the same sign of b and c behaviour of a}
     \end{subfigure}
     \caption{Region $\mathcal{P}$ is on the left side of $\partial\mathcal{P}$}
     \label{for the different sign of b and c behaviour of a}
\end{figure} 
 \begin{figure}
     \centering
     \begin{subfigure}[b]{0.49\textwidth}
         \centering
         \includegraphics[width=\textwidth]{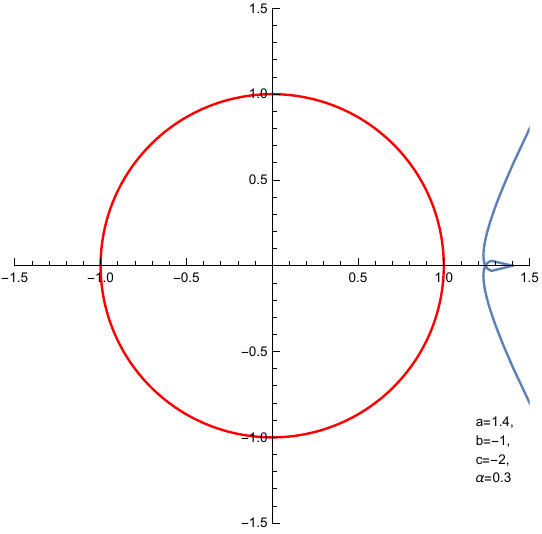}
         \caption{If we increase $a$, $\partial\mathcal{P}$ goes away $\widehat{Q}$ }
       %  \label{fig:y equals x}
     \end{subfigure}
     \hfill
     \begin{subfigure}[b]{0.49\textwidth}
         \centering
         \includegraphics[width=\textwidth]{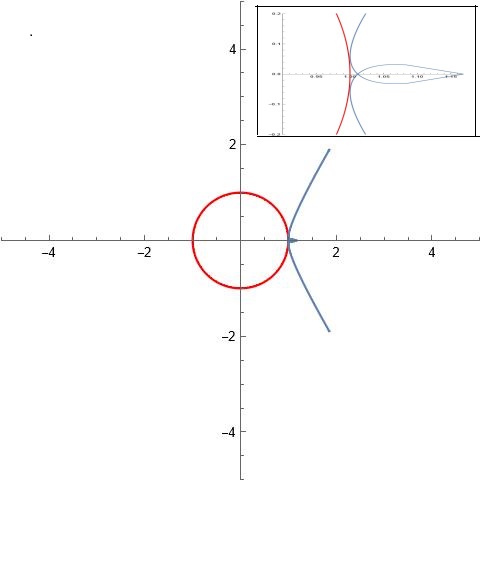}
         \caption{$\partial\mathcal{P}$ touches $\widehat{Q}$ from the right side}
        \label{for the same sign of b and c increase a}
     \end{subfigure}
     \caption{Region $\mathcal{P}$ is on the right side of $\partial\mathcal{P}$}
      \label{for the same sign of b and c behaviour of a}
\end{figure} 
If $P(iv)=x+i y$ then by separating the real and imaginary part  we get
\begin{align}
x=\dfrac{v^\alpha\cos(\frac{\alpha \pi}{2})+c v^{2\alpha}\cos(\alpha \pi)-a}{b}\quad\text{and}\label{eq3.5.2.2}
\end{align}
\begin{align}
y=\dfrac{v^\alpha\sin(\frac{\alpha \pi}{2})+c v^{2\alpha}\sin(\alpha \pi)}{b}.\label{eq3.5.2.1}
\end{align}
 Since $\partial \mathcal{P}$ and $\partial \widehat{Q}$ (unit circle) have a point say $(x,y)$ in common, it should satisfy the equations \eqref{eq3.5.2.2}, \eqref{eq3.5.2.1} and the equation of the unit circle
\begin{equation}\label{circleequation}
x^2+y^2=1.
\end{equation} 
In this case, at the point $(x,y)$, the slopes of the tangents to $\partial \mathcal{P}$ and $\partial\widehat{Q}$ should match.\\
We find the expressions for the tangent curve in different cases, by eliminating $v$ between equations \eqref{eq3.5.2.2} and \eqref{eq3.5.2.1} and equating the slopes.
\subsection{Bifurcation analysis for various values of $b$ and $\alpha$ }
In this section, we find the bifurcation curves separating the delay independent stable/ unstable region with the delay dependent stable/ unstable region. We consider four cases viz. $b>0$ and $0<\alpha<1/2$, $b<0$, $0<\alpha<1/2$, $b>0$ and $1/2<\alpha<1$ and $b<0$ and $1/2<\alpha<1$. In the section \eqref{section3.8}, we provide bifurcation curves separating different behaviors in the delay dependent stable/ unstable regions. 
\subsubsection{ When $b>0$ and $\alpha$ between $0$ to $1/2$ }\label{when b>0 and 0<alpha<1/2}
 
  Solving \eqref{eq3.5.2.1}, we get $v^\alpha=\dfrac{\csc(\alpha \pi)\Big(-\sin(\frac{\alpha\pi}{2})\pm\sqrt{\sin^2(\frac{\alpha\pi}{2})+4 b c y\sin(\alpha\pi)}\Big)}{2 c}.$
 Note that, the stability analysis of system \eqref{eq3.2} in the I, II and III  quadrants of $a_1c$-plane for $b>0$ and $0<\alpha<1/2$ is provided in Theorems \eqref{th3.1}, \eqref{drcthm3.1}(ii) and \eqref{drcthm1}. The only remaining part is fourth quadrant i.e. $a_1>0$ and $c<0$.\\  

So, for $c$ negative 
\begin{equation}\label{value of v for c<0}
v^\alpha=\dfrac{\csc(\alpha \pi)\Big(-\sin(\frac{\alpha\pi}{2})-\sqrt{\sin^2(\frac{\alpha\pi}{2})+4 b c y\sin(\alpha\pi)}\Big)}{2 c}.
\end{equation}
 
 Putting this value of $v^\alpha$ in equation \eqref{eq3.5.2.2}, we get the equation of $\partial\mathcal{P}$ as
 \begin{equation}\label{curveequation}
 \begin{split}
\dfrac{1+4 a c +4 b c x+\csc(\frac{\alpha\pi}{2})\sqrt{\sin^2(\frac{\alpha\pi}{2})+4 b c y \sin(\alpha\pi)}}{4c}-\\
 \dfrac{\cot(\alpha\pi)\csc(\alpha\pi)\big(\sin(\frac{\alpha\pi}{2})+\sqrt{\sin^2(\frac{\alpha\pi}{2})+4 b c y \sin(\alpha\pi)}\big)^2}{4 c}=0.
\end{split}
\end{equation}  
 Equating the slopes of $\partial \widehat{Q}$  and $\partial\mathcal{P}$ using \eqref{circleequation} and \eqref{curveequation}, we get the equation of their common tangent as  
\begin{equation}\label{tangenteq}
\frac{x}{y}+\frac{\csc(\frac{\alpha\pi}{2})\sqrt{\sin^2(\frac{\alpha\pi}{2})+4 b c y \sin(\alpha\pi)}}{-\cot(\frac{\alpha\pi}{2})+\cot(\alpha\pi)\Big(1+\csc(\frac{\alpha\pi}{2})\sqrt{\sin^2(\frac{\alpha\pi}{2})+4 b c y \sin(\alpha\pi)}\Big)}=0.
\end{equation}
Furthermore, we want the tangency conditions in terms of parameters $a_1$ and $c$. So, we eliminate $x$ and $y$ from the equations \eqref{circleequation}, \eqref{curveequation} and \eqref{tangenteq} to find the bifurcation curve $\Gamma_5$ in the $a_1c$-plane. On the right side of $\Gamma_5$, there will not be any intersection between $\mathcal{P}$ and $\widehat{Q}$ (cf. Figure (\ref{stable for all tau})) and the system \eqref{eq3.2} is stable, $\forall \tau\geq0$. The region bounded by the curves $\Gamma_2$ and $\Gamma_5$ (see Figure \eqref{mainfig4.1} ), is stable for $\tau=0$. We observed that, there exists a positive root of $v$ by the equation \eqref{value of valpha} in this region. Corresponding to that positive root of $v$ we have critical values of $\tau$ given by equations \eqref{critical value of tau with plus sign} and \eqref{critical value of tau with minus sign}.Note that at the smallest critical value of $\tau_*$ $Re(\frac{d\lambda}{d\tau}|_u=0)>0$. If $Re(\frac{d\lambda}{d\tau}|_u=0)<0$ at the smallest critical value $\tau_*$ then we must have some characteristic root which is moving from right to left half plane at $\tau_*$. Since $\nexists$ any characteristic root in the right half plane it cannot possible. This shows that, the system \eqref{eq3.2} is unstable for some $\tau\geq0$. There are two behaviors viz SSR and SS. The details will be given in Section \eqref{section3.8}.\\
First, we solve \eqref{curveequation} for $y$. For $c<0$, the expression of y is given in the data set-1 accompanying this article. All the data sets in this paper are also available at \url{https://drive.google.com/drive/folders/147KhyNARmYlQhIt5caLDqGK9GHqvz2Yb?usp=sharing}.\\
Further, we put this y in equation \eqref{circleequation} and solve it for $x$. At the end,  put these values of $x$ and $y$ in the tangent equation \eqref{tangenteq} to obtain the required bifurcation curve $\Gamma_5$ in terms of $a_1$, $b$, $c$ and $\alpha$ for $b>0$ and $0<\alpha<1/2$ is given in the data set 1. \\
In the Note \eqref{note1} after Theorem \eqref{th3.2}, we obtained a curve $\Gamma_4$ in the fourth quadrant of $a_1c$-plane such that the system \eqref{eq3.2} is stable on the right side of $\Gamma_4$. Now, we show that, the curve $\Gamma_5$ improves that stable region i.e. the exact value of bifurcation occurs at $\Gamma_5$ which is on the left of $\Gamma_4$. We observed that the curve $\Gamma_5$ is always on the left side of $\Gamma_4$. Therefore, we can ignore $\Gamma_4$ and consider $\Gamma_5$ as the bifurcation curve in the fourth quadrant, in this case. 
\begin{figure}
     \centering
     \begin{subfigure}[b]{0.45\textwidth}
         \centering
         \includegraphics[width=\textwidth]{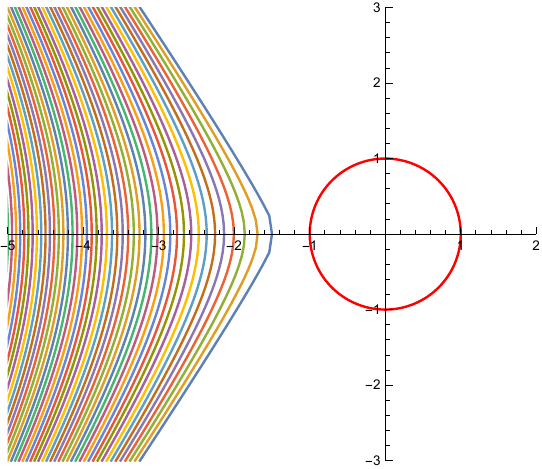}
         \caption{System \eqref{eq3.2} is stable for all $\tau\geq0$}
        \label{stable for all tau}
     \end{subfigure}
     \hfill
     \begin{subfigure}[b]{0.45\textwidth}
         \centering
         \includegraphics[width=\textwidth]{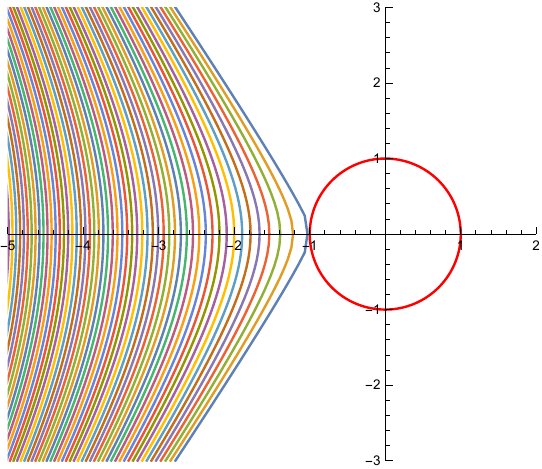}
         \caption{System is stable for all $\tau$ upto the tangent equation}
         \label{stable upto this point}
     \end{subfigure}
     \hfill
     \begin{subfigure}[b]{0.45\textwidth}
         \centering
         \includegraphics[width=\textwidth]{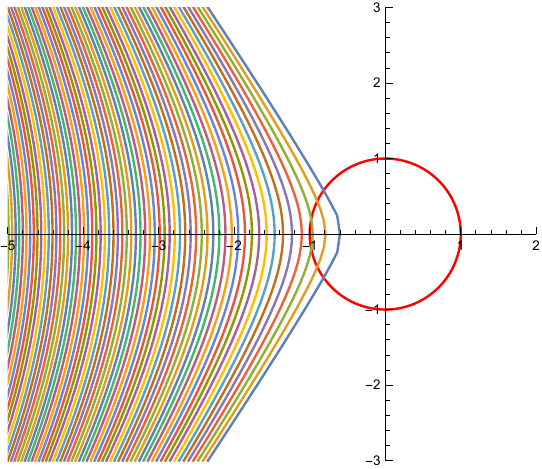}
         \caption{Stability depends on delay parameter}
         \label{stablity depend on delay}
     \end{subfigure}
     \hfill
      \caption{ Different behavior of $P(\lambda)$ when $\lambda$ is a complex root and $b$ and $c$ are of opposite sign}
       \label{p(iv) bounded from right and unbounded from left}
\end{figure} 
 
 \subsubsection {We consider $b<0$ and $0<\alpha<1/2$.}\label{positive complex root for b negative and alpha between 0 to half}
     \begin{itemize}
 \item[\textbf{Case 1:}] $c>0$
  \end{itemize}  
  
  The Theorem \eqref{th3.1} in Section \eqref{sec3.4} and Theorem \eqref{drcthm3.1}(a) in Section \eqref{section of positive real root} provide the stability analysis except for $2b<a_1<0$ or $b<a<-b$. We discuss this region in following subcases. 
  \begin{itemize}
 \item[\textbf{Subcase 1.1:}]  $b<a<0$
  \end{itemize} 
 \textbf{Subcase: When $b<a<0$}
  \begin{equation}\label{dlambda by dtau}
  Re\Big[\dfrac{1}{\alpha}\dfrac{d\lambda}{d\tau}\Big|_{u=0}\Big]=v^{2\alpha}+(2)c^2v^{4\alpha}-a v^\alpha\cos(\frac{\alpha\pi}{2})-2a c v^{2\alpha}\cos(\alpha\pi)+3 c v^{3\alpha}\cos(\frac{\alpha\pi}{2})
 \end{equation}
 
The sign of $\dfrac{1}{\alpha}\dfrac{d\lambda}{d\tau}$ \cite{bhalekar2019analysis} is positive for $a<0$, $c>0$ and $0<\alpha<1/2$.\\
\begin{itemize}
 \item[\textbf{Subcase 1.2:}]  $0<a<-b$
  \end{itemize} 

 We have from equation (16) given in \cite{bhalekar2019analysis}
 \begin{equation} \label{rewrite}
 -2 a c v^{2\alpha}\cos(\alpha\pi)-a v^{\alpha}\cos(\frac{\alpha\pi}{2})=-v^{2\alpha}-c^{2}v^{4\alpha}-2 c v^{3\alpha}\cos(\frac{\alpha\pi}{2})+av^{\alpha}\cos(\frac{\alpha\pi}{2})+b^2-a^2
 \end{equation}
 Putting \eqref{rewrite} in \eqref{dlambda by dtau} we get 
 \begin{equation*}
 Re\Big[\dfrac{1}{\alpha}\dfrac{d\lambda}{d\tau}\Big|_{u=0}\Big]=c^2v^{4\alpha}+c v^{3\alpha}\cos(\frac{\alpha\pi}{2})+a v^\alpha\cos(\frac{\alpha\pi}{2})+b^2-a^2>0
 \end{equation*}
Therefore, $Re\Big[\dfrac{1}{\alpha}\dfrac{d\lambda}{d\tau}\Big|_{u=0}\Big]>0$ in both cases. Hence, by Theorem 3.1 given in \cite{bhalekar2019analysis} system \eqref{eq3.2} undergoes Hopf bifurcation at
\begin{equation}\label{smallest critical value of delay}
 \tau_*=\frac{\arccos(\frac{v^{\alpha}\cos(\frac{\alpha\pi}{2})+c v^{2\alpha}\cos(\alpha\pi)-a}{b})}{v}
 \end{equation}
 such that the system \eqref{eq3.2} is stable for all $0<\tau<\tau_*$ and unstable for $\tau>\tau_*$ in the region $2b<a_1<0$ when $b<0$, $c>0$ and $0<\alpha<1/2$ (cf. Figure \eqref{mainfig4.2}).
 \begin{itemize}
 \item[\textbf{Case 2:}] $c<0$
 \end{itemize}
 Theorem \eqref{drcthm3.1}(c) shows that the 3rd quadrant is unstable $\forall \tau\geq0$.\\
 Let us consider the fourth quadrant in the $a_1c$-plane. Since, $b<0$ and $c<0$, we have ${\lim_{\lambda\to\infty}}Re(P(\lambda))\to \infty$ i.e. the region bounded by $P(i v)$ is on the right side of $\partial\mathcal{P}$ as shown in Figure (\ref{complex root when b negative c negative}). As in Section \eqref{section to find the common tangent}, we have a bifurcation curve $\Gamma_6$ which is obtained by using the condition for the curves $\partial\mathcal{P}$ and $\partial\widehat{Q}$ to have a common tangent.\\
 The expression for $\Gamma_6$ is given in the data set 2 accompanying this paper.\\
 If we take the parameter values $(a_1,c)$ on the right of $\Gamma_6$ then $\mathcal{P}$ does not intersect $\widehat{Q}$ and the system \eqref{eq3.2} is stable for all $\tau\geq0$. We observe that, the bifurcation curve $\Gamma_3$ obtained in Section \eqref{sec3.4} is always on the right of $\Gamma_6$ (as shown in Figure \eqref{mainfig4.1}). We already have shown that the system \eqref{eq3.2} is stable $\forall\tau\geq0$ on the right of $\Gamma_3$. Therefore, $\Gamma_6$ provides a better estimate for the bifurcation curve and we can ignore $\Gamma_3$.
  \begin{figure}
     \centering
     \begin{subfigure}[b]{0.38\textwidth}
         \centering
         \includegraphics[width=\textwidth]{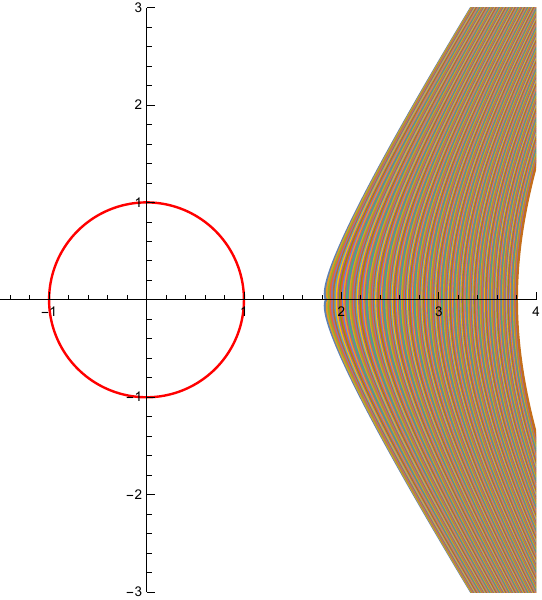}
         \caption{System \eqref{eq3.2} is stable for all $\tau\geq0$}
         \label{stable for all tau when b and c is negative}
     \end{subfigure}
     \hfill
     \begin{subfigure}[b]{0.37\textwidth}
         \centering
         \includegraphics[width=\textwidth]{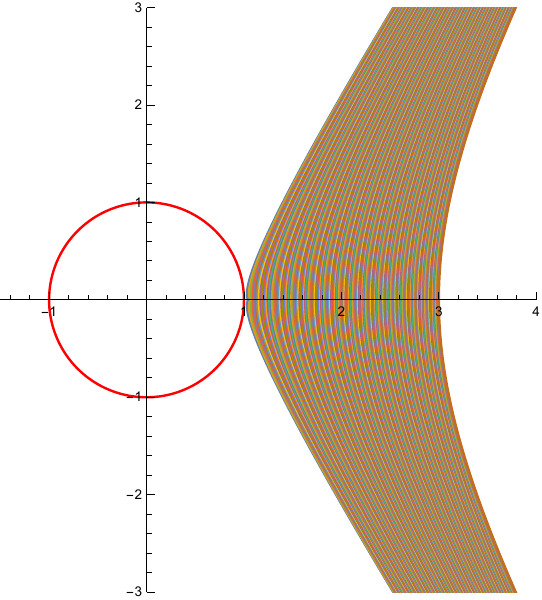}
         \caption{System \eqref{eq3.2} is stable for all $\tau\geq0$ upto the tangent equation}
       %  \label{fig:three sin x}
     \end{subfigure}
  \caption{ Different behavior of $P(\lambda)$ when $\lambda$ is a complex root and $b$ and $c$ are of same sign}
       \label{complex root when b negative c negative}
\end{figure} 
\subsubsection{We consider $b>0$ and $1/2<\alpha<1$.}\label{section for instability}
\begin{itemize}
\item [\textbf{Case 1:}] $c>0$
\end{itemize}
For $b>0$ and $c>0$ the region of $\mathcal{P}$ will be on the right side of $\partial\mathcal{P}$ as ${\lim_{u\to\infty}}Re(P(u+iv))\to \infty$ and if we fix $b$, $\alpha$, $c$ and change the parameter $a$, the region of $\mathcal{P}$ gets translated from right side to the left side in the $a_1c$-plane (cf. Figure (\ref{complex root when b negative c negative})). So, here also we have to find the tangent equation as given in section \eqref{section to find the common tangent}. So, note that for $c>0$, we get 
\begin{equation}\label{expression for v when c is positive}
v^\alpha=\dfrac{\csc(\alpha \pi)\Big(-\sin(\frac{\alpha\pi}{2})+\sqrt{\sin^2(\frac{\alpha\pi}{2})+4 b c y\sin(\alpha\pi)}\Big)}{2 c}.
\end{equation}
 By putting $v^\alpha$ in equation \eqref{eq3.5.2.2}, we get the expression for $\partial\mathcal{P}$
 \begin{equation}\label{curveequation1}
 \begin{split}
\dfrac{1+4 a c +4 b c x-\csc(\frac{\alpha\pi}{2})\sqrt{\sin^2(\frac{\alpha\pi}{2})+4 b c y \sin(\alpha\pi)}}{4c}-\\
 \dfrac{\cot(\alpha\pi)\csc(\alpha\pi)\big(\sin(\frac{\alpha\pi}{2})+\sqrt{\sin^2(\frac{\alpha\pi}{2})+4 b c y \sin(\alpha\pi)}\big)^2}{4 c}=0.
\end{split} 
 \end{equation}
  If we equate slope of circle \eqref{circleequation} and slope of $\partial\mathcal{P}$ \eqref{curveequation1}, we get the expression for common tangent as 
 \begin{equation}\label{tangenteq1}
\frac{x}{y}+\frac{\csc(\frac{\alpha\pi}{2})\sqrt{\sin^2(\frac{\alpha\pi}{2})+4 b c y \sin(\alpha\pi)}}{\cot(\frac{\alpha\pi}{2})+\cot(\alpha\pi)\Big(-1+\csc(\frac{\alpha\pi}{2})\sqrt{\sin^2(\frac{\alpha\pi}{2})+4 b c y \sin(\alpha\pi)}\Big)}=0.
 \end{equation}
 Using equations \eqref{curveequation1}, \eqref{tangenteq1} and \eqref{circleequation}, we can eliminate $x$ and $y$ as given in the section \eqref{section to find the common tangent}. We get the bifurcation curve $\Gamma_7$ given in data set 3 in the $a_1c$-plane. On the left side of $\Gamma_7$, we don't have any intersection between $\mathcal{P}$ and $\widehat{Q}$ (as shown in Figure (\ref{stable for all tau when b and c is negative})) and system \eqref{eq3.2} is stable $\forall \tau\geq0$. Note that, the bifurcation curve $\Gamma_7$ intersects the $c-$axis in the $a_1c$ plane at some value $c=c_{2}$. The expression for $c_{2}$ in terms of $b>0$ and $1/2<\alpha<1$ is given in the data set 4 which is given in this article.\\
 The curve $\Gamma_7$ provides a bifurcation in second quadrant of $a_1c$-plane in this case.\\
 Note that, Theorem \eqref{drcthm3.1}(b) shows that the first quadrant of $a_1c$-plane in this case is unstable region $\forall\tau\geq0$. 
 \begin{itemize}
 \item[\textbf{Case 2:}] $c<0$
 \end{itemize}
 Note that for $c<0$, the region of $\mathcal{P}$ is on the left side of the boundary $\partial\mathcal{P}$ as in Figure (\ref{stable for all tau}). For $c<0$, the value of $v^\alpha$ is given by the equation \eqref{value of v for c<0} and the common tangent equation is given by \eqref{tangenteq}. As in Section \eqref{section to find the common tangent}, we eliminate $x$ and $y$ from \eqref{value of v for c<0}, \eqref{tangenteq} and \eqref{circleequation} to get the bifurcation curve. Let us call the equation of the bifurcation curve for $b>0$, $c<0$ and $1/2<\alpha<1$ as $\Gamma_8$ (see data set 5). On the right side of $\Gamma_8$, we don't have any intersection between $\mathcal{P}$ and $\widehat{Q}$ and the system \eqref{eq3.2} is stable for all $\tau\geq0$.\\
 By Theorem \eqref{drcthm3.1}(c) in Section \eqref{section of positive real root} the $3rd $ quadrant of $a_1c$-plane is unstable $\forall$ $\tau\geq0$, for this case.
\subsubsection{Consider $b<0$ and $1/2<\alpha<1$.}
 
\begin{itemize}
\item [\textbf{Case 1:}] $c>0$.
\end{itemize} 
 If $b<0$, $c>0$ and $1/2<\alpha<1$, then $\mathcal{P}$ is bounded from right and unbounded from left as ${\lim_{u\to\infty}}Re(P(u+iv))\to-\infty$ (cf. Figure \eqref{p(iv) bounded from right and unbounded from left}). If we fixed $b$, $\alpha$, $c$ and increase $a$ the region of $\mathcal{P}$ goes away from $\widehat{Q}$ (cf. Figure \eqref{stable for all tau}). If we decrease $a$ then we observed that the boundary $\partial\mathcal{P}$ and boundary $\widehat{Q}$ intersect as shown in Figure \eqref{fig when a1<2b<0 and c>0} for a fix $b$, $c$ and $\alpha$. It means both $\partial \mathcal{P}$ and $\partial\widehat{Q}$ have a common tangent. For this, we have to find the expression for $\partial\mathcal{P}$. So, for $c>0$ we have $v^\alpha$ is equation \eqref{expression for v when c is positive} putting this $v^\alpha$ in equation \eqref{eq3.5.2.2} we get $\partial\mathcal{P}$ as equation \eqref{curveequation1}. Therefore, the expression for common tangent is obtained by equation \eqref{tangenteq1}. So, as in Section \eqref{section to find the common tangent} we find a bifurcation curve $\Gamma_9$ (see data set 6) such that on the left side of $\Gamma_9$ we don't have any intersection between $\mathcal{P}$ and $\widehat{Q}$ and system \eqref{eq3.2} is stable $\forall$ $\tau\geq0$. Note that this bifurcation curve $\Gamma_9$ intersect with line $a_1=2b$ at $c_7$. The expression for $c_7$ is given in data set 7 attached with this manuscript.\\
 
  \begin{figure}
     \centering
     \begin{subfigure}[b]{0.38\textwidth}
         \centering
         \includegraphics[width=\textwidth]{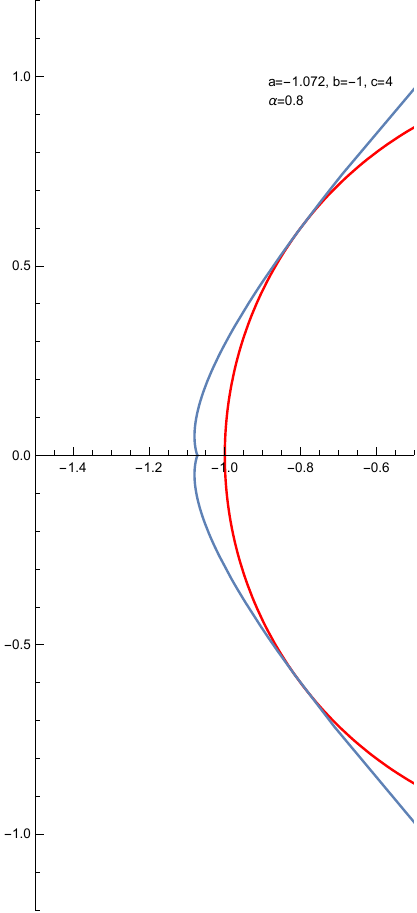}
         \caption{Bifurcation curve when $a_1<2b<0$}
        \label{fig when a1<2b<0 and c>0}
     \end{subfigure}
     \hfill
     \begin{subfigure}[b]{0.49\textwidth}
         \centering
         \includegraphics[width=\textwidth]{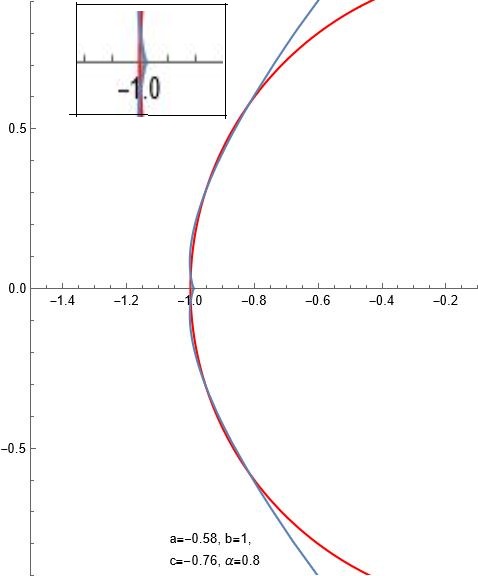}
         \caption{Bifurcation curve when $2b<a_1<0$}
         \label{fig when 2b<a1<0 and c>0}
     \end{subfigure}
 % \caption{ Different behavior of $P(\lambda)$ when $\lambda$ is a complex root and $b$ and $c$ are of same sign}
     %  \label{complex root when b negative c negative}
\end{figure}

\begin{itemize}
\item [\textbf{Case 2:}] $c<0$.\\
For $c<0$ and $b<0$ the region of $\mathcal{P}$ is unbounded from right and bounded from left (given as Figure \eqref{complex root when b negative c negative}). So, here also we have to find the condition such that the boundary $\partial\mathcal{P}$ touches to boundary of $\widehat{Q}$ i.e. unit circle. For that we need to equate the tangent of $\partial\mathcal{P}$ and $\partial\widehat{Q}$. Therefore, for $c<0$ the value of $v^\alpha$ is given by the equation \eqref{value of v for c<0}. The equation of $\partial\mathcal{P}$ and equation of common tangent are \eqref{curveequation} and \eqref{tangenteq} respectively. So, we obtain the bifurcation curve $\Gamma_{10}$ (see data set 8) as given in Section \eqref{section to find the common tangent} such that on the right side of $\Gamma_{10}$ we don't have any common points between $\mathcal{P}$ and $\widehat{Q}$ and system \eqref{eq3.2} is stable for all $\tau\geq 0$.
\end{itemize} 
Note that the first and third quadrant in the $a_1c$-plane is unstable for all $\tau\geq0$ by Theorems \eqref{drcthm3.1}(a) and \eqref{drcthm3.1}(c) given in Section \eqref{section of positive real root}.    
\section{Few more unstable regions} \label{sec3.7}
Based on our observations for sufficient number of parameter values, we conjecture that the following regions are unstable: \\
\begin{itemize}
\item The regions bounded by the curve $\Gamma_1$ and $\Gamma_2$ in the Figure \eqref{mainfig4.1} and \eqref{mainfig4.2}. 
\end{itemize}
It is open problem to prove these conjectures or to provide counter examples. If any of these regions is
not unstable then we can expect the instability switches.
\section{Delay dependent stability/ instability:}\label{section3.8}
\subsection{The regions SS and SSR}
We observed following types of bifurcations which depend on delay.\\
\textbf{Single stable region (SSR):} 
In this case, $\exists$ $\tau_*>0$ such that $0<\tau<\tau_*$ $\Longrightarrow$ the system is stable and $\tau>\tau_*$ $\Longrightarrow$ the unstable behavior (cf. Figure \eqref{single stable region for k=1}).\\
\textbf{Stability Switch (SS):} The stability switches are observed if $\exists$ $\tau_{0*}=0$ and positive constants $\tau_{1*}$, $\tau_{2*}$, $\tau_{3*},\ldots\tau_{k*}$ such that \\
$\tau_{2jk*}<\tau<\tau_{(2j+1)*}$  $\Longrightarrow$ stable, $j=0,1,\ldots,\frac{k-1}{2}$,\\
$\tau_{(2j+1)*}<\tau<\tau_{(2j+2)*}$  $\Longrightarrow$ unstable, $j=0,1,\ldots,\frac{k-3}{2}$
and \\
 $\tau>\tau_k$ $\Longrightarrow$ unstable.\\
Note that, $k$ is an odd number $\geq3$ e.g. if $k=5$, then we get the stability properties as shown in Figure \eqref{stability switches for k=5} for the delay value $\tau$.\\
\begin{figure}
    \centering
    \includegraphics{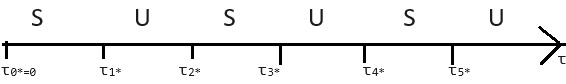}
    \caption{Stability switches for $k=5$ }
    \label{stability switches for k=5}
\end{figure}
If we allow $k=1$ then this reduces to the case SSR and the properties will be as in the Figure \eqref{single stable region for k=1}:\\
\begin{figure}
    \centering
    \includegraphics[scale=0.7]{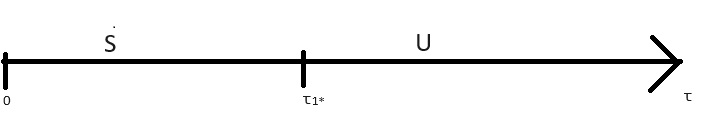}
    \caption{Single stable region for $k=1$ }
    \label{single stable region for k=1}
\end{figure}
We provide more details and the conditions for such switches in further sections.\\
In the previous Sections, we provided some delay independent stable/unstable regions. These are some regions in Figure(\eqref{mainfig4.1}-\eqref{mainfig4.4}) which were not discussed. Now, we show that the stability properties in these regions depend on the delay. The regions\\
\begin{itemize}
\item[(i)] bounded by $\Gamma_2$ and $\Gamma_5$ for $b>0$, $c<0$, $0<\alpha<1/2$ in Figure \eqref{mainfig4.1}
\item[(ii)] bounded by $a_1=2b$ and $2b<a_1<0$, for $c>0$, $b<0$, $0<\alpha<1/2$ in Figure \eqref{mainfig4.2}
\item[(iii)]bounded by $\Gamma_6$ and $\Gamma_2$ for $c<0$, $b<0$, $0<\alpha<1/2$ in Figure \eqref{mainfig4.3}
\item[(iv)] bounded by $\Gamma_7$ and the line $a_1=0$ for $b>0$, $c>0$, $1/2<\alpha<1$ in Figure \eqref{mainfig4.3}.
\item[(v)] bounded by $\Gamma_2$ and $\Gamma_8$ for $b>0$, $c<0$ and $1/2<\alpha<1$ in Figure \eqref{mainfig4.3}
\item[(vi)] bounded by $\Gamma_9$ and the line $a_1=0$ for $c>0$, $b<0$ and $1/2<\alpha<1$ in Figure \eqref{mainfig4.4}
\item[(vii)] bounded by $\Gamma_2$ and $\Gamma_{10}$ for $c<0$, $b<0$ and $1/2<\alpha<1$ in Figure \eqref{mainfig4.4} 
are subsets of the stable region at $\tau=0$ given in Figure \eqref{fig1}. \\
The system \eqref{eq3.2} is stable region at $\tau=0$ in all the above regions.\\

We observed that $\exists$ a real root of $v^\alpha$ of equation \eqref{value of valpha} in all these regions. Since, the system is stable at $\tau=0$, $\nexists$ any characteristic root in the right half complex plane. As $exists$ v, we can have critical values of $\tau$ defined by \eqref{critical value of tau with plus sign} and \eqref{critical value of tau with minus sign}. If $Re\Big(\frac{d\lambda}{d\tau}\Big|_{u=0}\Big)<0$ at the smallest critical values $\tau_*$ then we must have some characteristic root which is moving from right half plane to left half plane at $\tau_*$. Since, $\nexists$ any characteristic root in the right half this cannot be true. Therefore, we must have $Re\Big(\frac{d\lambda}{d\tau}\Big|_{u=0}\Big)>0$ at $\tau_*$.\\
This shows that $exists$ some $\tau_0>\tau_*$ such that system \eqref{eq3.2} becomes unstable at $\tau_0$.
\end{itemize}

\textbf{Instability switches (IS):}
In this case, we will have $\tau_{0*}=0$ and positive constants $\tau_{1*}$, $\tau_{2*},	\ldots,\tau_{k*}$ such that \\
$\tau_{2j*}<\tau<\tau_{(2j+1)*}\Longrightarrow $ Unstable, $j=0,1,\ldots,(k-2)/2$,\\
$\tau_{(2j+1)*}<\tau<\tau_{(2j+2)*}\Longrightarrow $ stable, $j=0,1,\ldots,(k-2)/2$\\
and $\tau>\tau_k\Longrightarrow$ Unstable. \\
E.g. if $k=6$ then the stability properties are as in Figure \eqref{instable region for k=6}.\\
\begin{figure}
    \centering
    \includegraphics[width=1.2\textwidth]{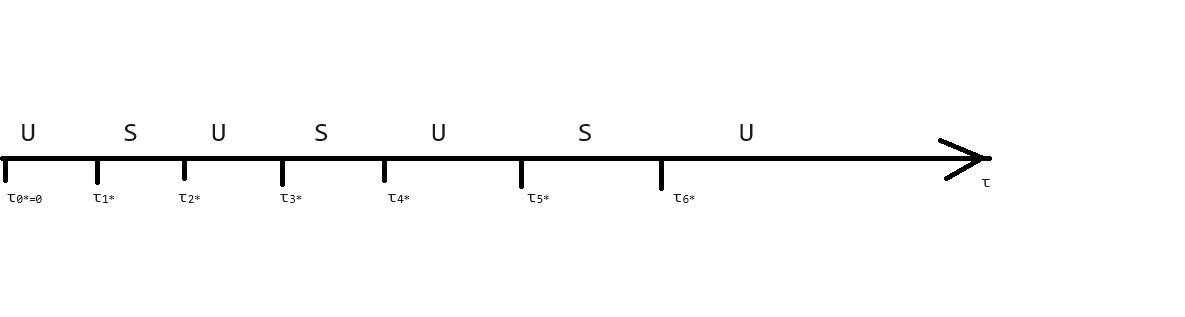}
    \caption{Instability switches for $k=6$ }
    \label{instable region for k=6}
\end{figure}
In both the cases SS and IS, the last unbounded region for $\tau$ gives instability. As we go arbitrary close to some bifurcation values, the number of switches increase.\\
 We consider four cases viz. $b>0$ and $0<\alpha<1/2$, $b<0$ and $0<\alpha<1/2$, $b>0$ and $1/2<\alpha<1$ and $b<0$ and $1/2<\alpha<1$. In each cases, we give the expressions for the bifurcation curves that separate SSR, SS, IS and the unstable region.
 
\subsection{We consider $b>0$ and $0<\alpha<1/2$.}\label{for b positive and alpha between 0 to half}

In the Sections \eqref{sec3.6} and \eqref{sec3.5}, we provided the stability analysis of all the points in the $a_1c$-plane except those in the region bounded by the curves $\Gamma_2$ and $\Gamma_5$ in the fourth quadrant. Solving equation \eqref{value of valpha} under these conditions on parameters $b$ and $\alpha$, we get two roots viz. $v_1$ and $v_2$ with $0<v_1<v_2$(see data set 9). These values will produce the following critical values of delay viz. $\tau_-(n)(v_1)$ and $\tau_+(n)(v_2)$ (see \eqref{critical value of tau with plus sign} and \eqref{critical value of tau with minus sign} for the expressions).

We observed that, $Re\big(\dfrac{d\lambda}{d\tau}|_{u=0}\big)$ is negative at $v_1$ and positive at $v_2$ for these values of parameters. Therefore, the characteristic roots move from left to right at $\tau=\tau_{+}(n)$ and from right to left at $\tau=\tau_{-}(n)$. Since, the system \eqref{eq3.2} is stable at $\tau=0$ in the fourth quadrant of $a_1c$-plane and $\tau_{+}(0)$ is smallest critical value, the system becomes unstable at $\tau=\tau_{+}(0)$.\\
Therefore, $0<\tau<\tau_{+}(0)\Longrightarrow$ stability.\\
This can be the only stable region (SSR case) or we can have stability switches (SS). We provide the bifurcation analysis below. \\
Let $\delta_1=\tau_{-}(n+1)-\tau_{-}(n)=\frac{2\pi}{v_1}$,\\
$\delta_2=\tau_{+}(n+1)-\tau_{+}(n)=\frac{2\pi}{v_2}$\\
and $\upmu=\tau_{-}(1)-\tau_{+}(0)$.\\
On the right side of the curve $\Gamma_2$ system \eqref{eq3.2} is stable at $\tau=0$ by the Section \eqref{sec3.1}. Note that $\tau_{+}(0)$ is the closest critical value to $0$ where $Re\left(\dfrac{d\lambda}{d\tau}|_{\tau_{+}(0)}\right)>0$.\\
If $\upmu<\delta_2$ then, $0<\tau_{+}(0)<\tau_{-}(1)<\tau_{+}(1)<\ldots<\tau_{+}(k)$ for some $k>0$ and we have stability switches (SS) upto $\tau_{+}(k)$. So, the switches will get disappear if $\exists$ smallest $k$ such that $\upmu_k=\tau_-(k+1)-\tau_+(k)>\delta_2$. Moreover, system become unstable for all $\tau>\tau_+(k)$.\\
If $\upmu>\delta_2$ then, $0<\tau_{+}(0)<\tau_{+}(1)<\tau_{-}(1)\ldots$ and we have only single stable region (SSR) $[0,\tau_{+}(0))$.
Therefore, the equation $\upmu=\delta_2$ is the bifurcation curve $\Gamma_{11}$ in the $a_1c$-plane. On the right side of $\Gamma_{11}$ we have SS and on the left we have SSR.\\ 
 For $b>0$ and $0<\alpha<1/2$ the expression of curve $\Gamma_{11}$ is attached with this manuscript in the data set 10.\\
Using all this analysis, we are able to provide the stability region of system \eqref{eq3.2} in the $a_1c$-plane for $b>0$ and $0<\alpha<1/2$ (cf. Figure \eqref{mainfig4.1}).\\

\subsection{We consider $b<0$ and $0<\alpha<1/2$.}
For $b<0$ and $0<\alpha<1/2$, we provided the stability results in Section \eqref{positive complex root for b negative and alpha between 0 to half} and Theorem \eqref{drcthm3.1}(d) (Section \eqref{section of positive real root}) except the region bounded by the curve $\Gamma_0$, vertical axis with $c<0$ and $\Gamma_7$ in $a_1c$ plane. We now propose the stability properties of system \eqref{eq3.2} in this remaining part. If we solve equation \eqref{value of valpha} with $b<0$ and $0<\alpha<1/2$, we get two values of $v^\alpha$ viz. $v_1$ and $v_2$ with $0<v_1<v_2$ (see data set 9). The critical values of $\tau$ corresponding to $v_1$ and $v_2$ are $\tau_-(n)(v_1)$ and $\tau_-(n)(v_2)$. \\
We proceed as in Section \eqref{for b positive and alpha between 0 to half}. If $\delta_1=\tau_{-}(n+1)(v_1)-\tau_{-}(n)(v_1)=\frac{2\pi}{v_1}$,\\
$\delta_2=\tau_{-}(n+1)(v_2)-\tau_{-}(n)(v_2)=\frac{2\pi}{v_2}$\\
and $\upmu=\tau_{-}(1)(v_2)-\tau_{-}(1)(v_1)$\\
then $\upmu=\delta_2$ is the bifurcation curve in the $a_1c$-plane if system \eqref{eq3.2} is stable at $\tau=0$ i.e. the right side of the curve $\Gamma_2$. Let us call the bifurcation curve $\upmu=\delta_2$ as $\Gamma_{12}$ whose expression in the $a_1c$-plane for $b<0$ and $0<\alpha<1/2$ is given in data-set 11. Moreover, the curve $\upmu=\delta_2$ intersects the curve $\Gamma_2$ at $c_0$ (Figure \eqref{mainfig4.2}). The expression for $c_0$ is given in data set 12. On the right side of the curve $\Gamma_{12}$, $\upmu<\delta_2$. So, we have $0<\tau_{-}(1)(v_2)<\tau_{-}(1)(v_1)<\ldots\tau_{-}(k)(v_2)$ and we have SS upto $\tau_{-}(k)(v_2)$. Note that we expect SS only upto $\tau_{-}(k)(v_2)$ where $k$ is the smallest positive number such that $\upmu_k=\tau_{-}(k)(v_1)-\tau_{-}(k)(v_2)>\delta_2$. System is unstable for all $\tau>\tau_{-}(k)(v_2)$. On the left side of the curve $\upmu=\delta_2$ we have $0<\tau_{-}(1)(v_2)<\tau_{-}(2)(v_2)<\ldots$ and it will give SSR from $[0,\tau_{-}(1)(v_2))$.\\ 
\subsubsection{Condition for the Instability switch (IS):} 
Note that region bounded by $\Gamma_0$ and $\Gamma_2$ is still remaining for $b<0$ and $0<\alpha<1/2$. By the Section \eqref{sec3.1} and Figure \eqref{fig1}, the region on the left side of the curve $\Gamma_2$ is unstable at $\tau=0$. If we solve equation \eqref{value of valpha} with $b<0$ and $0<\alpha<1/2$ we get $v_1$ and $v_2$ as two roots of $v^\alpha$ with $0<v_1<v_2$. In this region, the critical values of delay are $\tau_+(n)(v_1)$ and $\tau_-(n)(v_2)$.\\

We notice that $Re\big(\dfrac{d\lambda}{d\tau}|_{u=0}\big)$ is negative at $v_1$ and positive at $v_2$. It means that the roots of characteristic equation shift from right half plane to left half plane at $\tau=\tau_+$ and from left to right at $\tau=\tau_-$. Since, the system \eqref{eq3.2} is unstable at $\tau=0$ and if $\tau_-(0)$ is the smallest critical value then multiplicity of positive root in the right half plane is increased by two. On the other hand if $\tau_+(0)$ is the smallest critical value then the multiplicity of positive root in the right half plane is decreased by two. The former will provide IS and the later will give unstable solutions $\forall\tau>0$. Therefore, the curve bifurcating the unstable region with the IS region is provided by $\tau_{-}(0)=\tau_{+}(0)$ i.e.
\begin{equation*}
\frac{\arccos(\frac{v_1\cos(\frac{\alpha\pi}{2})+c v_1\cos(\alpha\pi)-a}{b})}{v_1^{1/\alpha}}-
\frac{2\pi-\arccos(\frac{v_2\cos(\frac{\alpha\pi}{2})+c v_2\cos(\alpha\pi)-a}{b})}{v_2^{1/\alpha}}=0.
\end{equation*}
Let us call this curve as $\Gamma_{13}$ (see data set 13). On the right side of $\Gamma_{13}$ we have $\tau_{+}(0)<\tau_{-}(1)$. So, the smallest critical delay is $\tau_+(0)$. By the Section \eqref{sec3.1}, we have only two roots of the characteristic equation at $\tau=0$ in the right half plane and those two roots shift from right half plane to left half plane as we increase $\tau$ from $\tau_+(0)$. Since, we don't have any root on the right half plane, the system becomes stable. As we further increase $\tau>\tau_-(1)$ the characteristic roots are shifted from left half plane to the right half plane and the system will become unstable. So, we have instability switches if $\tau_+(0)$ is the smallest critical value.\\
Let $\delta_1=\tau_{+}(n+1)-\tau_{+}(n)=\frac{2\pi}{v_1}$,\\
$\delta_2=\tau_{-}(n+1)-\tau_{-}(n)=\frac{2\pi}{v_2}$\\
and $\upmu=\tau_{-}(1)-\tau_{+}(0)$. 
\begin{itemize}
\item[(a)] If $\tau_+(0)$ is the smallest critical value and $\upmu<\delta_1$ then we have instability switches.
\item[(b)] If, in addition, $\upmu>\delta_1-\delta_2$ i.e. $\upmu<\delta_1<\upmu+\delta_2$ then we have switch of the form $U-S-U-S-U$.
Otherwise (i.e. if $\upmu<\delta_1-\delta_2$) we have switch of the form $U-S-U$.
\end{itemize} 
So, we have instability switch upto smallest $k$ such that $0<\tau_+(0)<\tau_-(1)<\tau_+(1)\ldots\tau_-(k)$ where $\upmu_k=\tau_{-}(k+1)-\tau_+(k)>\delta_1$.\\
On the left side of $\Gamma_{13}$, $\tau_{-}(1)> \tau_{+}(0)$ so the multiplicity of positive root on the right side increase by two and system becomes unstable $\forall \tau\geq0$.\\
Using this analysis, we are able to sketch the stability region in the complete $a_1c$-plane for $b<0$ and $0<\alpha<1/2$ (cf. Figure \eqref{mainfig4.2}).
\subsection{ Consider $b>0$ and $1/2<\alpha<1$.}
\begin{itemize}
\item [\textbf{Case 1:}] $c>0$.\\
If we solve equation \eqref{value of valpha}, we get two positive values of $v^\alpha$ namely $v_1$ and $v_2$ with $0<v_1<v_2$ as in previous sections. Corresponding to those $v_1$ and $v_2$ the critical values of delay $\tau$ are $ \tau_-(n)(v_1)$ and $\tau_-(n)(v_2)$.\\
We observed that, $Re\big(\dfrac{d\lambda}{d\tau}|_{u=0}\big)$ is negative at $v_1$ and positive at $v_2$ for these values of parameters.
Now, we define\\
 $\delta_1=\tau_{-}(n+1)(v_1)-\tau_{-}(n)(v_1)=\frac{2\pi}{v_1}$,\\
$\delta_2=\tau_{-}(n+1)(v_2)-\tau_{-}(n)(v_2)=\frac{2\pi}{v_2}$\\
and $\upmu=\tau_{-}(1)(v_2)-\tau_{-}(1)(v_1)$.\\
Note that $\upmu=\delta_2$ is the bifurcation curve which we call as $\Gamma_{14}$ (see data set 14). On the left side of $\Gamma_{14}$, we have $\upmu<\delta_2$ and $0<\tau_{-}(1)(v_2)<\tau_{-}(1)(v_1)<\ldots\tau_{-}(k)(v_2)$ which gives SS upto $\tau_{-}(k)(v_2)$. Moreover, $k$ is the smallest positive integer where $\upmu_k=\tau_-{(k)}(v_2)-\tau_-{(k)}(v_1)>\delta_2$ and if $\tau>\tau_{-}(k)$ then system becomes unstable. On the right of $\Gamma_{14}$, we have $\upmu>\delta_2$ so we get only SSR from $[0,\tau_{-}(1)(v_2))$ as the critical values occur as $0<\tau_{-}(1)(v_2)<\tau_{-}(2)(v_2)<\ldots$. Note that this bifurcation curve $\Gamma_{14}$ intersects the vertical axis at $c_1$. The expression for $c_1$ is given in data set 15 of this manuscript.
\end{itemize}
\begin{itemize}
\item [\textbf{Case 2:}] $c<0$.\\ The critical values of delay $\tau$ with respect to $v_1$ and $v_2$ are $ \tau_-(n)$ and $\tau_+(n)$ respectively. 
Note that $Re\big(\dfrac{d\lambda}{d\tau}|_{u=0}\big)$ is negative at $v_1$ and positive at $v_2$ for these values of parameters. We define,\\
 $\delta_1=\tau_{-}(n+1)-\tau_{-}(n)=\frac{2\pi}{v_1}$,\\
$\delta_2=\tau_{+}(n+1)-\tau_{+}(n)=\frac{2\pi}{v_2}$\\
and $\upmu=\tau_{+}(0)-\tau_{-}(1)$.\\
The bifurcation curve $\Gamma_{15}$ is obtained by solving $\upmu=\delta_2$ (see dataset-16). Note that the curve $\Gamma_{15}$ meets the curve $\Gamma_2$ at $c_{3}$. The expression for $c_{3}$ is also given in data set 17. On the right side of $\Gamma_{15}$ we have SS and on the left side we have SSR.
\end{itemize}
\subsection{ Consider $b<0$ and $1/2<\alpha<1$.}
\begin{itemize}
\item [\textbf{Case 1:}] $c>0$.\\
If we solve equation \eqref{value of valpha} we get two values of $v^\alpha$ as $v_1$ and $v_2$ with $0<v_1<v_2$. The critical values of $\tau$ with respect to $v_1$ are $\tau_+(n)(v_1)$ and $\tau_+(n)(v_2)$. 
Here we observed that $Re\Big(\dfrac{d\lambda}{d\tau}|_{u=0}\Big)$ is negative at $v_1$ and positive at $v_2$. We set  
 $\delta_1=\tau_{+}(n+1)(v_1)-\tau_{+}(n)(v_1)=\frac{2\pi}{v_1}$,\\
$\delta_2=\tau_{+}(n+1)(v_2)-\tau_{+}(n)(v_2)=\frac{2\pi}{v_2}$\\
and $\upmu=\tau_{+}(0)(v_2)-\tau_{+}(0)(v_2)$.\\ 
When $\upmu<\delta_2$, we have $0<\tau_+(0)(v_2)<\tau_+(0)(v_1)<\tau_+(1)(v_2)\ldots\tau_+(k)(v_2)$ and we have SS upto $\tau_+(k)(v_2)$ where $k$ is the smallest positive integer such that $\upmu_k=\tau_+(k)(v_2)-\tau_+(k)(v_1)>\delta_2$. Moreover, when $\upmu>\delta_2$, we get $0<\tau_+(0)(v_2)<\tau_+(1)(v_2)<\tau_+(0)(v_1)\ldots$ and we get SSR $\tau\in[0,\tau_+(0)(v_2))$. So, $\upmu=\delta_2$ is the bifurcation curve $\Gamma_{16}$ for $b<0$ and $1/2<\alpha<1$ (see Figure \eqref{mainfig4.4}). See the dataset-18 for the expression of $\Gamma_{16}$. Note that the curve $\Gamma_{16}$ and the tangent curve $\Gamma_9$ will meet at $(a_{5},c_{5})$. For $b=-1$ and $\alpha=0.8$ we get $(a_{5},c_{5})=(-1.97444,2.00571)$. 
\end{itemize}
\begin{itemize}
\item [\textbf{Case 2:}] $c<0$.\\
Let us consider the region bounded by the curve $\Gamma_2$ and the tangent curve $\Gamma_{10}$. In this region, we have two values $v_1$ and $v_2$ with $0<v_1<v_2$ if we solve equation \eqref{value of valpha}. Corresponding to $v_1$, we have critical values of delay $\tau$ as $\tau_-(n)(v_1)$ 
and with respect to $v_2$ the critical value of $\tau$ as $\tau_-(n)(v_2)$. 
We define,\\
 $\delta_1=\tau_{-}(n+1)(v_1)-\tau_{-}(n)(v_1)=\frac{2\pi}{v_1}$,\\
$\delta_2=\tau_{-}(n+1)(v_2)-\tau_{-}(n)(v_2)=\frac{2\pi}{v_2}$\\
and $\upmu=\tau_{-}(1)(v_2)-\tau_{-}(1)(v_1)$.\\
We have $\upmu=\delta_2$ as the bifurcation curve $\Gamma_{17}$ given in data-set 19 (see Figure \eqref{mainfig4.4}). On the right side of $\Gamma_{17}$ the system is stable at $\tau=0$ and the smallest critical value of $\tau$ is $\tau_{-}(1)(v_2)$. So, we have $0<\tau_{-}(1)(v_2)<\tau_-(1)(v_1)<\tau_-(2)(v_2)\ldots \tau_-(k)(v_2)$ and we get stability switches upto $\tau_-(k)(v_2)$. On the left side of $\Gamma_{17}$, we have SSR $[0,\tau_-(1)(v_2))$. Note that the curve $\Gamma_{17}$ meets with the curve $\Gamma_{2}$ at $c_{6}$.
 Note that some region which is on the left side of the curve $\Gamma_2$ is yet to be done (Figure \eqref{mainfig4.4}). On the left side of the curve $\Gamma_2$, system \eqref{eq3.2} is unstable at $\tau=0$. If we solve equation \eqref{value of valpha} in the region bounded by $\Gamma_2$, vertical axis with $c<0$ and the curve $\Gamma_0$ we get two values of $v$ say $v_1$ and $v_2$ with $0<v_1<v_2$ given in the data set 9. These $v_1$ and $v_2$ give the critical values of delay $\tau$ as  $\tau_+(n)(v_1)$ and $ \tau_+(n)(v_2)$.\\
We notice that  $Re\Big(\dfrac{d\lambda}{d\tau}|_{u=0}\Big)$ is negative at $v_1$ and positive at $v_2$. So, if $\tau_+(0)(v_1)<\tau_+(0)(v_2)$ then two roots are going from right half plane to the left half plane at $\tau+(0)(v_1)$ and no roots are on the right side. On the other hand, the roots are moving from the left to the right half plane at $\tau_+(0)(v_2)$ and the system becomes unstable.
 So, system \eqref{eq3.2} is stable for $\tau_+(0)(v_1)<\tau<\tau_+(0)(v_2)$. If $\tau_+(0)(v_1)>\tau_+(0)(v_2)$ it means that the system \eqref{eq3.2} unstable at $\tau=0$ and some more roots of characteristic equation \eqref{eq3.3} goes from left half plane to the right half plane. So, the system becomes unstable for $\forall\tau\geq0$. So, the bifurcation curve $\Gamma_{18}$ given in data set 20 that discriminate IS from unstable $\forall\tau$ is \\
\begin{equation*}
\tau_{+}(0)(v_1)=\tau_{+}(0)(v_2).
\end{equation*}
In the region bounded by $\Gamma_2$ and $\Gamma_{18}$ we have $\tau_{+}(0)(v_1)<\tau_{+}(0)(v_2)$. So, $0<\tau_+(0)(v_1)<\tau_+(0)(v_2)<\tau_+(1)(v_1)\ldots\tau_+(k)(v_2)$ and we have IS upto smallest positive $k$ such that $\upmu_k=\tau_+(k)(v_2)-\tau_+(k)(v_1)>\delta_1$ where $\delta_1=\tau_+(n+1)(v_1)-\tau_+(n)(v_1)=\frac{2 \pi}{v_1}$. On the left side of $\Gamma_{18}$, we have $\frac{\arccos(\frac{v_1\cos(\frac{\alpha\pi}{2})+c v_1\cos(\alpha\pi)-a}{b})}{v_1^{1/\alpha}}>\frac{\arccos(\frac{v_2\cos(\frac{\alpha\pi}{2})+c v_2\cos(\alpha\pi)-a}{b})}{v_2^{1/\alpha}}$ and system \eqref{eq3.2} is unstable for all $\tau\geq0$ (See Figure \eqref{mainfig4.4}). \\
\end{itemize}
Using all this analysis we provide the stability region for all the cases viz. $b>0$, $0<\alpha<1/2$, $b<0$, $0<\alpha<1/2$, $b>0$, $1/2<\alpha<1$ and $b<0$, $1/2<\alpha<1$.
\begin{figure}
    \centering
    \includegraphics[scale=1.]{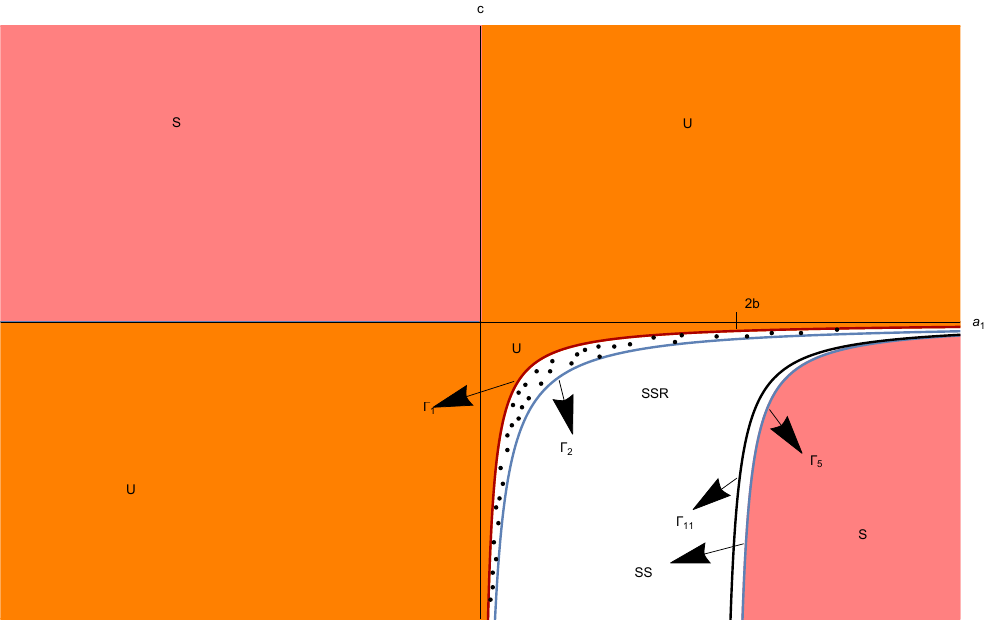}
    \caption{Graph for $b$ positive and $0<\alpha<1/2$ in complete $a_1c-$plane }\label{mainfig4.1}
\end{figure}
\begin{figure}
    \centering
    \includegraphics[scale=1.5]{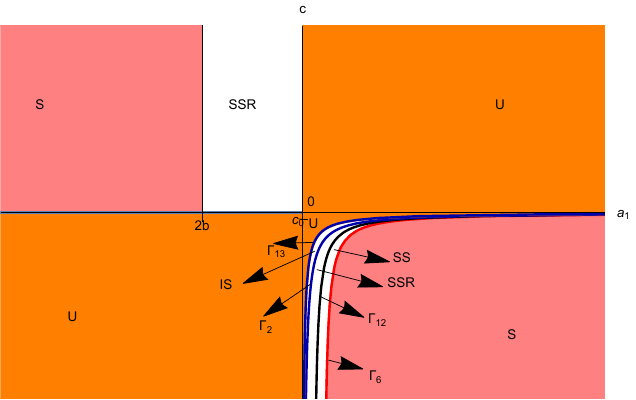}
    \caption{Graph for $b$ negative and $0<\alpha<1/2$ in complete $a_1c-$plane }\label{mainfig4.2}
\end{figure}
\begin{figure}
    \centering
    \includegraphics[scale=1.5]{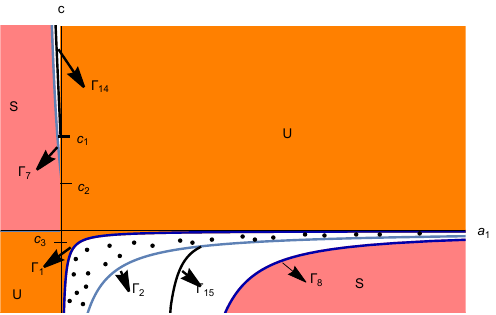}
    \caption{Graph for $b$ positive and $1/2<\alpha<1$ in complete $a_1c-$plane }\label{mainfig4.3}
\end{figure}
\begin{figure}
    \centering
    \includegraphics[scale=1.35]{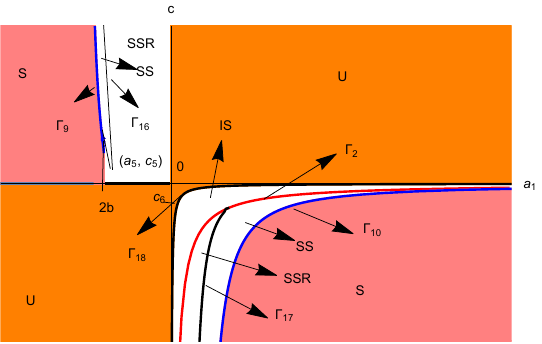}
    \caption{Graph for $b$ negative and $1/2<\alpha<1$ in complete $a_1c-$plane }\label{mainfig4.4}
\end{figure}
\section{Examples for the FDDE (\ref{eq3.2})}\label{sec3.9}
In this subsection, we verify the results derived in the manuscript and presented in Figures \eqref{mainfig4.1}, \eqref{mainfig4.2}, \eqref{mainfig4.3} and (\ref{mainfig4.4}). We utilize the numerical methods proposed in \cite{bhalekar2019analysis} to solve these examples.\\
In the regions S and U of Figures \eqref{mainfig4.1}-\eqref{mainfig4.4}, we verified the results by taking a wide range of parameter values. As a representative example, we include a set of parameter values in these regions and provide the solution curve. The details are given in the Following tables \eqref{table4.1} to \eqref{table4.4}.\\
\begin{table} 
\begin{tabular}{ |p{6.5cm}||p{3.5cm}|p{3cm}|  }
 \hline
 Parameter values & Nature of solution &  Figure\\
 \hline
$a=0.45,b=1,c=4,\alpha=0.3,\tau=1.1$ (First quadrant in Figure \eqref{mainfig4.1}) & Unstable & Figure \eqref{figex4.1}(a)\\
\hline
  $a=-9,b=4,c=4,\alpha=0.4,\tau=0.17$(Second quadrant in Figure \eqref{mainfig4.2})& Stable & Figure \eqref{figex4.1}(b)\\
 \hline 
 $a=-8.6, b=1.1, c=-4,\alpha=0.49, \tau=0.17$ (Third quadrant in Figure \eqref{mainfig4.1})& Unstable & Figure \eqref{figex4.1}(c)\\
     \hline
  $a=-1, b=1.3, c=-0.8,\alpha=0.35, \tau=1.7$ (In the fourth quadrant region bounded by $c<0$ and the curve $\Gamma_2$ in Figure \eqref{mainfig4.1} )& Unstable & Figure \eqref{figex4.1}(d) \\
    \hline 
   $a=3.8, b=1.2, c=-4,\alpha=0.3, \tau=1$ (Region on the left side of $\Gamma_6$)& Stable & Figure\eqref{figex4.1}(e) \\
    \hline
\end{tabular}
\caption{System \eqref{eq3.2} $b>0$ and $0<\alpha<1/2$}\label{table4.1}
  \end{table}

\begin{figure}
	\subfloat[Divergent solution for $\alpha=0.3$, $c=4$ and $\tau=1.1$]{%
		\includegraphics[scale=0.46]{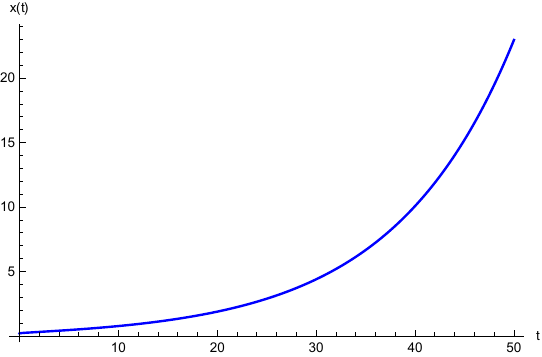}
	}\hspace{0.1cm}
	\subfloat[Convergent solution for $\alpha=0.4$, $a=-9$, $b=4$, $c=4$ and $\tau=0.17$]{%
		\includegraphics[scale=0.46]{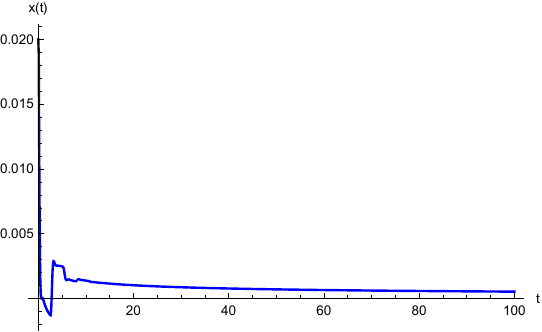}
	}\hspace{0.1cm}
	\subfloat[Divergent solution for $\alpha=0.49$, $a=-8.6$, $b=1.1$, $c=-4$ and $\tau=0.17$]{%
		\includegraphics[scale=0.46]{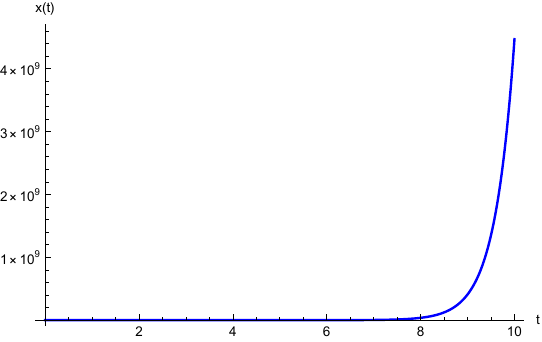}
	}\hspace{0.1cm}
	\subfloat[Unstable solution for $a=-1$, $b=1.3$, $\alpha=0.35$, $c=-0.8$ and $\tau=1.7$]{%
		\includegraphics[scale=0.47]{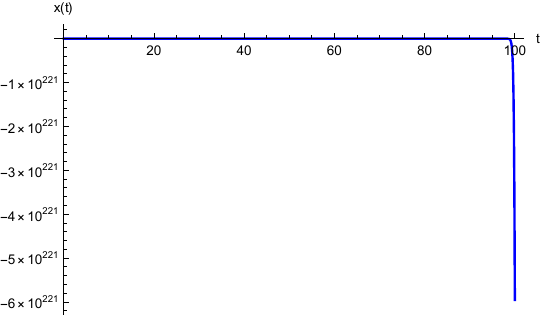}
	 }\hspace{0.1cm}
	\subfloat[Convergent solution for $a=3.8$, $b=1.2$, $\alpha=0.3$, $c=-1$ and $\tau=1$]{%
	\includegraphics[scale=0.47]{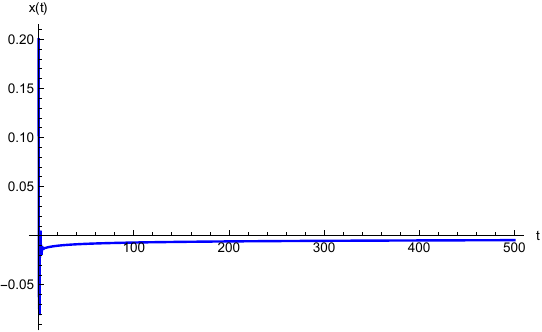}
	}	
\caption{ Figures representing Table \eqref{table4.1}}
	\label{figex4.1}
\end{figure}
\begin{table}
\begin{tabular}{ |p{6.5cm}||p{3.5cm}|p{3cm}|  }
 \hline
  Parameter values & Nature of solution & Figure\\  
 \hline
$a=4.5,b=-3,c=9,\alpha=0.45,\tau=0.3$ (First quadrant in Figure \eqref{mainfig4.2})  & Unstable & Figure \eqref{figex3.1}(a) \\
\hline
 $a=-8,b=-3.5,c=10,\alpha=0.29,\tau=0.8$ ((When $a_1<2b<0$ in Figure \eqref{mainfig4.2})  & Stable & Figure \eqref{figex3.1}(b) \\
 \hline 
 $a=-1.1,b=-3.5,c=-1,\alpha=0.38,\tau=1$ (Third quadrant in Figure \eqref{mainfig4.3}) & Unstable & Figure \eqref{figex3.1}(c)\\
 \hline
 $a=0.3,b=-0.1,c=-0.2,\alpha=0.38,\tau=0.3$ (Region bounded by $c<0$ and $\Gamma_{13}$ in Figure \eqref{mainfig4.2})  & Unstable & Figure \eqref{figex3.1}(d)\\
 \hline
 $a=7,b=-0.1,c=-2,\alpha=0.38,\tau=0.3$ (Region on the right side of $\Gamma_6$ )  & Stable & Figure \eqref{figex3.1}(e) \\
 \hline
\end{tabular}
 \caption{When $b<0$ and $0<\alpha<1/2$}\label{table4.2}
 \end{table}

\begin{figure}
	\subfloat[Divergent solution for $a=4.5$, $b=-3$, $\alpha=0.45$, $c=9$ and $\tau=0.3$]{%
		\includegraphics[scale=0.46]{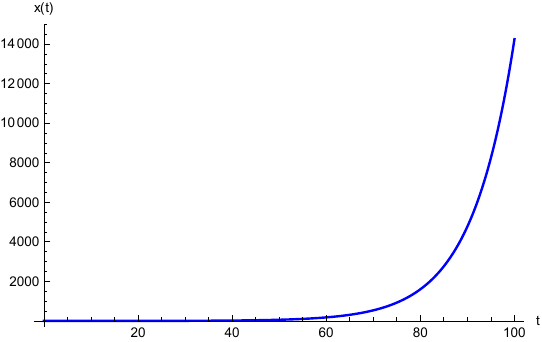}
	}\hspace{0.1cm}
	\subfloat[Convergent solution for $a=-8$, $b=-3.5$, $\alpha=0.29$, $c=10$ and $\tau=0.8$]{%
		\includegraphics[scale=0.46]{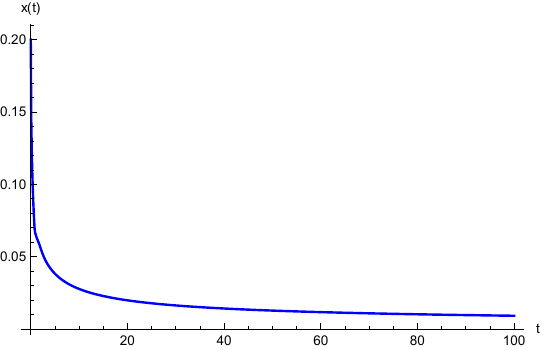}
	}\hspace{0.1cm}
	\subfloat[Divergent solution for $a=-1.1$, $b=-3.5$, $\alpha=0.38$, $c=-1$ and $\tau=1$]{%
		\includegraphics[scale=0.46]{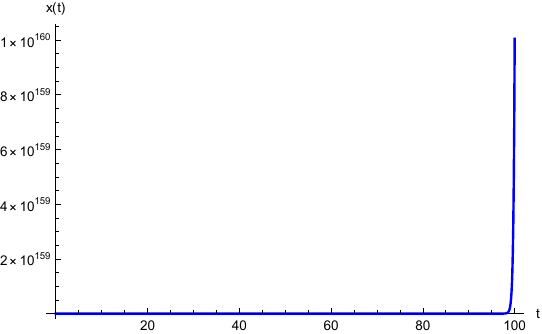}
	}\hspace{0.1cm}
	\subfloat[Divergent solution for $a=0.3$, $b=-0.1$, $\alpha=0.38$, $c=-0.2$ and $\tau=0.3$]{%
		\includegraphics[scale=0.47]{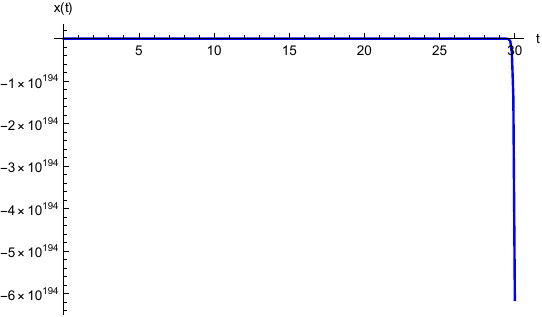}
	 }\hspace{0.1cm}
	\subfloat[Convergent solution for $a=7$, $b=-0.1$, $\alpha=0.38$, $c=-2$ and $\tau=0.3$]{%
	\includegraphics[scale=0.47]{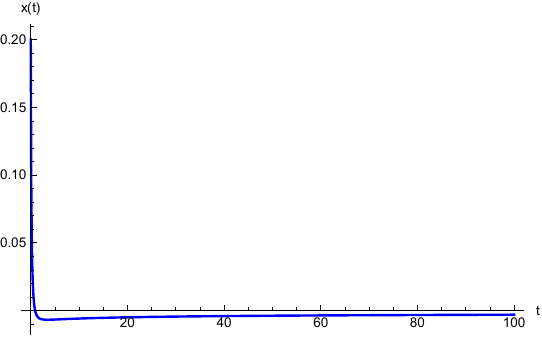}
	}	
\caption{Table \eqref{table4.2} for $b<0$ and $0<\alpha<1/2$}
	\label{figex3.1}
\end{figure}

\begin{table}
\begin{tabular}{ |p{8cm}||p{3.5cm}|p{3cm}|  }
 \hline
 Parameter values & Nature of solution & Figure\\  
 \hline
$a=1.06,b=8,c=4,\alpha=0.8,\tau=0.09$ (First quadrant in Figure \eqref{mainfig4.3}) & Unstable & Figure (\ref{figex4.3}(a)) \\
\hline
     $a=-8,b=3.5,c=4,\alpha=0.75,\tau=0.6$ (Region on the left side of $\Gamma_7$ in Figure \eqref{mainfig4.3})  & Stable & Figure (\ref{figex4.3} (b))\\
 \hline 
 $a=-8,b=5,c=-9,\alpha=0.93,\tau=0.9$ (Unstable region in the third quadrant in Figure \eqref{mainfig4.3})  & Unstable & Figure (\ref{figex4.3} (c)) \\
 \hline
 $a=-4.9,b=5,c=-2,\alpha=0.96,\tau=0.9$ (Region on the left side of $\Gamma_2$ in Figure \eqref{mainfig4.3})  & Unstable & Figure (\ref{figex4.3} (d))\\
 \hline
 $a=200,b=5,c=-2,\alpha=0.96,\tau=2.3$ (Region on the right side of $\Gamma_8$ in Figure \eqref{mainfig4.3})    & Stable & Figure (\ref{figex4.3} (e))\\
 \hline
 \end{tabular}
 \caption{System \eqref{eq3.2} when $b>0$ and $1/2<\alpha<1$}\label{table4.3}

 \end{table}
 
\begin{figure}
	\subfloat[Divergent solution for $a=1.06$, $b=8$, $\alpha=0.8$, $c=4$ and $\tau=0.09$]{%
		\includegraphics[scale=0.46]{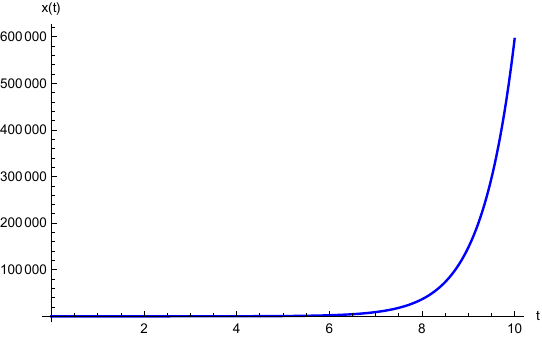}
	}\hspace{0.1cm}
	\subfloat[Convergent solution for $a=-8$, $b=3.5$, $\alpha=0.75$, $c=4$ and $\tau=0.6$]{%
		\includegraphics[scale=0.46]{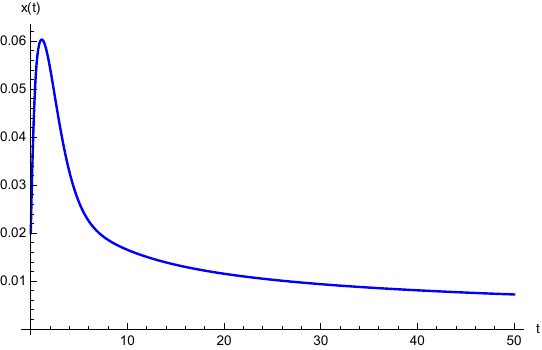}
	}\hspace{0.1cm}
	\subfloat[Divergent solution for $a=-8$, $b=5$, $\alpha=0.93$, $c=-9$ and $\tau=0.9$]{%
		\includegraphics[scale=0.46]{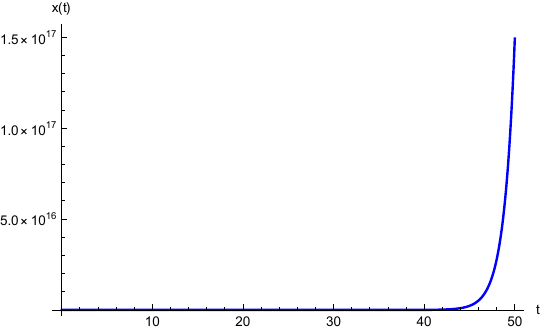}
	}\hspace{0.1cm}
	\subfloat[Divergent solution for $a=-4.9$, $b=5$, $\alpha=0.96$, $c=-2$ and $\tau=0.9$]{%
		\includegraphics[scale=0.47]{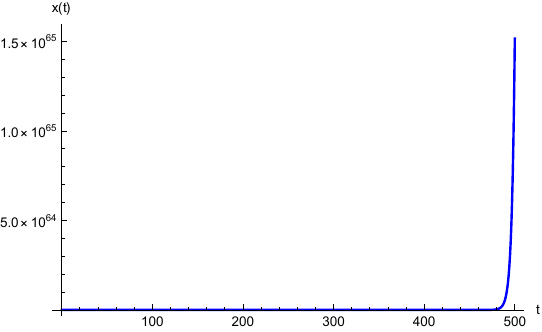}
	 }\hspace{0.1cm}
	\subfloat[Convergent solution for $a=200$, $b=5$, $\alpha=0.96$, $c=-2$ and $\tau=2.3$]{%
	\includegraphics[scale=0.47]{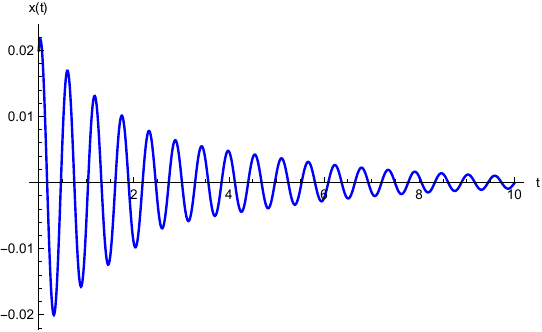}
	}	
\caption{Table  for $b>0$ and $1/2<\alpha<1$}
	\label{figex4.3}
\end{figure}
\begin{table}
\begin{center}
\begin{tabular}{ |p{8cm}||p{3.5cm}|p{3cm}|  }
 \hline
  Parameter values & Nature of solution & Figure\\  
 \hline
$a=4,b=-3.3,c=19,\alpha=0.85,\tau=3.5$ (First quadrant in Figure \eqref{mainfig4.4})  & Unstable & Figure (\ref{figex4.4} (a)) \\
\hline
$a=-7,b=-3.3,c=19,\alpha=0.85,\tau=1.8$ (Second quadrant in Figure \eqref{mainfig4.4})  & Stable & Figure (\ref{figex4.4} (b))\\
 \hline 
$a=-7,b=-3.3,c=-5,\alpha=0.79,\tau=1.3$ (Third quadrant in Figure \eqref{mainfig4.4})  & Unstable & Figure (\ref{figex4.4} (c)) \\
 \hline
$a=3.002,b=-3,c=-1,\alpha=0.7,\tau=1.3$ (Region on the left side of $\Gamma_{18}$ in Figure \eqref{mainfig4.4}) & Unstable & Figure (\ref{figex4.4} (d)) \\
 \hline
$a=10,b=-3,c=-2,\alpha=0.7,\tau=1.7$ (Region on the right side of $\Gamma_{10}$ in Figure \eqref{mainfig4.4})  & Stable & Figure (\ref{figex4.4} (e))\\
 \hline
 \end{tabular}
 \caption{System \eqref{eq3.2} when $b<0$ and $1/2<\alpha<1$}\label{table4.4}
\end{center}
\end{table}
\begin{figure}
	\subfloat[Divergent solution for $a=4$, $b=-3.3$, $\alpha=0.85$, $c=19$ and $\tau=3.5$]{%
		\includegraphics[scale=0.46]{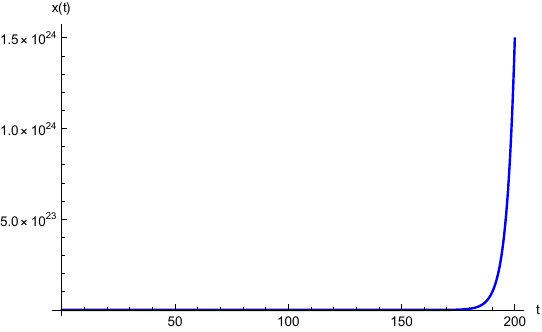}
	}\hspace{0.1cm}
	\subfloat[Convergent solution for $a=-7$, $b=-3.3$, $\alpha=0.85$, $c=19$ and $\tau=1.8$]{%
		\includegraphics[scale=0.46]{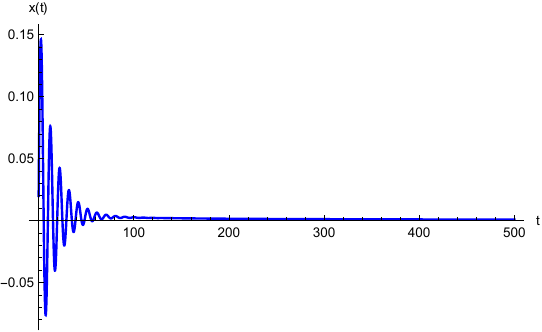}
	}\hspace{0.1cm}
	\subfloat[Unstable solution for $a=-7$, $b=-3.3$, $\alpha=0.79$, $c=-5$ and $\tau=1.3$]{%
		\includegraphics[scale=0.46]{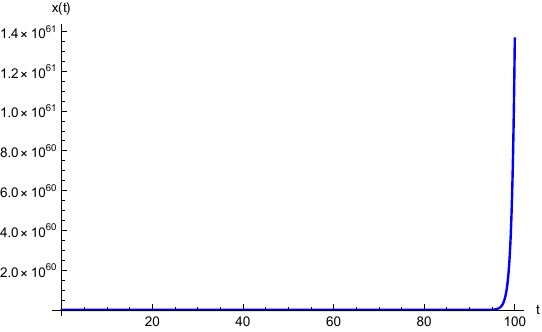}
	}\hspace{0.1cm}
	\subfloat[Unstable solution for $a=3.002$, $b=-3$, $\alpha=0.7$, $c=-1$ and $\tau=1.3$]{%
		\includegraphics[scale=0.47]{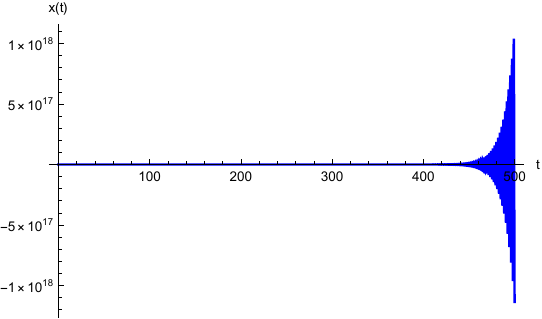}
	 }\hspace{0.1cm}
	\subfloat[Stable solution for $a=10$, $b=-3$, $\alpha=0.7$, $c=-2$ and $\tau=1.7$]{%
	\includegraphics[scale=0.47]{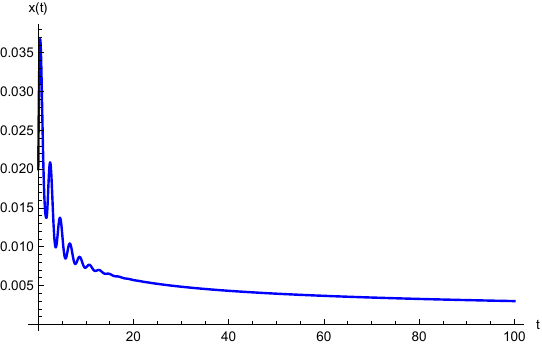}
	}	
\caption{Table for $b<0$ and $1/2<\alpha<1$}
	\label{figex4.4}
\end{figure}

Now, Let us consider the examples in each cases where the stability depend on delay parameter.\\
\begin{Ex}
Consider $b>0$ and $0<\alpha<1/2$ (Figure \eqref{mainfig4.1}). Only region that depends on $\tau$ is the region bounded by $\Gamma_2$ and $\Gamma_5$. Let us fix $c=-0.4$, $\alpha=0.3$ and $b=1$. At $\Gamma_2$, $\Gamma_{11}$ and $\Gamma_5$, the bifurcation values of $a_1$ are $0.78726$, $2.78762$ and $2.8219$ or $a=-0.21274$, $1.78762$ and $1.8219$ respectively. So, the system \eqref{eq3.2} will be unstable for $a<-0.21274$. We will have SSR for $-0.21274<a<1.78762$ and SS for $1.78762<a<1.8219$ from Figure \eqref{mainfig4.1}. System becomes stable $\forall\tau\geq0$ if $a>1.8219$. The unstable solution for $a=-0.7$ and $\tau=0.1$ is given in Figure \eqref{figex4.5}(a). For $a=1.7$ we get two positive values of $v^\alpha$ as $1.12581$ and $2.12454$. Corresponding to $v^\alpha=1.12581$ the critical values of delay are $2.18454, 6.41737, 10.6502,\ldots$. Furthermore, $v^\alpha=2.12454$ gives critical $\tau$ as $0.212729, 0.722411, 1.23209,\ldots$. The SSR is given by $\tau\in[0,0.212729)$. Note that, at the critical values $0.722411$ and $1.23209$, $Re\Big(\dfrac{d\lambda}{d\tau}|_{u=0}\Big)>0$ as described in Section \eqref{section3.8}. Therefore, the system \eqref{eq3.2} remains unstable for $\tau>0.212729$.  The stable solution for $\tau=0.1$ is shown in Figure \eqref{figex4.5}(b) and the unstable solution for $\tau=0.3$ is given in Figure \eqref{figex4.5}(c). Now let us consider $a=1.8$ in the SS region. For this parameter the two positive values of $v^\alpha$ are $1.46349$ and $1.88464$. Corresponding to $v^\alpha=1.46349$, we get the critical values of $\tau$ as $0.874706, 2.64025, 4.40578, 6.17132,\ldots$ and corresponding to $v^\alpha=1.88464$ we get $0.343884, 1.10378, 1.86368,\ldots$. We have $Re\Big(\dfrac{d\lambda}{d\tau}|_{u=0}\Big)>0$ at $v^\alpha=1.88464$ and $Re\Big(\dfrac{d\lambda}{d\tau}|_{u=0}\Big)<0$ at  $v^\alpha=1.46349$. So, the characteristics roots will shift from left to right half plane at $\tau=0.343884, 1.10378, 1.86368,\ldots$ and those roots will again shift back at $\tau=0.874706, 2.64025, 4.40578, 6.17132,\ldots$.  We get stable solution for $\tau=0.3$(cf. Figure \eqref{figex4.5}(d)), unstable solution for $\tau=0.4$(cf. Figure \eqref{figex4.5}(e)) and again stable solution for $\tau=1$(cf. Figure \eqref{figex4.5}(f)).
\end{Ex}
\begin{figure}
	\subfloat[Unstable solution for $a=-0.7$, $b=1$, $\alpha=0.3$, $c=-0.4$ and $\tau=0.1$]{%
		\includegraphics[scale=0.46]{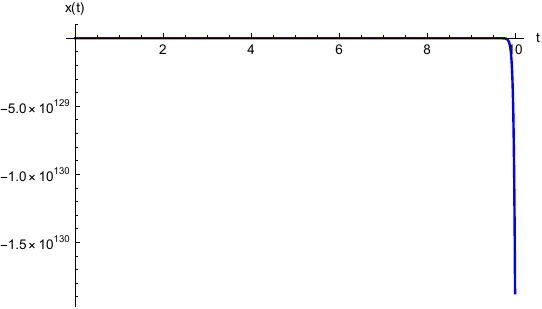}
	}\hspace{0.1cm}
	\subfloat[Stable solution for $a=1.7$, $b=1$, $\alpha=0.3$, $c=-0.4$ and $\tau=0.1$]{%
		\includegraphics[scale=0.46]{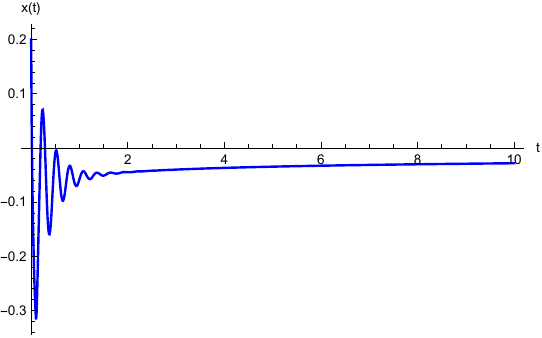}
	}\hspace{0.1cm}
	\subfloat[Unstable solution for $a=1.7$, $b=1$, $\alpha=0.3$, $c=-0.4$ and $\tau=0.3$]{%
		\includegraphics[scale=0.46]{unstablesolution10.pdf}
	}\hspace{0.1cm}
	\subfloat[Stable solution for $a=1.8$, $b=1$, $\alpha=0.3$, $c=-0.4$ and $\tau=0.3$]{%
		\includegraphics[scale=0.46]{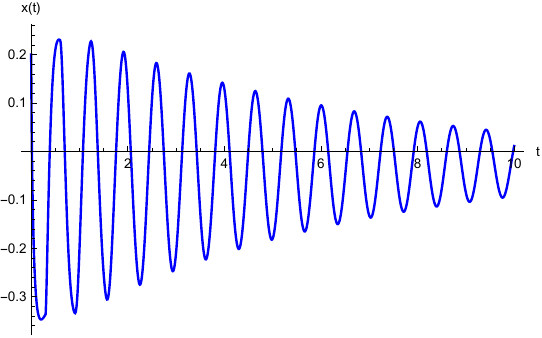}
	 }\hspace{0.1cm}
	\subfloat[Unstable solution for $a=1.8$, $b=1$, $\alpha=0.3$, $c=-0.4$ and $\tau=0.4$]{%
	\includegraphics[scale=0.46]{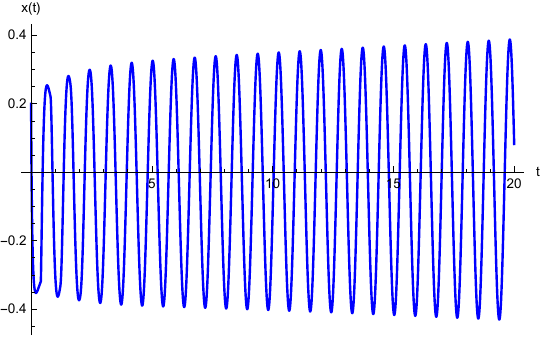}
	}\hspace{0.1cm}
	\subfloat[Stable solution for $a=1.8$, $b=1$, $\alpha=0.3$, $c=-0.4$ and $\tau=1$]{%
	\includegraphics[scale=0.46]{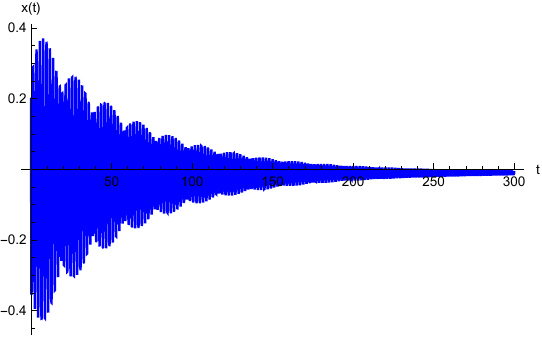}
	}	
\caption{}
	\label{figex4.5}
\end{figure}

\begin{Ex}\label{eg4.1}
Consider the case $b<0$ and $0<\alpha<1/2$ (Figure \eqref{mainfig4.2}). Let us verify the region $c>0$ and $2b<a_1<0$. We have $a=-0.3$, $b=-10$, $\alpha=0.3$ and $c=0.5$. We get only one positive value of $v^\alpha$ as $3.59088$ and corresponding to that $v^\alpha$ we have $Re\Big(\dfrac{d\lambda}{d\tau}|_{u=0}\Big)>0$. So, the smallest critical value of delay from equation \eqref{smallest critical value of delay} is $\tau=0.0336727$. The stable and unstable solution for $\tau=0.02$ and $\tau=0.04$ are given in Figure \eqref{figex4.6}(a) and \eqref{figex4.6}(b), respectively.
\end{Ex}
\begin{figure}
	\subfloat[Stable solution for $a=-0.3$, $b=-10$, $\alpha=0.3$, $c=0.5$ and $\tau=0.02$]{%
		\includegraphics[scale=0.8]{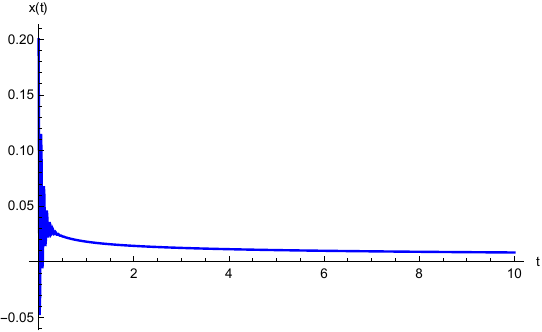}
	}\hspace{0.1cm}
	\subfloat[Unstable solution for $a=-0.3$, $b=-10$, $\alpha=0.3$, $c=0.5$ and $\tau=0.04$]{%
		\includegraphics[scale=0.8]{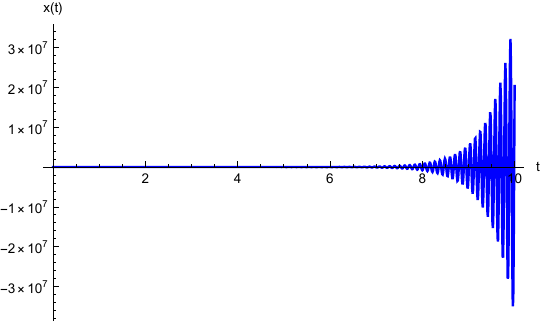}
	}
\caption{Figure of Example \eqref{eg4.1}}
	\label{figex4.6}
\end{figure}
\begin{Ex}
Now let us consider the fourth quadrant of Figure \eqref{mainfig4.2} that depend on delay $\tau$. When $b=-1$, $\alpha=0.45$ we have $c_0=-0.195086$.
\begin{itemize}
\item[\ding{118}] So, when $c=-0.3<c_0$ then at bifurcation curve $\Gamma_{13}$, $a_1=0.674322$ ($a=1.674322$), at $\Gamma_2$, $a_1=1.44121$ ($a=2.44121$), at $\Gamma_{12}$, $a_1=1.50111$ ($a=2.50111$) and at $\Gamma_{6}$, $a_1=1.62678$ ($a=2.62678$).
 \begin{itemize} 
 \item[$\bullet$] If $a=2.35$ the two positive values of $v^\alpha$ are $2.03726$ and $3.1629$. We have $Re\Big(\dfrac{d\lambda}{d\tau}|_{u=0}\Big)>0$ at $v^\alpha=3.1629$ and $Re\Big(\dfrac{d\lambda}{d\tau}|_{u=0}\Big)<0$ at $v^\alpha=2.03726$. The critical values of $\tau$ corresponding to $2.03726$ are $0.0192191,1.31166,2.60411,\ldots$ and for $3.1629$ we have $0.397771, 0.884042, 1.37031,\ldots$. So, we get unstable solution for $\tau\in[0,0.0192191)$, stable solution for $\tau\in(0.0192191,0.397771)$ and it will remain unstable $\forall\tau>0.397771$. The unstable solution for $\tau=0.01$ (Figure \eqref{figex4.7}(a)), stable solution for $\tau=0.02$ (Figure \eqref{figex4.7}(b)) and again unstable solution for $\tau=0.5$ (Figure \eqref{figex4.7}(c)) (Instability switch).
 \item[$\bullet$]  When $a=2.48$, we are in the SSR region from Figure \eqref{mainfig4.2}. We get positive values of $v^\alpha$ as $2.26432$ and $3.08213$.  Note that we have $Re\Big(\dfrac{d\lambda}{d\tau}|_{u=0}\Big)<0$ at $v^\alpha=2.26432$ and $Re\Big(\dfrac{d\lambda}{d\tau}|_{u=0}\Big)>0$ at $v^\alpha=3.08213$. Critical values of $\tau$ are $1.01403, 2.03598, 3.05793,\ldots$ corresponding to $2.264320$ and $0.450092, 0.965138, 1.48018,\ldots$ corresponding to $3.08213$. So, we have stable solution for $\tau\in[0,0.450092)$ and unstable for all $\tau>0.450092$. The stable solution for $\tau=0.3$ is given in Figure \eqref{figex4.7}(d) and unstable solution for $\tau=0.5$ is given in Figure \eqref{figex4.7}(e).
 \item[$\bullet$] Now, let us take $a=2.6$ so that we are in the SS region from Figure \eqref{mainfig4.2}. So, by solving equation \eqref{value of valpha} for $v^\alpha$ we get $2.56269$ and $2.91154$. We have $Re\Big(\dfrac{d\lambda}{d\tau}|_{u=0}\Big)<0$ at $v^\alpha=2.56269$ and $Re\Big(\dfrac{d\lambda}{d\tau}|_{u=0}\Big)>0$ at $v^\alpha=2.91154$. The critical values of $\tau$ are $0.740921, 1.51711, 2.2933,\ldots$ corresponding to $2.56269$ and $0.522197,\\
  1.10671, 1.69123,\ldots$ corresponding to $2.91154$. We get a stability switch. The stable solution for $\tau=0.4$ is given in Figure \eqref{figex4.7}(f), unstable solution for $\tau=0.6$ is given in Figure \eqref{figex4.7}(g) and stable solution for $\tau=1$ is given in Figure \eqref{figex4.7}(h).
 \end{itemize}
 \end{itemize}
\end{Ex}
\begin{figure}
	\subfloat[Unstable solution for $a=2.35$, $b=-1$, $\alpha=0.45$, $c=-0.3$ and $\tau=0.01$]{%
		\includegraphics[scale=0.46]{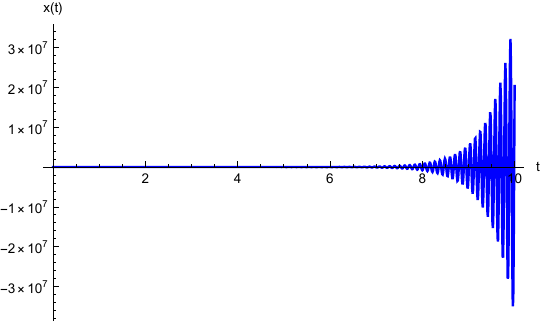}
	}\hspace{0.1cm}
	\subfloat[Stable solution for $a=2.35$, $b=-1$, $\alpha=0.45$, $c=-0.3$ and $\tau=0.02$]{%
		\includegraphics[scale=0.46]{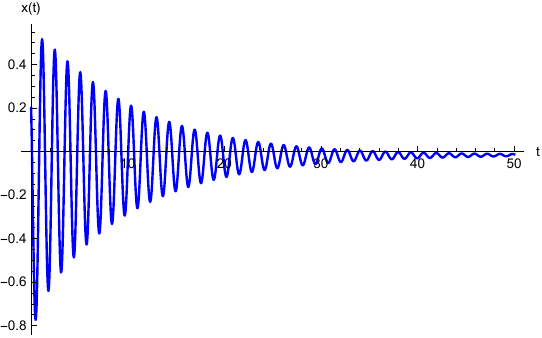}
	}\hspace{0.1cm}
	\subfloat[Unstable solution for $a=2.35$, $b=-1$, $\alpha=0.45$, $c=-0.3$ and $\tau=0.5$]{%
		\includegraphics[scale=0.46]{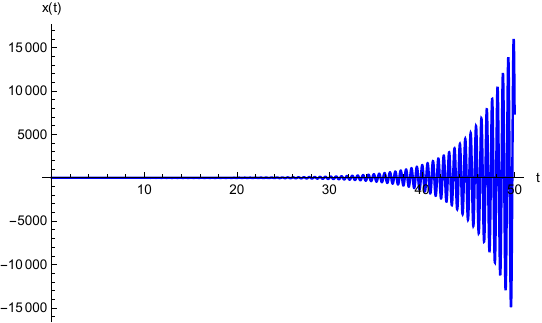}
	}\hspace{0.1cm}
	\subfloat[Stable solution for $a=2.48$, $b=-1$, $\alpha=0.45$, $c=-0.3$ and $\tau=0.3$]{%
		\includegraphics[scale=0.46]{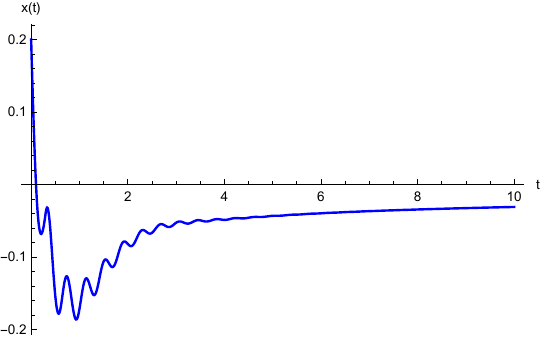}
	 }\hspace{0.1cm}
	\subfloat[Unstable solution for $a=2.48$, $b=-1$, $\alpha=0.45$, $c=-0.3$ and $\tau=0.5$]{%
	\includegraphics[scale=0.46]{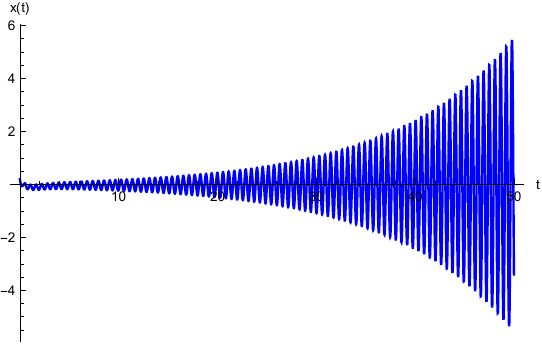}
	}\hspace{0.1cm}
	\subfloat[Stable solution for $a=2.6$, $b=-1$, $\alpha=0.45$, $c=-0.3$ and $\tau=0.4$]{%
	\includegraphics[scale=0.46]{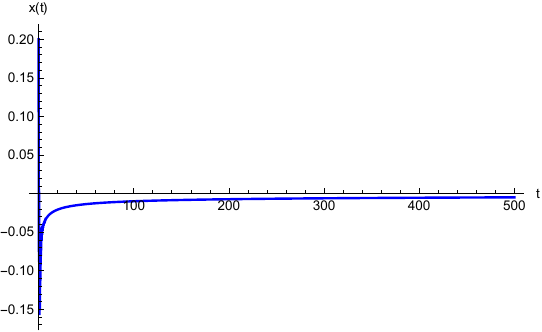}
	}\hspace{0.1cm}
	\subfloat[Unstable solution for $a=2.6$, $b=-1$, $\alpha=0.45$, $c=-0.3$ and $\tau=0.6$]{%
	\includegraphics[scale=0.46]{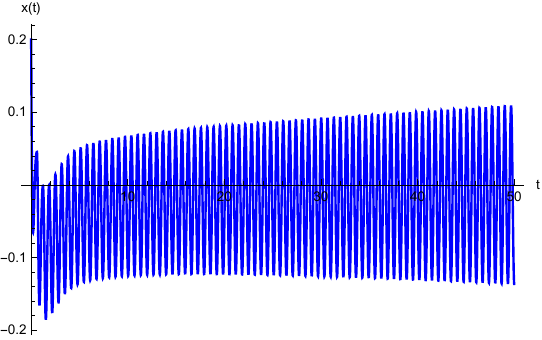}
	}\hspace{0.1cm}
	\subfloat[Stable solution for $a=2.6$, $b=-1$, $\alpha=0.45$, $c=-0.3$ and $\tau=1$]{%
	\includegraphics[scale=0.46]{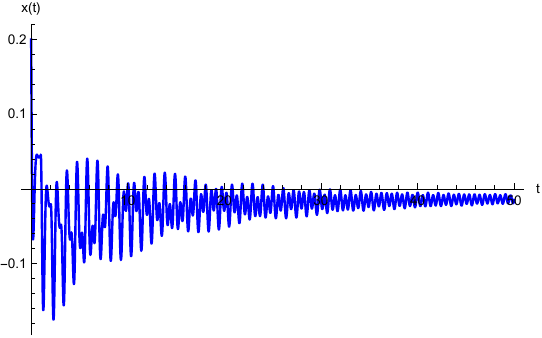}
	}	
\caption{}
	\label{figex4.7}
\end{figure}

\begin{Ex}
Consider $b>0$ and $1/2<\alpha<1$ in Figure \eqref{mainfig4.3}. Note that when $c>0$ the region bounded by $\Gamma_{7}$ and the vertical axis with $c>0$ is the only region that depends on delay $\tau$. Let us fix $b=1$ and $\alpha=0.8$ then $c_1=2.75575$ and $c_2=2.52097$. If $c=4>c_1$ then at the bifurcation curve $\Gamma_7$ we have $a_1=-0.0726$ or $a=-1.0726$ and at $\Gamma_{16}$, $a_1=-0.0418347$ or $a=-1.0418347$. Now, if we take $a=-1.06$ we get two positive values of $v^\alpha$ as $0.28028999$ and $0.385268362$. Note that $Re\Big(\dfrac{d\lambda}{d\tau}|_{u=0}\Big)<0$ at $0.28028999$ and $Re\Big(\dfrac{d\lambda}{d\tau}|_{u=0}\Big)>0$ at $0.385268362$. Corresponding to  $0.28028999$ the critical values of $\tau$ are $28.5127, 59.3212, 90.1297, 120.938,\ldots$ and corresponding to $0.385268362$ critical values of $\tau$ are $18.0739, 38.7741, 59.4744, 80.1746,\ldots$. So, we get stability switches here i.e. for $\tau\in[0,18.0739)$ we get stable solution (cf. Figure \eqref{figex4.14}(a) for $\tau=17$), $\tau\in(18.0739,28.5127)$ gives unstable solution (cf. Figure \eqref{figex4.14}(b) for $\tau=20$), $\tau\in(28.5127,38.7741)$ gives stable solution (cf. Figure \eqref{figex4.14}(c)) for $\tau=34$) and so on. Now, if we take $a=-1.02$, we get two positive value of $v^\alpha$ as $0.18462087812$ and $0.420968$ and $Re\Big(\dfrac{d\lambda}{d\tau}|_{u=0}\Big)<0$ at $0.18462087812$ and $Re\Big(\dfrac{d\lambda}{d\tau}|_{u=0}\Big)>0$ at $0.420968$. Corresponding to $0.18462087812$ we get critical values of $\tau$ are $49.7825, 101.702, 153.621, 205.54, 257.46,\ldots$ and corresponding to $0.420968$ we get critical $\tau$ are $15.7097, 34.2394, 52.7691, 71.2988,\ldots$. This gives SSR region. So, we get stable solution for $\tau\in[0,15.7097)$ and it will remain unstable $\forall\tau>15.7097$. 
\end{Ex}
\begin{figure}
	\subfloat[Stable solution for $a=-1.06$, $b=1$ and $\tau=17$]{%
		\includegraphics[scale=0.46]{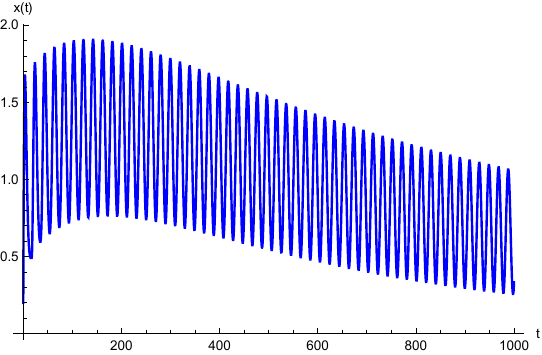}
	}\hspace{0.1cm}
	\subfloat[Unstable solution for $a=-1.06$, $b=1$ and $\tau=20$]{%
		\includegraphics[scale=0.46]{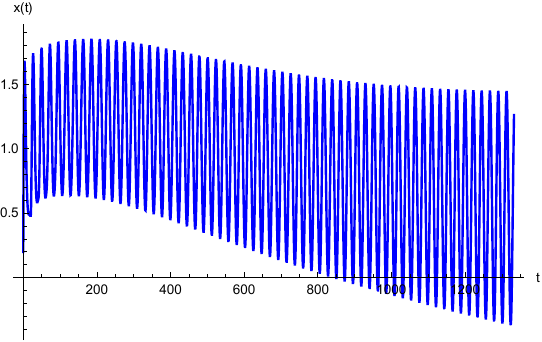}
	}\hspace{0.1cm}
	\subfloat[Stable solution for $a=-1.06$, $b=1$ and $\tau=34$]{%
		\includegraphics[scale=0.46]{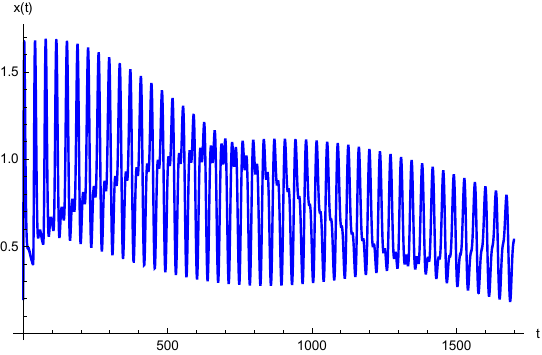}
	}	
\caption{}
	\label{figex4.14}
\end{figure}

\begin{Ex}
Consider $c<0$ in Figure \eqref{mainfig4.3}. Note that the value of $c_3$ is $-0.751566$ for $\alpha=0.8$ and $b=1$. Let us fix $c=-0.4>c_3$ then at $\Gamma_2$, $a_1=6.54508$ ($a=5.54508$) and at $\Gamma_8$, $a_1=10.5099$ or $a=9.5099$. If we choose $a=6$ then we are in the region bounded by the curves $\Gamma_2$ and $\Gamma_8$ from Figure \eqref{mainfig4.3}. With these parameters, we get two positive values of $v^\alpha$ viz. $3.52068$ and $4.19432$. Corresponding to $v^\alpha=3.52068$, we have $Re\Big(\dfrac{d\lambda}{d\tau}|_{u=0}\Big)<0$ and corresponding to $v^\alpha=4.19432$ we have $Re\Big(\dfrac{d\lambda}{d\tau}|_{u=0}\Big)>0$. The critical values of $\tau$ for $3.52068$ are $0.744534, 2.04739, 3.35024,\ldots$ and for $4.19432$ these are $0.0246074, 1.07138, 2.11816,\ldots$. We get stability switch $[0,0.0246074)(S)$; $(0.0246074,0.744534)(U)$; $(0.744534,1.07138)(S)$, $\ldots$. The stable solution for $\tau=0.01$ is given in Figure \eqref{figex4.8}(a). The unstable solution for $\tau=0.03$ is given in Figure \eqref{figex4.8}(b). Figure \eqref{figex4.8}(c) shows stable solution for $\tau=0.9$. So, we are in the SS region as shown in Figure \eqref{mainfig4.3}.  
\end{Ex}
\begin{figure}
	\subfloat[Stable solution for $\tau=0.01$]{%
		\includegraphics[scale=0.46]{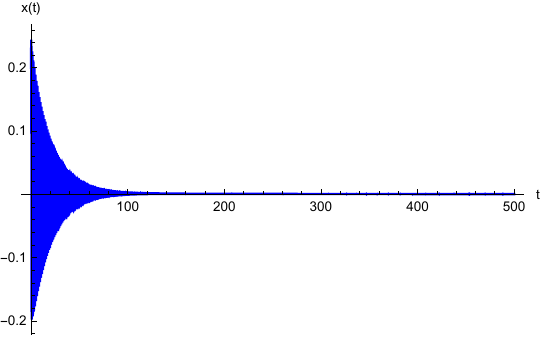}
	}\hspace{0.1cm}
	\subfloat[Unstable solution for $\tau=0.03$]{%
		\includegraphics[scale=0.46]{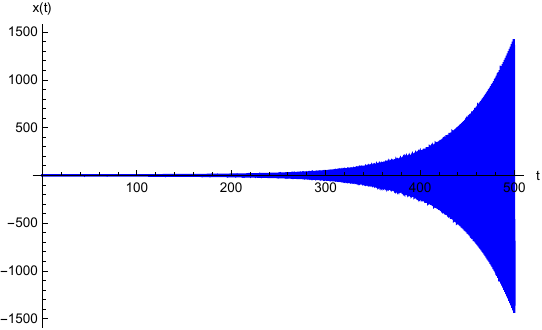}
	}\hspace{0.1cm}
	\subfloat[Stable solution for $\tau=0.9$]{%
		\includegraphics[scale=0.46]{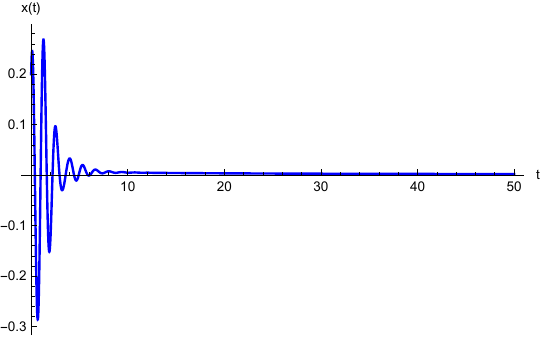}
	}	
\caption{}
	\label{figex4.8}
\end{figure}

\begin{Ex}
 Let us take $c>0$ in Figure \eqref{mainfig4.4}. We have two bifurcation values of $c$ viz. $c_7$ and $c_5$. Let us deal with each cases. When $b=-1$ and $\alpha=0.8$ we have $c_7=2.52097$ and $c_5=2.00571$.
\begin{itemize}
 \item[\ding{118}] When $c=9>c_7$ then at the bifurcation curve $\Gamma_9$ we have $a_1=-2.2168$ or $a=-1.2168$ and at $\Gamma_{16}$, $a_1=-2.08023$ or $a=-1.08023$.
\begin{itemize}
\item[$\bullet$] When $a<-1.2168$ then no positive value of $v$ exists and since system is stable at $\tau=0$ it will remain stable $\tau\geq0$.
\item[$\bullet$] When $-1.2168<a<-1.08023$ then we have SS from Figure \eqref{mainfig4.4}. So, let us take $a=-1.2166$ and the two positive values of $v^\alpha$ are $v_1=0.28027$ and $v_2=0.286791$. Note that $Re\big(\dfrac{d\lambda}{d\tau}|_{u=0}\big)>0$ at $v_2$ and $Re\big(\dfrac{d\lambda}{d\tau}|_{u=0}\big)<0$ at $v_1$. Critical values of delay corresponding to these $v^\alpha$ are given in the table \eqref{table4.5}. So, we get stable solution for $\tau\in[0,11.2217)$ (cf. Figure \eqref{figex4.10}(a) for $\tau=10$), unstable solution for $\tau\in(11.2217,11.7248)$ (cf. Figure \eqref{figex4.10}(b) for $\tau=11.4$) and stable solution for $\tau\in(11.7248,41.1597)$ (cf. Figure \eqref{figex4.10}(c) for $\tau=20$) so on. 
\item[$\bullet$] Now, if we take $a>-1.08023$ then we are in SSR region. So, let us take $a=-1.07$ we get two values of $v^\alpha$ as $v_1=0.13912$ and $v_2=0.342681$. $Re\big(\dfrac{d\lambda}{d\tau}|_{u=0}\big)<0$ at    $v_1$ and $Re\big(\dfrac{d\lambda}{d\tau}|_{u=0}\big)>0$ at $v_2$. The critical values corresponding to $v_1$ and $v_2$ are given in Table \eqref{table4.6}. So, we get stable region for $\tau\in[0,7.23693)$ and unstable solution for $\tau>7.23693$. Figure \eqref{figex4.10}(d) shows stable solution for $\tau=7$ and Figure \eqref{figex4.10}(e) shows unstable solution for $\tau=9$.
 \end{itemize}
 \item[\ding{118}] Now, let us take $c=2.2$ which is in between $c_7$ and $c_5$. We have $a_1=-2$ at $a_1=2b$ or $a=-1$, at $\Gamma_9$, $a_1=-1.9788$ or $a=-0.9788$ and at $\Gamma_{16}$, $a_1=-1.98325$ or $a=-0.98325$.
 \begin{itemize}
 \item[$\bullet$] Let $a=-0.99$ we are in the SSR region from Figure \eqref{mainfig4.4} and there exist only one positive value of $v^\alpha=0.0384923$ and $Re\big(\dfrac{d\lambda}{d\tau}|_{u=0}\big)>0$ at $v^\alpha$. The critical values of $\tau$ corresponding to  $v^\alpha$ are $182.001$, $550.522$, $919.044$, $1287.56$ etc. So, we have the stable solution of system \eqref{eq3.2} for $\tau\in[0,182.001)$. The stable solution for $\tau=9$ is shown in Figure \eqref{figex4.10}(e) and one of the complex root with positive real part for $\tau=184$ is given by $3.79477*10^{-7} + 0.0168663I$. We know that one positive root is sufficient for the instability.
 \item[$\bullet$] Now, let us take $a=-0.98$ then from Figure \eqref{mainfig4.4} we are in SS region. We get three positive value of $v^\alpha$ namely $v_1=0.114518$,
$v_2=0.208169$ and $v_3=0.357206$. Note that $Re\big(\dfrac{d\lambda}{d\tau}|_{u=0}\big)>0$ at $v_1$, $Re\big(\dfrac{d\lambda}{d\tau}|_{u=0}\big)<0$ at $v_2$ and $Re\big(\dfrac{d\lambda}{d\tau}|_{u=0}\big)>0$ at $v_3$. So, the critical values of $\tau$ corresponding to each $v_1$, $v_2$ and $v_3$ are given in Table \eqref{table4.7}. We get stability switch $[0,9.46047)(S)$; $(9.46047,20.5158)(U)$; $(20.5158,32.2131)(S)$, $\ldots$.                                                                                                                                                                                                                                                                                                                                                                                                                                                                                                                                                                                                                                                                                                                                                                                                                                                                                                                                                                                                                                                                                                                                                                                                                                                                                                                                                                                                                                                                                                                                                                                                                                                                                                                                                                                                                            
\end{itemize}
\item[\ding{118}] If $c=1<c_5$ then we have only one bifurcation value of $a_1=-2$, $(a=-1)$ from figure \eqref{mainfig4.4}. So,
\begin{itemize}  
\item[$\bullet$] if we take $a=-1.1$ then no positive $v^\alpha$ exist and system become stable $\forall\tau\geq0$ 
\item[$\bullet$] if we take $a=-0.9$ then only one positive $v^\alpha=0.364932$ exist and critical values of $\tau$ are $9.52706, 31.6791, 53.8312,\ldots$. We get stable solution for $\tau\in[0,9.52706)$ (cf. Figure \eqref{figex4.11}(a) for $\tau=9$) and unstable for $\tau>9.52706$(cf. Figure \eqref{figex4.11}(b) for $\tau=10$).
\end{itemize}
\end{itemize}
\end{Ex}
\begin{figure}
\subfloat[Stable solution for $\tau=10$]{%
		\includegraphics[scale=0.46]{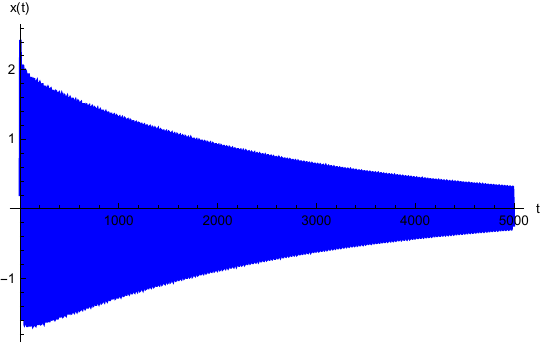}
	}\hspace{0.1cm}
	\subfloat[Unstable solution for $\tau=11.4$]{%
		\includegraphics[scale=0.46]{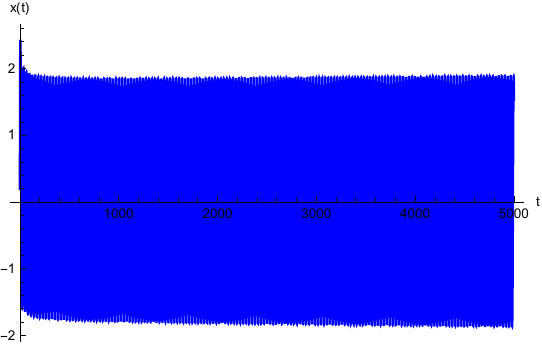}
	}\hspace{0.1cm}
	\subfloat[Stable solution for $\tau=20$]{%
		\includegraphics[scale=0.46]{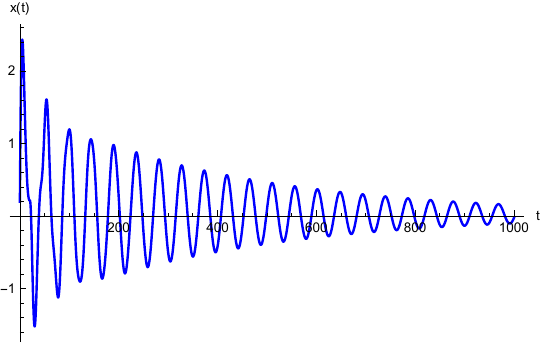}
	}\hspace{0.1cm}
	\subfloat[Stable solution for $a=-1.07$ and $\tau=10$]{%
		\includegraphics[scale=0.46]{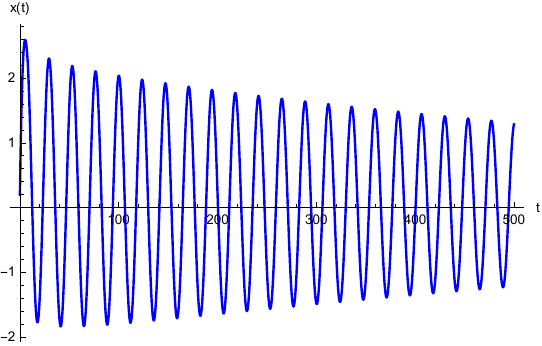}
	}\hspace{0.1cm}
	\subfloat[Unstable solution for $a=-1.07$ and $\tau=9$]{%
		\includegraphics[scale=0.46]{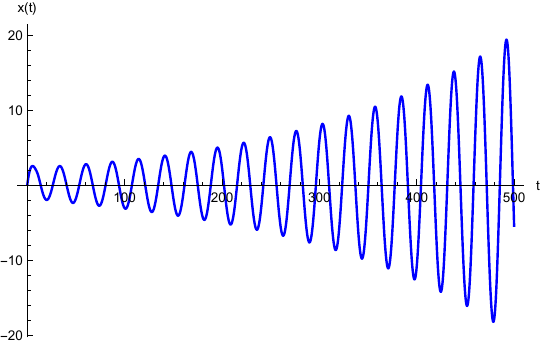}
	}\hspace{0.1 cm}
	\subfloat[Stable solution for $a=-0.99$ and $\tau=9$]{%
		\includegraphics[scale=0.46]{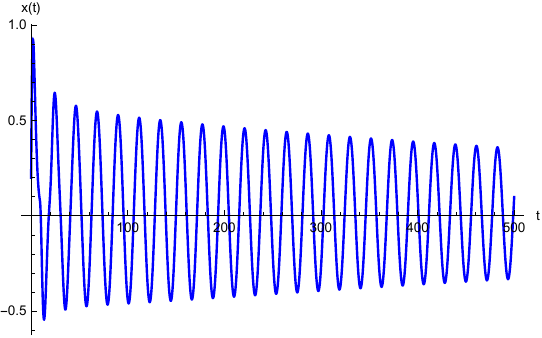}
	}\hspace{0.1cm}	
\caption{}
	\label{figex4.10}
\end{figure}

\begin{table}
\begin{center}
\begin{tabular}{ |p{3cm}||p{1.5cm}| p{1.5cm}|p{1.5cm}|p{1.5cm}| }
 \hline
 $v_1=0.28027$& $11.7248$ & $42.536$ & $73.3473$ & $104.159$ \\  
 \hline
$v_2=0.286791$ & $11.2217$ & $41.1597$ & $71.0978$ & $101.036$\\
\hline
\end{tabular}
 \caption{value of $v^\alpha$ and critical values of delay $\tau$ }\label{table4.5}
\end{center}
\end{table}

\begin{table}
\begin{center}
\begin{tabular}{ |p{3cm}||p{1.5cm}| p{1.5cm}|p{1.5cm}|p{1.5cm}| }
 \hline
 $v^\alpha=0.13912$ & $34.1872$& $108.138$& $182.089$ &$256.04$ \\  
 \hline
$v^\alpha= 0.342681$ & $7.23693$& $31.2013$& $55.1657$& $79.1301$\\
\hline
\end{tabular}
 \caption{Value of $v^\alpha$ and critical values of delay $\tau$ }\label{table4.6}
\end{center}
\end{table}

\begin{table}
\begin{center}
\begin{tabular}{ |p{3cm}||p{1.5cm}| p{1.5cm}|p{1.5cm}|p{1.5cm}| }
 \hline
 $v_1=0.114518$& $45.264$ & $139.581$ &$233.898$& $328.215$  \\  
 \hline
$v_2=0.208169$ & $20.5158$ & $ 65.2005$ &$109.885$&$ 154.57$\\
\hline
$v_3=0.357206$ & $9.46047$ & $32.2131$& $54.9656$& $77.7182$\\
\hline
\end{tabular}
 \caption{Value of $v^\alpha$ and critical values of delay $\tau$ }\label{table4.7}
\end{center}
\end{table}
\begin{figure}
\subfloat[Stable solution for $\tau=9$]{%
		\includegraphics[scale=0.66]{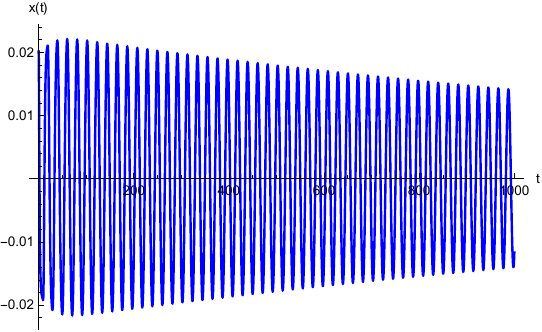}
	}\hspace{0.1cm}
	\subfloat[Unstable solution for $\tau=10$]{%
		\includegraphics[scale=0.66]{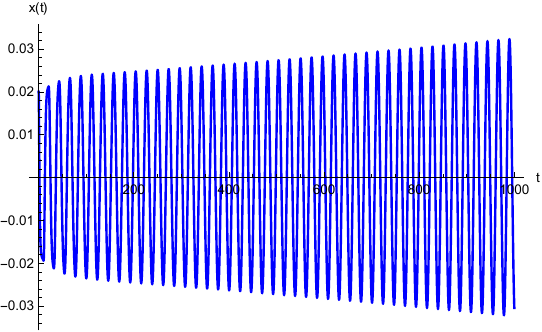}
	}\hspace{0.1cm}	
\caption{}
	\label{figex4.11}
\end{figure}

\begin{Ex}
Let us take $c<0$. The region bounded between $\Gamma_{18}$ and $\Gamma_{10}$ depends on delay $\tau$ (cf. Figure \eqref{mainfig4.4}). If we fix $b=-1$ and $\alpha=0.8$ then the intersection of $\Gamma_{17}$ and $\Gamma_2$ is $c_6=-1.5598$.
\begin{itemize}
\item[\ding{118}] Let us first take $c=-0.9>c_6$ and then at $\Gamma_{18}$, $a_1=0.0220747$ or $a=1.0220747$ at $\Gamma_2$, $a_1=2.90893$ or $a=3.90893$ and at $\Gamma_{10}$, $a_1=4.695$ or $a=5.695$.
\begin{itemize} 
\item[$\bullet$] If we fix $a=1.03$ we get two positive value of $v^\alpha$ as $0.0889235$ and $1.45654$. The critical values of delay corresponding to $0.088923516$ are $1.65726, 131.05, 260.442,\ldots$ and corresponding to $1.4565436$ we get $1.79705, 5.72373, 9.65041,\ldots$. So, we have IS from Figure \eqref{mainfig4.4} i.e $\tau\in[0,1.65726)(U)$; $(1.65726,1.79705)(S)$ and again unstable $\forall\tau>1.79705$. The unstable solution for $\tau=1.5$ is given in Figure \eqref{figex4.12}(a). The stable solution for $\tau=1.7$ is given in Figure \eqref{figex4.12}(b). Unstable solution for $\tau=1.9$ shown in Figure \eqref{figex4.12}(c).
\item[$\bullet$] If we take $a=4.1$ we are in the SS region from Figure \eqref{mainfig4.4}. We get two positive values of $v^\alpha$ as $1.86275$ and $2.36524$. The critical values of delay corresponding to $1.86275$ and $2.36524$ are $2.85785, 5.74513,8.6324,\ldots$ and $1.34019, 3.48227, 5.62435,\ldots$ respectively. So, we get stable solution for $\tau\in[0,1.34019)$ (Figure \eqref{figex4.12}(d) for $\tau=1.2$), unstable solution for $\tau\in(1.34019,2.85785)$ (Figure \eqref{figex4.12}(e) for $\tau=2.4$), again stable solution for $\tau\in(2.85785,3.48227)$ (Figure \eqref{figex4.12}(f) for $\tau=3.3$) and so on.
\item Now for $a>5.695$ we get stable solution for all $\tau\geq0$ because there does not exist any positive value of $v^\alpha$ and system is stable at $\tau=0$ so it will remain stable $\forall\tau\geq0$.
\end{itemize}
\item[\ding{118}] Now let us take $c=-3<c_6$. At the bifurcation curve $\Gamma_{18}$, $a_1=0.00178628$ or $a=1.00178628$, at $\Gamma_2$, $a_1=0.872678$ or $a=1.872678$, at $\Gamma_{17}$, $a_1=1.24228$ or $a=0.24228$ and at $\Gamma_{10}$ we have $a_1=2.34685$ or $a=3.34685$. 
\begin{itemize}
\item[$\bullet$] If $a=1.8$ then we are in IS region from Figure \eqref{mainfig4.4}. The two positive values of $v^\alpha$ viz. $v_1=0.514075$ and $v_2=0.955756$. We have $Re\big(\dfrac{d\lambda}{d\tau}|_{u=0}\big)>0$ at $v_2$ and $Re\big(\dfrac{d\lambda}{d\tau}|_{u=0}\big)<0$ at $v_1$. The critical values of delay are given in Table \eqref{table4.8}. So, we have $[0,0.0526283) (U)$ (cf. Figure \eqref{figex4.13}(a) for $\tau=0.04$); $(0.0526283,4.14761) (S)$ (cf. Figure \eqref{figex4.13}(b) for $\tau=0.07$) and $\tau>4.14761$ we have unstable (cf. Figure \eqref{figex4.13}(c) for $\tau=5$). 
\item[$\bullet$] If we take $a=2$ then from figure \eqref{mainfig4.4} we are in the SSR region. We get two positive values of $v^\alpha$ as $0.581679$ and $0.981791$. The critical values of $\tau$ are given in Table \eqref{table4.10}. So, we get SSR from $\tau\in[0,4.10744)$. The stable solution for $\tau=3$ is shown in Figure \eqref{figex4.15}(a) and unstable solution for $\tau=5$ is shown in Figure \eqref{figex4.15}(b).
\item[$\bullet$] If we take $a=2.8$ then we are in SS region as shown in Figure \eqref{mainfig4.4}. We get two positive values of $v^\alpha$ as $0.820205$ and $1.05854$. The corresponding critical values of $\tau$ are given in Table \eqref{table4.11}. The stable solution for $\tau=3$, unstable solution for $\tau=5$ and stable solution for $\tau=9$ are shown in Figures \eqref{figex4.13}(d), (e) and (f) respectively.
\end{itemize}
\end{itemize}
\end{Ex}
\begin{figure}
	\subfloat[Unstable solution for $\tau=1.5$]{%
		\includegraphics[scale=0.46]{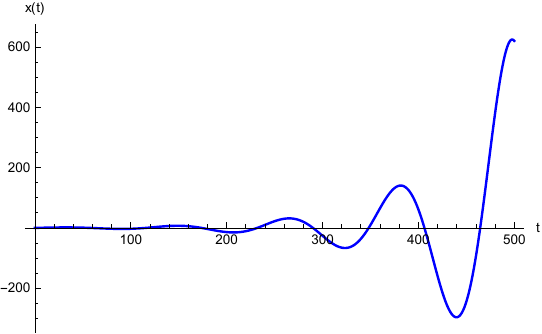}
	}\hspace{0.1cm}
	\subfloat[Stable solution for $\tau=1.7$]{%
		\includegraphics[scale=0.46]{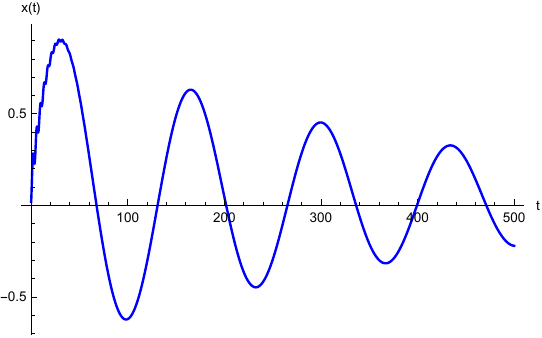}
	}\hspace{0.1cm}
	\subfloat[Unstable solution for $\tau=1.9$]{%
		\includegraphics[scale=0.46]{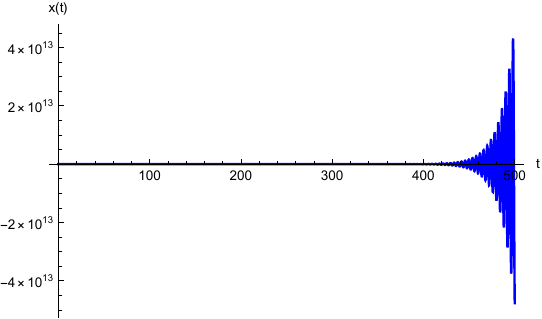}
	}\hspace{0.1cm}
	\subfloat[Stable solution for $\tau=1.2$]{%
		\includegraphics[scale=0.46]{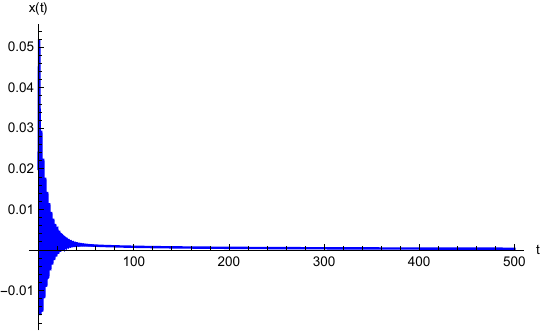}
	 }\hspace{0.1cm}
	\subfloat[Unstable solution for $\tau=2.4$]{%
	\includegraphics[scale=0.46]{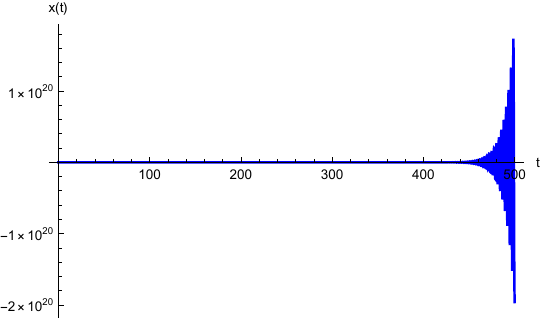}
	}\hspace{0.1cm}
	\subfloat[Stable solution for $\tau=3.3$]{%
	\includegraphics[scale=0.46]{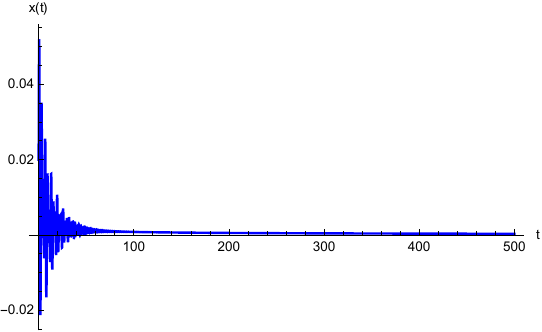}
	}	
\caption{}
	\label{figex4.12}
\end{figure}
\begin{table}
\begin{center}
\begin{tabular}{ |p{3cm}||p{1.5cm}| p{1.5cm}|p{1.5cm}|p{1.5cm}| }
 \hline
 $v_1=0.514075$& $0.0526283$& $14.487$& $28.9213$ & $43.3556$  \\  
 \hline
$v_2=0.955756$ & $4.14761$&$ 10.7965$& $17.4453$& $24.0941$\\
\hline
\end{tabular}
 \caption{Value of $v^\alpha$ and critical values of delay $\tau$ }\label{table4.8}
\end{center}
\end{table}
\begin{table}
\begin{center}
\begin{tabular}{ |p{3cm}||p{1.5cm}| p{1.5cm}|p{1.5cm}|p{1.5cm}| }
 \hline
 $v_1=0.581679$& $12.2832$& $24.652$&$ 37.0207$&$ 49.3894$  \\  
 \hline
$v_2=0.981791$ & $4.10744$&$ 10.5366$& $16.9658$&$ 23.395$\\
\hline
\end{tabular}
 \caption{Value of $v^\alpha$ and critical values of delay $\tau$ }\label{table4.10}
\end{center}
\end{table}
\begin{table}
\begin{center}
\begin{tabular}{ |p{3cm}||p{1.5cm}| p{1.5cm}|p{1.5cm}|p{1.5cm}| }
 \hline
 $v_1=0.820204$& $7.51375$&$ 15.5634$&$ 23.613$&$ 31.6627$  \\  
 \hline
$v_2=1.058537$ & $4.15684$&$ 10.0087$&$ 15.8607$&$ 21.7126$\\
\hline
\end{tabular}
 \caption{Value of $v^\alpha$ and critical values of delay $\tau$ }\label{table4.11}
\end{center}
\end{table}
\begin{figure}
	\subfloat[Unstable solution for $\tau=0.04$]{%
		\includegraphics[scale=0.46]{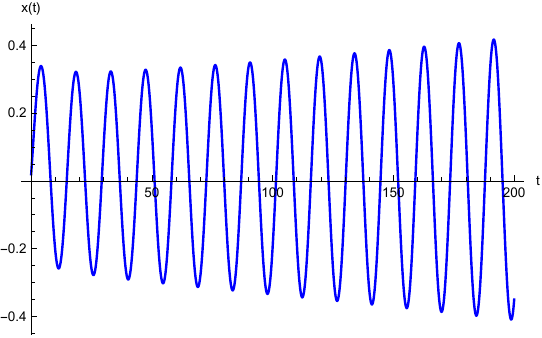}
	}\hspace{0.1cm}
	\subfloat[Stable solution for $\tau=0.07$]{%
		\includegraphics[scale=0.46]{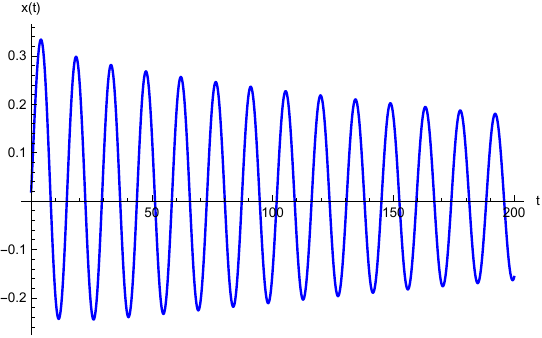}
	}\hspace{0.1cm}
	\subfloat[Unstable solution for $\tau=5$]{%
		\includegraphics[scale=0.46]{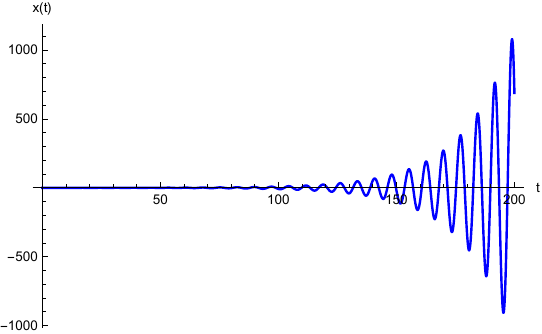}
	}\hspace{0.1cm}
	\subfloat[Stable solution for $a=2.8$ and $\tau=3$]{%
		\includegraphics[scale=0.46]{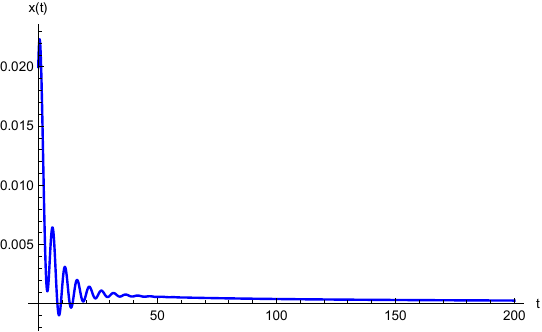}
	}\hspace{0.1cm}
	\subfloat[Stable solution for $a=2.8$ and $\tau=5$]{%
		\includegraphics[scale=0.46]{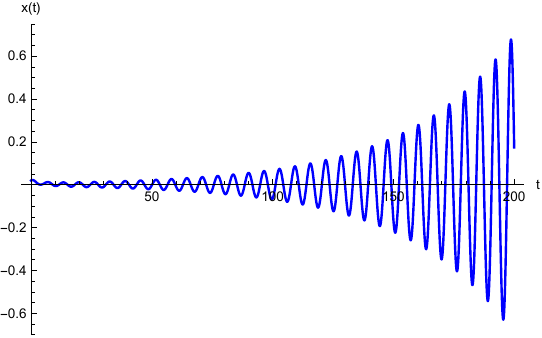}
	}\hspace{0.1cm}
	\subfloat[Stable solution for $a=2.8$ and $\tau=9$]{%
		\includegraphics[scale=0.46]{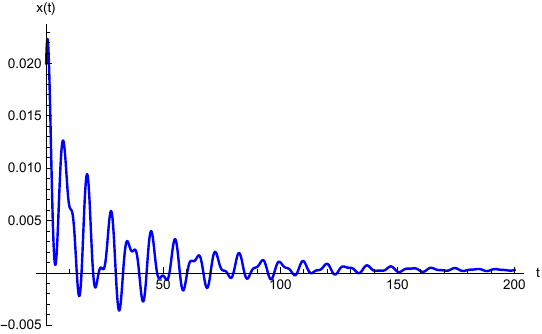}
	}		
\caption{}
	\label{figex4.13}
\end{figure}
\begin{figure}
	\subfloat[Stable solution for $a=2$ and $\tau=3$]{%
	\includegraphics[scale=0.66]{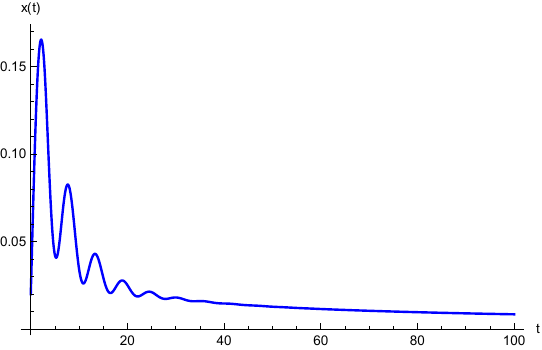}
	}\hspace{0.1cm}
	\subfloat[Stable solution for $a=2$ and $\tau=5$]{%
	\includegraphics[scale=0.66]{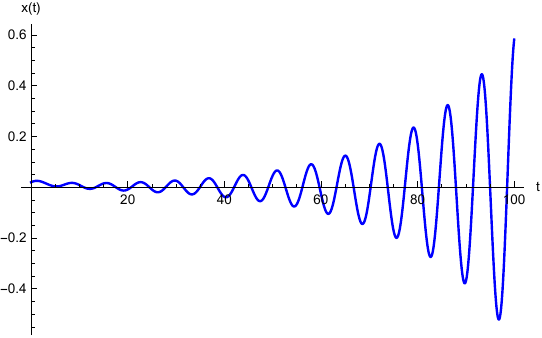}
	}\caption{}
	\label{figex4.15}
\end{figure}
	\section{Conclusions}\label{sec3.10}
The transcendental nature of the characteristic equation $\lambda^{2\alpha}x(t)+c \lambda^{\alpha}x(t)-a -b e^{-\lambda\tau}=0$ makes the two-term fractional-order delay differential equation a more complex than the usual ODE or a one-term FDDE. We provided various conditions under which the stability of the proposed equation does not depend on the delay. Further, we used the hypothesis that the ``change in stability can occur only when the characteristic root crosses the imaginary axis in the complex plane". This leads us to provide the boundary of the stable region in the parameter plane. Furthermore, we worked on the stability switch and an instability switch too. This made the work complete in all aspects. We hope this work will be very useful to the scientists dealing with the systems involving memory and hereditary properties. The conditions provided by us are very simple and depend only on the parameter values.
\section*{Acknowledgments}
S. Bhalekar acknowledges the University of Hyderabad for Institute of Eminence-Professional Development Fund (IoE-PDF) by MHRD (F11/9/2019-U3(A)).
D. Gupta thanks University Grants Commission for financial support (No.F.82-44/2020(SA-III)).
	
	\bibliography{newref}
\bibliographystyle{ieeetr}	

\end{document}